\documentclass[11pt,reqno]{amsart}
\usepackage{amsmath}
\usepackage{amsthm}
\usepackage{graphicx}
\usepackage{upgreek}
\usepackage{hyperref}
\usepackage{mathrsfs}
\usepackage{xcolor}
\usepackage{bm}
\usepackage{amsfonts}
\usepackage{amssymb}
\usepackage{csquotes}
\usepackage{stmaryrd}
\usepackage[normalem]{ulem}
\usepackage{enumerate}

\usepackage{soul}
\usepackage[margin=1.1in]{geometry}
\numberwithin{equation}{section}

\newtheorem{theorem}{Theorem}[section]
\newtheorem{lemma}[theorem]{Lemma}
\newtheorem{metatheorem}[theorem]{Meta-Theorem}
\newtheorem{proposition}[theorem]{Proposition}

\newtheorem{definition}[theorem]{Definition}
\newtheorem{remark}[theorem]{Remark}

\newcommand{\sredm}[1]{\ifmmode\text{\xout{\ensuremath{\displaystyle \textcolor{red}{#1}}}}\else\sout{\textcolor{red}{#1}}\fi} 

\def\ud{\mathrm{d}}
\def\E{{\mathbb E}}
\def\ES{\llbracket d \rrbracket}
\def\ESm{\llbracket m \rrbracket}
\def\ESdm{\llbracket d-m \rrbracket}

\def\R{{\mathbb R}}
\def\RR{{\mathbb R}}

\def \fd {\mathfrak{d}}
\def \be {\begin{equation}}
\def \ee {\end{equation}}

\begin{document}
\title[Finite state MFG{s} with Wright--Fisher common noise]{Finite state Mean Field Games with Wright--Fisher common noise}

 \author[E. Bayraktar]{Erhan Bayraktar \address{(E. Bayraktar) Department of Mathematics, University of Michigan, Ann Arbor, Michigan 48109, United States}\email{erhan@umich.edu}} 
 
\author[A. Cecchin]{Alekos Cecchin \address{(A. Cecchin) Universit\'e C\^ote d'Azur, CNRS, Laboratoire J.A. Dieudonn\'e, 06108 Nice, France
}\email{alekos.Cecchin@unice.fr
}}

 \author[A. Cohen]{Asaf Cohen \address{(A. Cohen) Department of Mathematics, University of Michigan, Ann Arbor, Michigan 48109, United States}\email{asafc@umich.edu}}

\author[F. Delarue]{ Fran\c{c}ois Delarue 
\address{(F. Delarue) Universit\'e C\^ote d'Azur, CNRS, Laboratoire J.A. Dieudonn\'e, 06108 Nice, France
}\email{francois.delarue@unice.fr
}}

\thanks{--E. Bayraktar is partially supported by the National Science Foundation and by the Susan M. Smith chair. 
 A. Cecchin and F. Delarue acknowledge the financial support of French ANR project ANR-16-CE40-0015-01 on ``Mean Field Games''. F. Delarue also thanks Institut Universitaire de France and French ANR projet ANR-19-P3IA-0002 ``3IA C\^ote d'Azur - Nice - Interdisciplinary Institute for Artificial Intelligence''. A. Cohen acknowledges the financial support of Research supported by the National Science Foundation (DMS-2006305).}

\thanks{--This is the final version of the paper. To appear in {\it Journal de Math\'ematiques Pures et Appliqu\'ees}.}

\date{\today}


\begin{abstract}
We force uniqueness in finite state mean field games by adding a Wright--Fisher common noise. We achieve this by analyzing the master equation of this game, which is a degenerate parabolic second-order partial differential equation set on the simplex whose characteristics solve the stochastic forward-backward system associated with the mean field game; see Cardaliaguet et al. \cite{CardaliaguetDelarueLasryLions}. We show that this equation, which is a non-linear version of the Kimura type equation studied in Epstein and Mazzeo \cite{EpsteinMazzeo}, has a unique smooth solution whenever the normal component of the drift at the boundary is strong enough. Among others, this requires a priori estimates of H\"older type for the corresponding Kimura operator when the drift therein is merely continuous.
\end{abstract}

\maketitle

\tableofcontents
\section{Introduction} 
{Aiming at} forcing uniqueness in the theory of mean field games (MFGs), a more complete account of which we provide below, we analyze here  
a system of parabolic partial differential equations (PDEs) with the main feature of being set on the space of probability measures on $\ES := \{1,\cdots,d\}$, the latter being referred to as the $(d-1)$-dimensional simplex, for a fixed integer $d \geq 1$.
This system is indexed by the elements $i$ of $\ES$ itself and 
has the following generic form:
\begin{equation}
\label{eq:master:equation:introl}
\begin{split}
&\partial_t U^i -\frac{1}{2}\sum_{j=1}^d(U^i-U^j)^2_++f^i(t,p) + 
\sum_{j =1}^d \varphi(p_{j})
\bigl[ U^j - U^i \bigr]
\\
&\qquad+
\sum_{1\le j,k \le d} p_k\bigl[ \varphi(p_{j}) + (U^k-U^j)_+ \bigr] \left( \partial_{p_{j}} U^i - \partial_{p_{k}} U^i\right) 
\\
&\qquad+ \varepsilon^2\sum_{j =1}^d p_{j} \left( \partial_{p_{i}} U^i - \partial_{p_j} U^i\right)
+\frac{\varepsilon^2}{2} \sum_{1\le j,k \le d}(p_j \delta_{jk}-p_{j} p_{k}) \partial^2_{p_{j} p_{k}} U^i =0,\\
&U^i(T,p)= g^i(p),
\end{split}
\end{equation}
for $i\in \ES$,
where $(t,p) \in [0,T] \times {\mathcal S}_{d-1}$ and $\mathcal S_{d-1}$ is the $(d-1)$-dimensional simplex. 
Whenever $\varepsilon$ is equal to zero and $\varphi$ is also identically equal to zero, this system is the so-called
\textit{master equation}
that describes the values of the equilibria in a (finite state) MFG
driven by a simple continuous time Markov decision process on $\ES$ and by the functions $(f^i)_{i \in \ES}$ and $(g^i)_{i \in \ES}$ as respective running and terminal costs.
To wit, the first line in \eqref{eq:master:equation:introl}, which has a form similar to a Hamilton-Jacobi equation on $\ES$, accounts for the optimization problem in the underyling MFG, whilst the second line accounts for the dynamics of the equilibria.
Although we provide a longer review on MFG later in the text, we feel useful to quote, at this early stage of the introduction, 
{\cite{gomes2014dual, gomes2014socio,Lionscollege2}}
and \cite[Chapter 7]{CarmonaDelarue_book_I} as references on the master equation for finite state mean field games, 
and
{\cite{Cardaliaguet,ben-fre-yam2015, CarmonaDelarue_book_II, CardaliaguetDelarueLasryLions, cha-cri-del2019,Lionscollege1,Lionsvideo}}
as references
for continuous state mean field games.
The main novelty here is the third line in 
\eqref{eq:master:equation:introl}. Therein, $\varepsilon$ is a (strictly) positive viscosity parameter 
which we call the {\it intensity of the common noise}. This terminology comes from the fact that equation \eqref{eq:master:equation:introl} is associated with a new form of MFG, which we are going to describe later, in which equilibria are no longer deterministic but are subjected to a so-called {\it common noise} and are hence randomized. 
Under the action of the common noise, the master equation becomes a system of second order PDEs, the principal part of 
which is the second-order operator in the third line of 
\eqref{eq:master:equation:introl} and is called a Kimura operator on the simplex (see \cite{EpsteinMazzeo,Kimura}). 
Accordingly, the master equation here reads as 
 a system of non-linear parabolic PDEs of Kimura type. 
 As for the additional function $\varphi$ in the first two lines of \eqref{eq:master:equation:introl}, it should be understood as a forcing term in the dynamics of the equilibria that allow the latter to escape for free from the boundary of the simplex. 
 In this context, one of our contributions  (see 
{Meta-Theorem  
 \ref{main:metathm-PDE}}
 and Theorem \ref{existence:master}) is to show that, when 
 $\varepsilon$ is strictly positive and the
 functions $(f^i)_{i \in \ES}$ and $(g^i)_{i \in \ES}$ satisfy some smoothness conditions, we can choose $\varphi$ large enough in the neighborhood of the boundary of $\mathcal S_{d-1}$ and null everywhere else in such a way that \eqref{eq:master:equation:introl} has a unique smooth
 solution 
 (in a so-called Wright--Fischer H\"older space of functions that are once differentiable in time and twice in space with a suitable behavior at the boundary of the simplex). Accordingly, our main result is that the corresponding MFG is uniquely solvable for a prescribed initial condition (see 
{Meta-Theorem \ref{main:metathm}
 and}
 Theorem \ref{main:thm}). Importantly, there are many examples for which the latter is false when $\varepsilon =0$, 
 which explains why we refer quite often to the concept of ``forcing uniqueness''.
%

A general framework to analyze linear {\it Kimura PDEs} was introduced by Epstein and Mazzeo in \cite{EpsteinMazzeo} and this framework was extended subsequently in \cite{eps-maz2016, eps-pop2017, pop2017b, POP}. Generally speaking, the analysis of Kimura PDEs suffers from two main difficulties: (1) the simplex boundary is not smooth, and (2) the PDE degenerates at the boundary.  
Despite these difficulties, the authors of \cite{EpsteinMazzeo} were able to prove the existence and uniqueness of smooth solutions to linear Kimura PDEs under enough regularity of the coefficients. 
However, these results do not apply to \eqref{eq:master:equation:introl} because 
the coefficients therein are time-dependent (Kimura operators are assumed to be time-homogeneous in 
\cite{EpsteinMazzeo})
 and, most of all, because the equation is non-linear.
While the additional time dependence can be handled with relative ease (see Lemma \ref{lem:4.6}), the non-linearity requires a sophisticated analysis, which, in fact, is the main technical part of this paper. In this respect, the main step in our study is Theorem \ref{main:holder}, which provides an a priori H\"older estimate to solutions of linear Kimura PDEs when driven by merely continuous drift terms that point inward the simplex in a sufficiently strong manner, whence our need for the additional $\varphi$ in equation \eqref{eq:master:equation:introl}. The proof of this a priori estimate uses a tailor-made {\it coupling by reflection} argument inspired by earlier works on couplings for multidimensional processes (see e.g., \cite{ChenLi}). However, the coupling by itself, as usually implemented in the literature for proving various types of smoothing effects for diffusion processes, is in fact not enough for our purpose. 
We indeed pay a price for the degeneracy of the equation at the boundary and, similar to other works on Kimura operators (see for instance \cite{albanesemangino}), we need to perform an induction over the dimension to handle the degeneracy properly; see Proposition \ref{prop:induction} for the details of the induction property.
Once we reach this point, the proof of existence of a solution to \eqref{eq:master:equation:introl} is straightforward, provided that 
$\varphi$ therein is chosen in a relevant way, and uses Schauder's fixed point theorem on the proper Wright--Fisher\footnote{Most of the time, we just say Wright--Fisher space instead of Wright--Fisher H\"older space.}  space, as well as Schauder's estimates derived in \cite{EpsteinMazzeo} for the linear equation and Lemma \ref{lem:4.6} mentioned earlier
(see Theorem \ref{existence:master}).

Let us now clarify our technical contribution into the context of MFGs. 
MFGs were introduced in the seminal works of Lasry and Lions \cite{Lasry2006,LasryLions2,LasryLions}, and Huang, Malham{\'e}, and Caines \cite{Huang2006,Huang2007}. Merging intuition from statistical physics and classical game theory, this paradigm provides the asymptotic behavior of many weakly interacting strategic players who are in a Nash equilibrium. Formally, this asymptotic equilibrium is described as
the fixed point of a best response map, which sends a given flow of measures to the distribution of a controlled state-dynamics.
For recent theoretical developments and applications of this theory, we refer the reader to { \cite{MR3268061, 
Bensoussan2013, Cardaliaguet,
CarmonaDelarue_book_I, CarmonaDelarue_book_II}} and the references therein.
MFGs with (a fixed and) finite number of states were {analyzed} by {\cite{GomesMohrSouza_discrete, Gomes2013, gue2015,Lionscollege2}}; for a probabilistic approach to finite state MFGs we refer to \cite{car-wan2018b, Cecchin2017}.

Typically, MFGs do not admit unique solutions. 
Two known instances of uniqueness are the small $T$ case 
and the so-called monotonous case due to Lasry and Lions, see \cite[Section 4]{Lasry2006} for the latter.
The thrust of our paper is to establish uniqueness by adding a common noise\footnote{\label{foo:BLL}Recently, Bertucci, Lasry, and Lions \cite{BertucciLasryLions} mentioned that ``The addition of a common noise in the MFG setting remains one of the most important questions in the MFG theory." We feel that our paper may be one step forward in this direction.} that emerges from the limiting behavior of Wright--Fisher population-genetics models. The special structure of the common noise we use leads to stochastic dynamics evolving inside the multidimensional simplex ${\mathcal S}_{d-1}$ and eventually to the second order form of \eqref{eq:master:equation:introl}. Besides the forcing uniqueness result, this is another interest of our work to incorporate population-genetics models into MFGs; to the best of our knowledge, this is a new feature in the field. 
{In this regard, it is worth pointing out that, even though we do not speak about it in this paper, it is in fact possible to explain 
 the common noise at the level of a particle system by a diffusion approximation. We refer to our companion work \cite{BCCD-game} for a complete overview.
}

In fact, we must stress that 
the recent work \cite{BertucciLasryLions} (to which we already alluded in the footnote \ref{foo:BLL})
also addresses a form of common noise for finite state MFGs.
As explained therein, the key point in this direction is to force the finite-player system to have many simultaneous jumps at some random times prescribed by the common noise. Although we  share a similar idea in our construction, our common noise structure is in the end different from \cite{BertucciLasryLions}:  While the simultaneous jumps 
in \cite{BertucciLasryLions} are governed by a deterministic transformation of the state space, 
they here obey a resampling procedure that is typical, as we have just said, of population-genetics models.    
Moreover, 
one of the questions in 
\cite{BertucciLasryLions}
is to decide whether the solution preserves monotonicity of the coefficients; in this regard, forcing uniqueness (outside the monotonous setting) is not discussed in 
\cite{BertucciLasryLions}. 
In fact, forcing uniqueness for MFGs was addressed in other works but in different 
settings. 
Recently, Delarue \cite{DelarueSPDE} established a forcing uniqueness result for a continuous state MFG  
obtained by forcing a deterministic (meaning that the players follow ordinary differential equations) MFG by means of a common noise.
 In this case, the common noise is infinite-dimensional and henceforth differs from the most frequent instance of common noise used in the literature, since the latter has very often a finite dimension, see e.g., \cite{CarmonaDelarue_book_II}. 
At this point, it is worth mentioning that forcing uniqueness is studied in 
\cite{fog2018} under the action of a standard finite-dimensional common noise, but 
for a linear quadratic MFG. This is due to the fact that the equilibrium distribution in that paper is Gaussian and is parametrized by its mean and variance, which reduces the dimension of the problem. 
On a more prospective level, forcing uniqueness by common noise might enable a selection criterion by taking small noise limit for cases where the limiting problem does not have a unique equilibrium.
This question was addressed in some specific cases in \cite{delfog2019} for a continuous state model and 
in \cite{cecdaifispel} for a finite state model (see also \cite{bayzhang2019})  and is the purpose of the forthcoming work \cite{cecdel2019} in a more general setting (with a finite state space). 

Once the master equation
\eqref{eq:master:equation:introl} 
 has been solved,
the equilibrium distribution of the mean field game, which becomes random under the action of the common noise, is {provided} by the solution of the forward component of the following {\it stochastic mean field game system}, given by the forward-backward stochastic differential equation 
\begin{equation}
\label{eq12}
\begin{split}
&d P_{t}^i =  \sum_{j =1}^d \Bigl( P_{t}^j \bigl( \varphi(P_{t}^i) + (u^j_t-u^i_t)_+ \bigr) - P_{t}^i \bigl( \varphi(P_{t}^j) 
+ (u^i_t-u^j_t)_+ \bigr)
\Bigr) \ud t
\\
&\hspace{30pt}
  + \frac{{\varepsilon}}{\sqrt 2} 
\sum_{j =1}^d
\sqrt{P_{t}^i P_{t}^j}  
d \bigl[ W_{t}^{i,j} - W_{t}^{j,i} \bigr],
\\
&\displaystyle du_{t}^i = - \Bigl( 
\sum_{j =1}^d \varphi(P^j_{t})
\bigl[ u_{t}^j - u_{t}^i \bigr]
-\frac{1}{2}\sum_{j=1}^d(u^i_{t}-u^j_{t})^2_+ + f^i(P_{t}) \Bigr) \ud t 
\\
&\hspace{30pt} - \frac{\varepsilon}{\sqrt{2}} \sum_{j =1}^d 
\sqrt{\frac{{P_{t}^j}}{{P_{t}^i}}} \bigl( \nu_{t}^{i,i,j} - \nu_{t}^{i,j,i} \bigr) \ud t
+ \sum_{1\le j \not = k\le d} \nu_{t}^{i,j,k} \ud W_{t}^{j,k},
\end{split}
\end{equation}
with a given initial deterministic probability vector {$(P_{0}^i = p_{0,i})_{i\in\ES}$} for the forward equation and the terminal condition  
{$(u^i_T=g^i(P_T))_{i\in\ES}$} for the backward equation. The process  
{$((W_{t}^{i,j})_{i,j \in \ES : i \not = j})_{0 \leq t \leq T}$}
is a  Wiener process and, as usual in the theory of backward stochastic differential equations, the role of the processes {$(((\nu^{i,j,k}_t))_{0\leq t\leq T})_{i,j,k\in \ES: j\neq k}$} is to force the solution {$((u^i_t)_{0 \leq t \leq T})_{i\in \ES }$} to be non-anticipating. The process {$((P^i_t)_{i\in \ES })_{0 \leq t \leq T}$} is  a Wright--Fisher diffusion process 
(taking values in the $(d-1)$-dimensional simplex), see \cite{Feller,Ethier,Sato}; accordingly, the forward equation in 
\eqref{eq12} must be interpreted as a stochastic Fokker-Planck equation on ${\mathcal S}_{d-1}$. The process {$((u^i_t)_{i\in \ES })_{0 \leq t \leq T}$} stands for the game-value for the representative player and the system of equations that it solves in 
\eqref{eq12} must be read as a (backward) stochastic Hamilton--Jacobi--Bellman equation on ${\mathcal S}_{d-1}$. 
The connection between \eqref{eq:master:equation:introl} and \eqref{eq12} is given by the following relationship.
\[u^i_t= U^i(t,P_t),\qquad\text{and}\qquad 
\nu_t^{i,j,k} = V^{i,j,k}(t,P_t), 
\]
where
\[
V^{i,j,k}(t,p)= \frac{\varepsilon}{\sqrt{2}} \left(\partial_{p_j}U^i(t,p)-\partial_{p_k} U^i(t,p)\right) \sqrt{{p_j p_k}}.
\] 
In fact, this relationship is the cornerstone to prove uniqueness of the solution of \eqref{eq:master:equation:introl} through a verification argument, see  Theorem \ref{lem:existence:mfg:from:master}. This argument is inspired from the original four-step-scheme in \cite{MaProtterYong}; in the framework of continuous state MFGs, it has already been 
used in \cite{CarmonaDelarue_book_II, CardaliaguetDelarueLasryLions, cha-cri-del2019}. 
On a more elaborated level, 
we should point out that the master equation has been a key tool
{(e.g., 
see 
\cite{CardaliaguetDelarueLasryLions} 
 for continuous state mean field games
and
\cite{bay-coh2019, cec-pel2019} 
for finite state mean field games)}
 to show convergence of the closed loop Nash-equilibrium of the $N$-player system to the MFG equilibrium. In both 
\cite{bay-coh2019, cec-pel2019}, there is no common noise and, in all the latter three cases, 
the Lasry--Lions monotonicity condition is assumed to be force, which is obviously in stark contrast 
to the setting of the current paper. The convergence problem in our setup is thus an interesting question, {which we resolve in our aforementioned work \cite{BCCD-game}.} 

{The rest of the paper is organized as follows. In Section \ref{se:MFG} we introduce the MFG model and provide preliminary versions of the main MFG and PDE results of the paper. The first one (Meta-Theorem 
\ref{main:metathm})
states that the MFG with common noise admits a unique solution; 
the second one (Meta-Theorem 
\ref{main:metathm-PDE}) states that the related master equation  
\eqref{eq:master:equation:introl}
has a unique smooth solution; the last one
(Theorem \ref{main:holder})
 is an a priori H\"older regularity estimate for linear PDEs driven by a merely continuous drift and a Kimura operator. 
In Section  
  \ref{sec:new:Kimura} we provide more material for the analysis of PDEs set on the simplex. This allows us to formulate more rigorous versions of 
  Meta-Theorems 
\ref{main:metathm} and 
\ref{main:metathm-PDE}, see Theorems \ref{main:thm}
and \ref{existence:master}. 
Section 
\ref{sec:MFG-master} 
is dedicated to the proofs of Theorems
 \ref{main:thm}
and \ref{existence:master}, taken for granted the a priori estimate 
from 
Theorem \ref{main:holder}. 
A key point therein is to make the connection between the master equation 
\eqref{eq:master:equation:introl} and 
 the MFG forward-backward system \eqref{eq12}. 
 The proof of Theorem \ref{main:holder} is the most demanding one of the paper and is given in Section \ref{sec:apriori}. The main two ingredients in the proof are the coupling construction provided in Proposition \ref{prop:coupling:2} and the induction step in Proposition \ref{prop:induction}.}

We now provide frequently used notation.

\vspace{5pt}
\noindent\textbf{Notation.}
For $a,b\in\R$, we let $a\wedge b:=\min\{a,b\}$. 
We use the notation $M^\dagger$ to denote the transpose of a matrix $M$. 
Moreover, we use the generic notation $p = (p_i)_{i \in \ES}$ (with $p$ in lower case and $i$ in subscript) for elements of $\R^d$, while processes are usually denoted by ${\boldsymbol P}=((P^i_t)_{i=1,\dots,d })_{0 \leq t \leq T}$ (with $P$ in upper case and $i$ in superscript).
For a subset {$A$} of a Euclidean space, we denote by 
 $\textrm{\rm Int}(A)$ the interior of $A$. 
Also, we recall the notation ${\mathcal S}_{d-1}:= \{ (p_{1},\cdots,p_{d}) \in (\RR_{+})^d : \sum_{i\in\llbracket d \rrbracket}p_{i}=1\}$, where $\llbracket d \rrbracket:=\{1,\ldots,d\}$. We can identify 
${\mathcal S}_{d-1}$
with the convex polyhedron of ${\mathbb R}^{d-1}$ $\hat{\mathcal S}_{d-1}:=\{ (x_{1},\cdots,x_{d-1}) \in (\RR_{+})^{d-1} : \sum_{i \in \llbracket d-1 \rrbracket} x_{i} \leq 1\}$. In particular, we sometimes write 
``the interior'' of ${\mathcal S}_{d-1}$; in such a case, we implicitly consider the interior 
 of ${\mathcal S}_{d-1}$ as the $(d-1)$-dimensional interior of $\hat{\mathcal S}_{d-1}$. 
Obviously, the interior of ${\mathcal S}_{d-1}$, when regarded as a subset of ${\RR}^d$, is empty, which makes it of a little interest. 
To make it clear, 
 for some $p \in {\mathcal S}_{d-1}$, 
 we sometimes write $p \in 
 \textrm{Int}(\hat{\mathcal S}_{d-1})$, meaning that 
 $p_{i}>0$ for any $i \in \ES$. 
We use the same convention when speaking about the boundary of ${\mathcal S}_{d-1}$:
 For some $p \in {\mathcal S}_{d-1}$, 
 we may write $p \in 
 \partial \hat{\mathcal S}_{d-1}$ to say that 
 $p_{i}=0$ for some $i \in \ES$. 
For $p \in {\mathcal S}_{d-1}$, we write $\sqrt{p}$ for the vector $(\sqrt{p_{1}},\cdots,\sqrt{p_{d}})$.  Finally, 
$\delta_{i,j}$ is the Kronecker symbol and $r_+$ denotes the positive part of $r\in\mathbb{R}$.

%
%
%

\section{Model and preliminary versions of the main results}
\label{se:MFG}

The purpose of this section is to introduce step by step {the model and then to provide preliminary versions of our main results. Although not definitive, those versions should help the reader to have a quick overview of the content of the paper. More complete statements are given in the next section. As we already accounted for in the introduction, our general objective is to prove that a relevant form of common noise may force uniqueness of equilibria to mean field games on a finite state space, see 
Meta-Theorem
\ref{main:metathm}. Our approach relies on the analysis of the master equation \eqref{eq:master:equation:introl}, whose solvability is addressed by means of 
a suitable smoothing property for so-called Kimura operators. We refer to Subsection 
 \ref{subsec:new:meta:PDE}
for an outlook on our PDE results. The complete description of Kimura operators  
is postponed to Section \ref{sec:new:Kimura}.} 


\subsection{A preliminary version of the mean field game}
\label{subse:3:1}

The first point that we need to clarify is the form of the mean field game itself. 
Whilst it is absolutely standard when there is no common noise, the mean field game addressed below takes indeed a more intricate and less obvious form in the presence of common noise. In fact, the somewhat non-classical structure that we use throughout the paper is specifically designed in order 
to be in correspondence with the class of second order differential operators on the simplex, referred to as Kimura operators in the text,
for which we can indeed prove the smoothing results announced in introduction, 
see Subsections
\ref{subsec:new:meta:PDE}
for a first account and
\ref{subse:weak:solvability}
for more details.

Clearly, the sharpest way to derive the form of mean field games that is used below would consist in going back to a game with a large but finite number $N$ of players and in justifying that, under the limit $N \rightarrow \infty$, this finite game converges in some sense 
to our form of mean field games. 
{Instead, we directly write down our version of mean field game with a common noise on a finite state space. 
Our rationale for doing so is that it allows the reader to jump quickly into the article. If she or he is interested, she or he may have a look at 
our companion paper
\cite{BCCD-game}, in which the discrete model is described in depth and the convergence problem is entirely resolved}.

\subsubsection{Mean field game without a common noise}
\label{subsubse:without}
When there is no common noise, our form of mean field game is directly taken from 
the earlier works 
\cite{GomesMohrSouza_discrete, Gomes2013}.
In short, a given tagged player  
evolves according to a Markov process with values in a finite state space $E$, which we will take for convenience as 
$E=\llbracket d \rrbracket :=\{1,\cdots,d\}$.
At any time $t \in [0,T]$, for a finite time horizon $T>0$, she chooses her transition rates in the form of 
a time-measurable $d \times d$--matrix $(\beta_{t}^{i,j})_{i,j \in \ES}$ satisfying 
\begin{equation}
\label{eq:Q-matrix}
\beta_{t}^{i,j} \geq 0, \quad i \not =j, \quad \beta_{t}^{i,i} = -\sum_{j \not = i} \beta_{t}^{i,j}, \quad t \in [0,T].
\end{equation}
Given the rates $((\beta^{i,j}_{t})_{i,j \in \llbracket d \rrbracket})_{0 \le t \le T}$, the marginal distribution
$((Q_{t}^i)_{i \in \llbracket d \rrbracket})_{0 \le t \le T}$ 
of the states of the tagged player evolves according to the discrete Fokker--Planck (or Kolmogorov) equation:
\begin{equation}
\label{eq:FPK:0noise}
\frac{\ud}{\ud t}Q_{t}^i = \sum_{j \in \ES} Q_{t}^j \beta_{t}^{j,i}, \quad t \in [0,T], 
\end{equation}
the initial statistical state $(Q_{0}^i)_{i \in \ES}$ being prescribed as an element of ${\mathcal S}_{d-1}$.
In words, $Q_{t}^i$ is the probability that the tagged player be in state $i$ at time $t$. 

With the tagged player, we assign a cost functional depending on a deterministic time-measurable 
${\mathcal S}_{d-1}$-valued
path $(P_{t})_{0 \le t \le T}$, referred to as an environment and starting from the same initial state as $(Q_{t})_{0 \le t \le T}$, namely $Q_{0}^i=P_{0}^i$ for $i \in \ES$. Intuitively, $P_{t}$ is understood as the statistical state 
at time $t$ of all the \textit{other} players in the continuum, which are basically assumed to be independent and identically distributed. 
Given $(P_{t})_{0 \le t \le T}$, the cost to the tagged player 
is written in the form
\begin{equation}
\label{eq:limit:cost:functional:0:noise}
{\mathcal J}\bigl((\beta_{t})_{0 \le t \le T},(P_{t})_{0 \le t \le T}\bigr)
:= \sum_{i \in \ES}
\biggl[ Q_{T}^i g (i,P_{T}) 
+ \int_{0}^T Q_{t}^i\Bigl(  f \bigl(t,i,P_{t} \bigr) + \frac12
 \sum_{j \not = i}
\bigl\vert \beta_{t}^{i,j} \bigr\vert^2
\Bigr) \ud t\biggr],
\end{equation}
where 
$g$ is a function from $\ES \times {\mathcal S}_{d-1}$ into ${\mathbb R}$
and 
$f$ is a function from $[0,T] \times \ES \times {\mathcal S}_{d-1}$ into ${\mathbb R}$.
To simplify the notations, we will sometimes write ${\boldsymbol \beta}$ for $(\beta_{t})_{0 \leq t \leq T}$
and ${\boldsymbol P}$ for $(P_{t})_{0 \leq t \leq T}$. Accordingly, we will write 
${\mathcal J}({\boldsymbol \beta},{\boldsymbol P})$ for the cost to the tagged player.

In this setting, a mean field equilibrium is a path ${\boldsymbol P}=(P_{t})_{0 \le t \le T}$ as before for which we can find 
an optimal control $(\beta_{t}^\star)_{0 \le t \le T}$ to ${\mathcal J}(\cdot,{\boldsymbol P})$
such that the corresponding solution to \eqref{eq:FPK:0noise} is $(P_{t})_{0 \le t \le T}$ itself. We stress the fact that here  ${\boldsymbol P}$ and ${\boldsymbol Q}$ are deterministic paths.

\subsubsection{Stochastic Fokker--Planck equation}\label{sec:fok-pla}
We now introduce a special form of common noise in order to force the equilibria to satisfy a relevant form of 
diffusion processes with values in the simplex ${\mathcal S}_{d-1}$. To make it clear, our aim is to force equilibria to satisfy the following stochastic variant of equation \eqref{eq:FPK:0noise}:
\begin{align}
\label{eq:weak:sde:1}
&\ud P_{t}^i = \sum_{j \in \ES}  P_{t}^j 
\alpha_{t}^{j,i}
\ud t + \frac{\varepsilon}{\sqrt 2} 
\sum_{j \in \ES}
\sqrt{P_{t}^i P_{t}^j}  
\ud \bigl[ W_{t}^{i,j} - W_{t}^{j,i} \bigr],
\end{align}
for $t \in [ 0,T]$, 
where 
$((W_{t}^{i,j})_{0 \leq t \leq T})_{i,j \in \ES : i \not = j}$
is a collection of 
independent $1d$ Brownian motions, referred to as the common noise, and 
${\boldsymbol \alpha}=(\alpha_{t})_{0 \leq t \leq T}$ is a progressively measurable process 
(with respect to the augmented filtration ${\mathbb F}^{{\boldsymbol W}}=({\mathcal F}_{t}^{\boldsymbol W})_{0 \leq t \leq T}$ 
generated by ${\boldsymbol W}=((W_{t}^{i,j})_{i,j \in \ES : i \not = j})_{0 \leq t \leq T}$) 
satisfying \eqref{eq:Q-matrix}. 
{All these processes are constructed on 
some probability space $(\Omega,{\mathcal A},{\mathbb P})$.}
Throughout, we use the convention ${\boldsymbol W}^{i,i}=(W^{i,i}_{t})_{0 \le t \le T} \equiv 0$, for any $i \in \ES$. Above, the parameter 
$\varepsilon$ reads as the intensity of the common noise. Accordingly, 
the collection $((\overline W^{i,j}_{t} := (W_{t}^{i,j}-W_{t}^{j,i})/\sqrt{2})_{0 \leq t \leq T})_{i,j \in \ES : i \not =j}$ 
forms an antisymmetric Brownian motion. 

Although it looks rather unusual, the form of the stochastic 
integration 
in  \eqref{eq:weak:sde:1} is in fact directly inspired by stochastic models of population genetics. To wit, for $i,j \in \ES$, the $(i,j)$-bracket writes
(with a somewhat abusive but quite useful 
notation in the first term in the right-hand side below)
\begin{equation}
\label{eq:bracket:weak:sde}
\begin{split}
\frac{\ud}{\ud t} \langle P^i,P^j \rangle_{t}
&=
\Bigl\langle \frac{\varepsilon}{\sqrt 2} 
\sum_{k \in \ES}
\sqrt{P_{t}^i P_{t}^k}  
\ud \bigl[ W_{t}^{i,k} - W_{t}^{k,i} \bigr],
\frac{\varepsilon}{\sqrt 2} 
\sum_{l \in \ES}
\sqrt{P_{t}^j P_{t}^l}  
\ud \bigl[ W_{t}^{j,l} - W_{t}^{l,j} \bigr] \Bigr\rangle
\\
&=
\varepsilon^2
\sum_{k,l \in \ES}
\sqrt{P_{t}^i P_{t}^j P_{t}^k
P_{t}^l}  \Bigl( \delta_{i,j}
\delta_{k,l}
- \delta_{i,l} \delta_{k,j} 
\Bigr)
=
\varepsilon^2 \Bigl( P_{t}^i \delta_{i,j}
-  P_{t}^i P_{t}^j
\Bigr) .
\end{split}
\end{equation}
The last term on the right-hand side is known as being the diffusion matrix of the Wright--Fisher model, see for instance 
 \cite{Feller,Ethier,Sato}. It is also the leading part of so-called Kimura operators,  see Subsection \ref{subse:weak:solvability}. 

Below, we will be specifically interested in cases when the equilibrium strategies are in feedback form, meaning that $\alpha_{t}^{i,j}=\upalpha(t,i,P_{t})(j)$ for a function $\upalpha : [0,T] \times \ES \times {\mathcal S}_{d-1} \times \ES
\ni (t,i,p,j) \mapsto \upalpha(t,i,p)(j) \in  \RR$ such that, for any $(t,p) \in [0,T] \times {\mathcal S}_{d-1}$ and any 
$i \in \ES$, 
\begin{equation}
\label{eq:constraint:B:drift}
\upalpha(t,i,p)(j) \geq 0, \quad j \in \ES \setminus \{i\}, \quad \upalpha(t,i,p)(i) = - \sum_{j \not = i} \upalpha(t,i,p)(j),
\end{equation} 
in which case 
\eqref{eq:weak:sde:1} becomes a stochastic differential equation, the well-posedness of which is addressed in the next section, at least in a setting that is relevant to us, see Proposition 
\ref{thm:approximation:diffusion:2}. 
The function $\upalpha$ is said to be a {\it feedback strategy}.
One of the key point in the latter statement is that
the solution takes values in ${\mathcal S}_{d-1}$ itself. Another key point is that, whenever each $p^i_{0}$ is in $(0,+\infty)$ and  each $\upalpha(t,j,p)(i)$ remains away from zero for $p_{i}$ is in the right neighborhood of $0$,
the coordinates of the solution are shown to remain almost surely (strictly) positive, which plays a crucial role in the definition of 
our mean field game with common noise. 
Below, we ensure strict positivity of the rate transition from $j$ to $i$ for $p_{i}$ small enough by forcing accordingly the dynamics   at the boundary of the simplex\footnote{\ We recall the convention introduced in the very beginning of the paper according to which the boundary is here understood as the boundary of $\hat{\mathcal S}_{d-1}$ under the identification of ${\mathcal S}_{d-1}$ and $\hat{\mathcal S}_{d-1}$ (and similarly for the interior). We take this convention for granted in the rest of the paper.} of the $(d-1)$-dimensional simplex (where $d$ is the cardinality of the state space); we make this point clear in \S \ref{subsub:repelled:from:boundary}.
Importantly, if this additional forcing at the boundary is strong enough, it also ensures that $\int_0^T 1/P^i_t dt$ has exponential moments of sufficiently high order, see Proposition \ref{expfin} 
 For the time being,  
we observe that the
strict positivity of the solution (provided that we take it for granted)  permits to rewrite the equation \eqref{eq:weak:sde:bis} in the form:
\begin{align}
\label{eq:weak:sde:bis}
&\ud P_{t}^i = \sum_{j \in \ES}  P_{t}^j 
\upalpha(t,j,P_{t})(i)
\ud t +  \varepsilon  
P_{t}^i \sum_{j \in \ES}
\sqrt{ \frac{P_{t}^j}{P_{t}^i}}  
\ud \overline W_{t}^{i,j}, \quad t \in [0,T],
\end{align}
 where, for consistency, we have replaced $\alpha_{t}^{j,i}$ by $\upalpha(t,j,P_{t})(i)$.

Now that we have equation \eqref{eq:weak:sde:bis},
we can formulate our mean field game. 
As we already accounted for, 
the first observation is that the Brownian motions 
$((W^{i,j})_{0 \le t \le T})_{i,j \in \ES : i \not = j}$ in 
\eqref{eq:weak:sde:bis} should be regarded as 
common noises (or the whole collection should be regarded as a common noise).
The second key point is that 
equation \eqref{eq:weak:sde:bis}
should be understood as the equation 
for an environment $(P_{t})_{0 \leq t \leq T}$, candidate for being a 
solution of the mean field game.
It thus remains to introduce the equation 
for a tagged player evolving within the environment $(P_{t})_{0\leq t \leq T}$. 
Our key idea in this respect is to linearize 
\eqref{eq:weak:sde:bis} in order to describe the 
statistical marginal states of the tagged player, provided $\int_0^T 1/P^i_t dt$ is enough exponentially integrable,
namely
\begin{align}
\label{eq:weak:sde:2}
&\ud Q_{t}^i = \sum_{j \in \ES}  Q_{t}^j 
\upbeta(t,j,P_{t})(i)
\ud t +  {\varepsilon} 
Q_{t}^i \sum_{j \in \ES}
\sqrt{ \frac{P_{t}^j}{P_{t}^i}}  
\ud  \overline W_{t}^{i,j}, \quad t \in [0,T],
\end{align}
where $\upbeta$ stands for the feedback function (hence satisfying 
\eqref{eq:constraint:B:drift})
used by the tagged player to implement her own strategy 
 in the form of a progressively-measurable (with respect to the filtration ${\mathbb F}^{\boldsymbol W}$ ) 
 process ${\boldsymbol \beta} = ((\beta_{t}^{i,j}=\upbeta(t,i,P_{t})(j))_{i,j \in \ES})_{0 \le t \le T}$. The main difficulty here is to interpret \eqref{eq:weak:sde:2} 
 in a convenient manner. 
 Notice in this regard that our choice to take here 
 $(\beta_{t})_{0 \le t \le T}$ in a closed feedback form (or semi-closed since it depends on the environment 
 $(P_{t})_{0 \le t\le T}$) 
 is only for consistency with 
 \eqref{eq:weak:sde:bis}
 and
just plays a little role in our interpretation.
Actually, one important fact in this respect is that  
the variable $Q_{t}^i$ in \eqref{eq:weak:sde:2} should read as a \textit{conditional expected mass} when the tagged player
is in state $i$ at time $t$. 
Here, the reader must be aware of the terminology that we use: We say conditional expected mass instead of conditional 
probability because, although  
$Q_{t}$ is shown below to have non-negative entries and to satisfy $\E[\sum_{i \in \ES} Q_{t}^i] =1$, it may not be a probability measure, meaning that 
$\sum_{i \in \ES} Q_{t}^i$ may differ from 1 with a positive probability, which is the whole subtlety of our model.

In order to clarify the equation \eqref{eq:weak:sde:2}, we may {indeed associate} a Lagrangian or particle representation with it. 
In the mean field game, we hence assume that the representative agent, at time $t$, has {not only a} position $X_t\in \ES$, but also 
{another feature $Y_t\in \R_+$, which we call a \textit{mass}.}
To state the dynamics more precisely, it is convenient 
{to enlarge the current probability space in the form of a product space 
$(\Omega \times \Xi,{\mathcal A} \otimes {\mathcal G}, {\mathbb P} \otimes {\mathbf P})$,
where 
$(\Xi,{\mathcal G}, {\mathbf P})$ denotes another probability space that is
just used here (and nowhere else in the text). 
Whilst 
$(\Omega,{\mathcal A}, {\mathbb P})$
is still equipped with the process ${\boldsymbol W}$, 
$(\Xi,{\mathcal G}, {\mathbf P})$ is now intended to carry the additional idiosyncratic noise 
to which the representative player is subjected. Accordingly, 
both processes $\boldsymbol{X}=(X_{t})_{0 \le t \le T}$ and $\boldsymbol{Y}=(Y_{t})_{0 \le t \le T}$
are constructed on the product space $\Omega \times \Xi$.} 
Our aim is then to show (at least informally) that, when the environment $\boldsymbol{P}$ is given {on the original space $(\Omega,{\mathcal A},{\mathbb P})$}, the solution $\boldsymbol{Q}$ of \eqref{eq:weak:sde:2} satisfies
\be
\label{Q=}
Q^i_t := 
{\mathbf E}\bigl[Y_t \mathbf{1}_{\{X_t=i\}} \bigr] , \qquad i\in \ES,
\ee
which means that $\boldsymbol{Q}$ is in fact the conditional expected mass of the reference player, {\textit{conditional} here being understood as
\textit{conditional on the common noise}}.
{For a given realization of the
common noise,  
the state 
process $\boldsymbol{X}$ is then
required to obey standard Markovian dynamics of the form 
\begin{equation}
\label{eq:markov:dyn}
{\mathbf P}\bigl(X_{t+h}=j |X_t=i \bigr) = \beta^{i,j}_t h +o(h), \quad t \in [0,T), \ h >0, \quad i,j \in \ES, \quad i \not =j,
\end{equation}
the transitions ${\boldsymbol \beta}$ 
being as in 
\eqref{eq:weak:sde:2}, namely ${\boldsymbol \beta} = ((\beta_{t}^{i,j}=\upbeta(t,i,P_{t})(j))_{i,j \in \ES})_{0 \le t \le T}$. Provided $\int_0^T 1/P^i_t \ud t$ is sufficiently exponentially integrable,
the dynamics of $\boldsymbol{Y}$
are instead given by an ${\boldsymbol X}$-dependent equation:
\be
\label{dynY}
\ud Y_t = \frac{\varepsilon}{\sqrt{2}}  Y_t \sum_{j\in\ES} \sqrt{\frac{P^j_t}{P^{X_t}_t}}\ud(W^{X_t,j}-W^{j,X_t}) =  \frac{\varepsilon}{\sqrt{2}} 
Y_t \sum_{i\in\ES} \mathbf{1}_{\{X_t=i\}}\sum_{j\in\ES} \sqrt{\frac{P^j_t}{P^{i}_t}}\ud(W^{i,j}-W^{j,i}).
\ee
Writing informally $\ud Q_{t}^i$ as 
\begin{equation*}
\ud Q_{t}^i = {\mathbf E}\bigl[{\mathbf 1}_{\{X_{t}=i\}} \ud Y_{t}\bigr] + 
\sum_{j \in \ES} Q_{t}^j \beta^{j,i}_t \ud t, \quad t \in [0,T], \quad i \in \ES,
\end{equation*}
we then recover 
\eqref{eq:weak:sde:2}. For sure, this argument may be made rigorous. In particular, it may be very convenient to represent the 
Markov dynamics of ${\boldsymbol X}$ 
in 
\eqref{eq:markov:dyn}
by means of an extra Poisson random measure constructed on $(\Xi,{\mathcal G},{\mathbf P})$, very much in the spirit of \cite{Cecchin2017}.}

\begin{remark}
{The very mechanism underpinning   
\eqref{eq:markov:dyn} and \eqref{dynY} is in fact different
from the true Wright--Fisher model even though it shares some similarities. Instead, when working with a finite version of the game, 
we can force the players with a pure Wright--Fisher noise by resampling their states (in $\ES$) independently, according to their instantaneous empirical distribution at the times of a Poisson process. Such an approach is simpler as it bypasses any additional mass process ${\boldsymbol Y}$, but it has a serious drawback for the game: The empirical measure is then strongly attractive since the system is resampled recurrently. This feature precludes any interesting deviating phenomenon, whilst, in our model, the tagged player may really deviate (in law) from the  population. In the end, the macroscopic equation \eqref{eq:weak:sde:bis} for the population
is a Wright--Fisher diffusion,
but the microscopic behavior of one particle in 
\eqref{eq:markov:dyn}
and \eqref{dynY}
is different. Noticeably, 
  the state $({\boldsymbol X},{\boldsymbol Y})$ of the tagged player is hence encoded, at any time $t \in [0,T]$, in the form 
  $(Y_{t} {\mathbf 1}_{\{X_{t}=i\}})_{i \in \ES}$. 
  Since $Q_{t}$ in 
  \eqref{Q=}
  is just the expectation of the latter quantity, the state variable of the tagged player thus takes values within the same 
  space as the feature $Q_{t}$ that is used to describe 
  its statistical behavior. 
  }   
\end{remark}

\subsubsection{Cost functional and first formulation of the game} It now remains to associate a cost functional with the tagged player. 
Consistently with 
\eqref{eq:limit:cost:functional:0:noise} 
we here let 
\begin{equation}
\label{eq:limit:cost:functional}
{\mathcal J}\bigl(\upbeta,(P_{t})_{0 \le t \le T}\bigr)
:= \sum_{i \in \ES}
{\mathbb E}\biggl[ Q_{T}^i g (i,P_{T}) 
+ \int_{0}^T Q_{t}^i\Bigl(  f \bigl(t,i,P_{t} \bigr) + \frac12
 \sum_{j \not = i}
\bigl\vert \bigl[ \upbeta(t,i,P_{t} \bigr) \bigr](j) \bigr\vert^2
\Bigr) \ud t
\biggr].
\end{equation}
In the above left-hand side, $\upbeta$ stands for the strategy used in the equation for ${\boldsymbol Q}=(Q_{t})_{0 \le t \le T}$ in 
\eqref{eq:weak:sde:2}; also, ${\boldsymbol P}=(P_{t})_{0 \le t \le T}$ denotes the environment (as  the cost functional does depend
upon the environment), defined as the solution of 
\eqref{eq:weak:sde:bis}.
\vskip 5pt

Hence, 
for an initial condition 
$p_{0}=(p_{0,i})_{i \in \ES} \in {\mathcal S}_{d-1}$ with positive entries (that is $p_{0,i} >0$ for each $i \in \ES$),
our definition of a mean field game solution  comes in the following three steps:
\begin{enumerate}
\item Consider a feedback function $\upalpha$
as in 
\eqref{eq:constraint:B:drift}
 such that 
\eqref{eq:weak:sde:bis}, with $p_{0}$ as initial condition, has a unique solution 
$(P_{t})_{0 \leq t \leq T}$ (say on {the} probability space $(\Omega,{\mathcal A},{\mathbb P})$ equipped with 
a collection of $1d$ Brownian motions $((W_{t}^{i,j})_{0 \le t \le  T})_{i,j \in \ES : i \not =j}$,
with the same convention as before that 
${\boldsymbol W}^{i,i}=(W^{i,i}_{t})_{0 \le t \le T} \equiv 0$ for $i \in \ES$), which 
remains positive with probability 1;
the process $(P_{t})_{0 \leq t \leq T}$ is then called an environment;
\item On the same space $(\Omega,{\mathcal A},{\mathbb P})$, solve, for any bounded and measurable feedback function 
$\upbeta$, equation \eqref{eq:weak:sde:2} for 
$(Q_{t})_{0 \leq t \le T}$   with $p_{0}$ as initial condition, and then find the optimal trajectories (if they exist) of the minimization problem
\begin{equation}
\label{eq:inf:B'}
\inf_{\upbeta}{\mathcal J}\bigl(\upbeta,{\boldsymbol P}\bigr). 
\end{equation}
\item Find an environment $(P_{t})_{0 \leq t \leq T}$
such that $(P_{t})_{0 \leq t \leq T}$ is an optimal trajectory of 
\eqref{eq:inf:B'}. 
Such a $(P_{t})_{0 \le t \le T}$ is called an MFG equilibrium or a solution to the MFG. 
\end{enumerate}
\vskip 5pt

The precise definition is given in the next section (Definition \ref{defmfg}). 
Let us remark that an equivalent definition could be also given in terms of the particle representation described above. Indeed, for a given environment $\boldsymbol{P}$ and control $\upbeta$, if $\boldsymbol{X}$ denotes the corresponding position of the reference player and $\boldsymbol{Y}$ its mass, the cost \eqref{eq:limit:cost:functional} rewrites, thanks to \eqref{Q=}, as
\begin{equation*}
{\mathcal J}\bigl(\upbeta,(P_{t})_{0 \le t \le T}\bigr):= 
{\mathbb E}\biggl[ Y_{T} g (X_T,P_{T}) 
+ \int_{0}^T Y_{t}\Bigl(  f \bigl(t,X_t,P_{t} \bigr) + \frac12
 \sum_{j \not = X_t}
\bigl\vert \bigl[ \upbeta(t,X_t,P_{t} \bigr) \bigr](j) \bigr\vert^2
\Bigr) \ud t
\biggr].
\end{equation*}
Then, a mean field game solution can be defined as a couple $(\upalpha, \boldsymbol{P})$ such that $\upalpha$ is optimal for \eqref{eq:inf:B'} and 
$P_t^i = \E[Y_t \mathbf{1}_{\{X_t=i\}} | \mathcal{F}_t^{\boldsymbol W}]$, where $\boldsymbol{X}$ and $\boldsymbol{Y}$ satisfy \eqref{eq:markov:dyn} and \eqref{dynY} for the given $\upalpha$ and $\boldsymbol{P}$.

It is worth noticing that, whenever $\varepsilon=0$ in 
\eqref{eq:weak:sde:bis} and \eqref{eq:weak:sde:2}, the system 
\eqref{eq:weak:sde:bis}--\eqref{eq:weak:sde:2}
becomes a simpler system of two decoupled Fokker--Planck (or Kolmogorov) equations
that are similar to \eqref{eq:FPK:0noise}. 
While existence of a mean field game solution  (in the case $\varepsilon=0$) is by now well-understood, uniqueness remains 
a difficult issue. In fact, there are few generic conditions that ensure uniqueness.
Generally speaking, the two known instances of uniqueness are (besides some specific examples that can be treated case by case) the short time horizon case (namely $T$
is small enough in comparison with the regularity properties of the underlying cost coefficients)
and the so-called monotonous case due to Lasry and Lions \cite{LasryLions,LasryLions2}
 (which does not require $T$ to be small enough). 
In short, the cost coefficients $f$ and $g$ are said to be monotonous (in the sense of Lasry and Lions) if, for any 
$p,p' \in {\mathcal S}_{d-1}$ and for any $t \in [0,T]$,  
\begin{equation}
\label{eq:monotonous}
\sum_{i \in \ES} \bigl( g(i,p) - g(i,p') \bigr) \bigl( p_{i}- p_{i}' \bigr) \geq 0,
\quad
\sum_{i \in \ES} \bigl( f(t,i,p) - f(t,i,p') \bigr) \bigl( p_{i} - p_{i}' \bigr) \geq 0.
\end{equation}
The main goal of the rest of the paper is precisely to prove that, whenever $\varepsilon$ in 
\eqref{eq:weak:sde:bis}--\eqref{eq:weak:sde:2} is strictly positive,
 uniqueness may hold true for our MFG under quite mild regularity conditions 
on the coefficients and in particular without requiring any monotonicity properties; in fact, the main constraint that we ask is that 
the coordinates of the solutions 
of \eqref{eq:weak:sde:bis} stay sufficiently far away from zero (provided that the coordinates of the initial condition 
themselves are not zero). We address this requirement in the next subsection: Basically, it will prompt us
to introduce a new term in the dynamics  of both $(P_{t})_{0 \leq t \le T}$
	and 
	$(Q_{t})_{0 \le t \le T}$
to force the coordinates to stay positive.	
	
\subsection{New MFG and first meta-statement}
As we already alluded to, an important observation is that, for any solution $((P_{t}^i)_{i \in \ES})_{0 \le t \le T}$
to
\eqref{eq:weak:sde:bis}, it holds
\begin{equation}
\label{eq:sum=1}
\ud \Bigl( \sum_{i \in \ES} P_{t}^i \Bigr) =0,
\end{equation}
which can be easily proved by summing over the coordinates in 
\eqref{eq:weak:sde:bis}.
In particular, since the initial condition is taken in ${\mathcal S}_{d-1}$, 
the mass remains constant, equal to 1. Subsequently, if the coordinates of $(P_{t})_{0 \le t \le T}$ 
remain non-negative (which we discuss right below), 
the process $(P_{t})_{0 \le t \le T}$ lives in ${\mathcal S}_{d-1}$, which is of a special interest for us.
In fact, non-negativity of the coordinates may be easily seen by rewriting 
\eqref{eq:weak:sde:bis} in the form
\begin{equation}
\label{eq:pti:marginal}
\ud P_{t}^i = 
a_{i}(t,P_{t})\ud t + 
 \varepsilon \sqrt{ P_{t}^i ( 1-P_{t}^i)}
\ud \widetilde W_{t}^i,
\end{equation}
for a new Brownian motion $(\widetilde W_{t}^i)_{0 \leq t \leq T}$, 
where ${a_{i}(t,p)} := \sum_{j \in \ES} [p_{j} \upalpha(t,j,p)(i) - p_{i} \upalpha(t,i,p)(j)]$, for 
$i \in \ES$,
 and 
where the form of the stochastic integral follows from 
\eqref{eq:bracket:weak:sde}
with $i=j$ therein.
(Notice that the form of $a_{i}$ differs from 
the writing used in 
\eqref{eq:weak:sde:bis}, but both are obviously equivalent since $\sum_{j \in \ES} \upalpha(t,i,p)(j)=0$.) We have that 
$a_{i}(t,p) \geq - C p_{i}$, for a constant $C>0$, since 
$\upalpha(t,j,p)(i) \geq 0$ for $j \not =i$ and we assume $\upalpha$ bounded. 
By stochastic comparison with Feller's ($1d$) branching diffusion \cite[Exercise 5.1]{Durrett}, we easily deduce 
that the coordinates of $(P_{t})_{0 \leq t \leq T}$ should remain 
non-negative (the details are left to the reader and a rigorous statement, tailored to our framework, is given below). 

\subsubsection{Equations that are repelled from the boundary}
\label{subsub:repelled:from:boundary}
In the sequel, we are interested in solutions
to \eqref{eq:weak:sde:bis} 
that stay sufficiently far away from the boundary. 
As we already explained several times, the reason is that our uniqueness result is based upon the smoothing properties of the operator generated by \eqref{eq:weak:sde:bis}. Since the latter degenerates at the boundary of the simplex, we want to keep the 
solutions to \eqref{eq:weak:sde:bis} as long as possible within the relative 
interior of ${\mathcal S}_{d-1}$.  
In this regard, 
it is worth observing from \cite[Exercise 5.1]{Durrett}  that the sole condition
\eqref{eq:constraint:B:drift} is not enough to prevent solutions 
to 
\eqref{eq:weak:sde:bis}
to touch the boundary of the simplex. 
To guarantee that no coordinate vanishes, more is needed. For instance, in
 Feller's branching diffusions,
 the solution does not vanish if the drift is sufficiently positive in the neighborhood 
 of $0$. This prompts us to revisit the 
two equations \eqref{eq:weak:sde:bis} and
\eqref{eq:weak:sde:2}
and to consider instead (notice that, in the two formulas below, the value of 
$\upalpha(t,i,P_{t})(i)$ is in fact useless)
\begin{equation}
\label{eq:weak:sde:final}
\begin{split}
\ud P_{t}^i &=  \sum_{j \in \ES} \Bigl( P_{t}^j \bigl( \varphi(P_{t}^i) + \upalpha(t,j,P_{t})(i) \bigr) - P_{t}^i \bigl( \varphi(P_{t}^j) 
+ \upalpha(t,i,P_{t})(j) \bigr)
\Bigr) \ud t
 + {\varepsilon}  
\sum_{j \in \ES}
\sqrt{P_{t}^i P_{t}^j}  
\ud \overline W_{t}^{i,j},
\end{split}
\end{equation}
and
\begin{equation}
\label{eq:weak:sde:final:2:b}
\begin{split} 
\ud Q_{t}^i &=   \sum_{j \in \ES} \Bigl( Q_{t}^j \bigl( \varphi(P_{t}^i) + \upbeta(t,j,P_{t})(i) \bigr) - Q_{t}^i 
\bigl( \varphi(P_{t}^j) 
+  \upbeta(t,i,P_{t})(j) \bigr)
\Bigr) \ud t
+  {\varepsilon}  
Q_{t}^i \sum_{j \in \ES}
\sqrt{\frac{P_{t}^j}{P_{t}^i}}  
\ud \overline W_{t}^{i,j}, 
\end{split}
\end{equation}
for $t \in [0,T]$, with the same deterministic initial condition $P_0=(P_{0}^i=p_{0,i})_{i \in \ES}$. Here the function $\varphi$ is 
 a non-increasing Lipschitz function from $[0,\infty)$ into itself such that 
\begin{equation}
\label{eq:varphi}
\varphi(r) := \left\{ \begin{array}{l}
\kappa \quad r \leq \delta,
\\
0 \quad r >2 \delta,
\end{array}
\right.
\end{equation}
$\delta$ being a positive parameter whose value next is somewhat arbitrary. As for $\kappa$, we clarify its {main role in the statements of 
Meta-Theorems 
\ref{main:metathm}
and 
\ref{main:metathm-PDE}
and of
Theorems \ref{main:holder}, 
\ref{main:thm} and \ref{existence:master}; the reader may also find a taste of it 
in the three propositions right below}. 
In the two equations
\eqref{eq:weak:sde:final}
and
\eqref{eq:weak:sde:final:2:b}, 
$\upalpha$ and $\upbeta$ are the same as in 
\eqref{eq:weak:sde:bis}
and
\eqref{eq:weak:sde:2}. 
Hence, the drift in the first equation now reads 
\begin{equation}
\label{eq:bi:new}
a_{i}(t,p) := \sum_{j \in \ES}
\Bigl( p_{j}
\bigl[ \varphi(p_{i}) + \upalpha\bigl( t,j,  p \bigr)(i) 
\bigr]
 -
p_{i} \bigl[ \varphi(p_{j}) + 
 \upalpha\bigl( t,i,p \bigr)(j) 
 \bigr]
 \Bigr). 
 \end{equation}
 It still satisfies $\sum_{i \in \ES} a_{i}(t,p)=0$. And, importantly, 
whenever $p_{i}=0$ (with $p=(p_{1},\cdots,p_{d}) \in {\mathcal S}_{d-1}$), it satisfies $a_{i}(t,p) \geq \kappa$.
In this framework, we have the following three statements, the proofs of which are postponed to  
Subsection \ref{subse:proofs}.

\begin{proposition}
\label{thm:approximation:diffusion:2}
Consider
$\varphi$ as in  
\eqref{eq:varphi}
with $\delta \in (0,1)$ and $\kappa \geq \varepsilon^2/2$, for $\varepsilon >0$. Then,
for a bounded (measurable) feedback function $\upalpha$ as in \eqref{eq:constraint:B:drift},
the stochastic differential equation
\eqref{eq:weak:sde:final}
 has a unique (strong) solution 
whenever the initial condition is prescribed and satisfies $p_{0,i}>0$ for each $i \in \ES$ and $\sum_{i \in \ES} p_{0,i}=1$. 
Moreover, the coordinates of the solution remain almost surely (strictly) positive and satisfy $\sum_{i\in\ES} P^i_t=1$, for any time. 
\end{proposition}

The following statement provides a stronger version.

\begin{proposition}
\label{expfin}
Under the assumptions and notation of Proposition 
\ref{thm:approximation:diffusion:2},
for  $\kappa$ as in 
\eqref{eq:varphi}
and for $\lambda >0$,
let 
$\gamma : = \kappa - {(1+\lambda)\varepsilon^2}/2 $. Then,
 the solution to 
 \eqref{eq:weak:sde:final}
 satisfies
\begin{equation}
\label{eq:expfin2}
{\mathbb E}\left[ \exp \left(\lambda \gamma \int_{0}^T  \frac{1}{P^{i}_{s}} \ud s\right)\right] \leq 
C p_{0,i}^{-\lambda},
\quad \mbox{for each } i \in \ES,
\end{equation}
together 
with
\begin{equation}
\label{eq:expfin2:bis}
\sup_{0 \le t \le T} {\mathbb E}\left[ 
\bigl( P^i_{t} \bigr)^{-\lambda}
\right] \leq 
C p_{0,i}^{-\lambda}, \quad \mbox{for each } i \in \ES,
\end{equation}
for a constant $C$ that only depends on $\delta$, $\kappa$, $\lambda$, $T$ and on the supremum norm of $\upalpha$.

\end{proposition}


\begin{proposition}
\label{expfin:2}
Under the assumption and notation of Proposition 
\ref{thm:approximation:diffusion:2}, 
assume that $\upbeta$ is bounded and measurable and that $\kappa$ in 
\eqref{eq:varphi}
satisfies 
$\kappa \geq {(61+d)} \varepsilon^2$. 
Assume also that 
the initial condition  
$p_{0}=(p_{0,i})_{i \in \ES} \in {\mathcal S}_{d-1}$ 
of 
\eqref{eq:weak:sde:final}
has positive entries (that is $p_{0,i} >0$ for each $i \in \ES$). 
Then, for any initial condition $q_{0}=(q_{0,i})_{i \in \ES} \in {\mathcal S}_{d-1}$ (possibly different from $p_{0}$), 
 \eqref{eq:weak:sde:final:2:b}  
has a unique pathwise solution $((Q_{t}^i)_{0 \leq t \leq T})_{i \in \ES}$.
It satisfies ${\mathbb E}[\sup_{0 \leq t \leq T} \vert Q_{t}^i \vert^{4}] < \infty$ for any $i \in \ES$. The coordinates of the solution  are (strictly) positive and satisfy
\begin{equation}
\label{eq:mean:q:1}
{\mathbb E} \biggl[ \sum_{i \in \ES} Q_{t}^i \biggr] = 1.
\end{equation}
\end{proposition}

\subsubsection{Reformulation of the game}
\label{subse:new:MFG}
We now have most of the needed ingredients to formulate the setting to which our main result applies. 
Roughly speaking, our result addresses the mean field game associated with the
pair  
\eqref{eq:weak:sde:final}--\eqref{eq:weak:sde:final:2:b}
instead of 
\eqref{eq:weak:sde:bis}--\eqref{eq:weak:sde:2} 
and with the cost functional 
\eqref{eq:limit:cost:functional}; 
so this is the same MFG as the one described in Subsection 
\ref{subse:3:1} except for the fact that 
we included the forcing $\varphi$
in the state equations and that, in the equation \eqref{eq:weak:sde:final:2:b}, we will
allow for a more general (random) rate function instead of 
the feedback function $\upbeta$.
So, in lieu of 
\eqref{eq:weak:sde:final:2:b}, we will consider (in the mean field game)
\begin{equation}
\label{eq:weak:sde:final:2}
\begin{split} 
\ud Q_{t}^i &=   \sum_{j \in \ES} \Bigl( Q_{t}^j \bigl( \varphi(P_{t}^i) + \beta_{t}^{j,i} \bigr) - Q_{t}^i 
\bigl( \varphi(P_{t}^j) 
+  \beta_{t}^{i,j} \bigr)
\Bigr) \ud t
+  {\varepsilon}  
Q_{t}^i \sum_{j \in \ES}
\sqrt{\frac{P_{t}^j}{P_{t}^i}}  
\ud \overline W_{t}^{i,j}, 
\quad t \in [0,T]. 
\end{split}
\end{equation}
Here, 
$((\beta_{t}^{i,j})_{0 \leq t \leq T})_{i,j \in \ES}$ is a collection of bounded ${\mathbb F}^{\boldsymbol W}$-progressively-measurable processes that are required to satisfy
\begin{equation}
\label{eq:condition:Q-matrix}
\beta_{t}^{i,j} \geq 0 \quad i,j \in \ES, \, i \not = j.
\end{equation}
(Notice that the diagonal terms $((\beta_{t}^{i,i})_{0 \leq t \leq T})_{i \in \ES}$
play no role.)
Such processes are called {\it admissible open-loop strategies}.
Somehow, this is to say that we can work 
(at least for 
\eqref{eq:weak:sde:final:2:b})
with strategies that may depend upon the whole past of the environment
${\boldsymbol P}=(P_{t})_{0 \leq t \leq T}$, which is in contrast to strategies of the form $(\upbeta(t,i,P_{t})(j))_{0 \leq t \leq T}$ in 
\eqref{eq:weak:sde:final:2:b} which depend, at time $t$, on the environment
through its current state only. 
Latter strategies are said to be semi-closed. We explain below how such semi-closed strategies manifest 
through the master equation. 

\begin{remark} 
\label{foot:extension:alpha:open loop}
We let the reader check that Propositions 
\ref{thm:approximation:diffusion:2}
and
\ref{expfin} given {above}, see Subsection 
\ref{subse:proofs} for their proofs, can be also extended to the case where the 
process
$((\upalpha(t,i,P_{t})(j))_{i,j \in \ES})_{0 \le t \le T}$ 
in 
\eqref{eq:weak:sde:final}
is replaced by a more general bounded progressively measurable process 
$((\alpha_{t}^{i,j})_{i,j \in \ES})_{0 \le t \le T}$
satisfying
the analogue of 
\eqref{eq:condition:Q-matrix}.
The proof of 
the solvability of 
\eqref{eq:weak:sde:final}
in the statement of 
Proposition \ref{thm:approximation:diffusion:2}
is even simpler since 
$((\alpha_{t}^{i,j})_{i,j \in \ES})_{0 \le t \le T}$
is then in open-loop form.
Similarly, we can replace 
\eqref{eq:weak:sde:final:2:b}
by
\eqref{eq:weak:sde:final:2}
in the statement of 
\eqref{expfin:2}. 
\end{remark}

Importantly, 
we regard the two cost coefficients $f$ and $g$
in 
\eqref{eq:limit:cost:functional}
as being defined on $[0,T] \times  \ES \times {\mathcal S}_{d-1}$ and 
$\ES \times {\mathcal S}_{d-1}$ respectively. 
To make it simpler, we write $g^{i}(p)$ for $g(i,p)$ and $f^{i}(t,p)$ for $f(t,i,p)$. Accordingly, 
for a progressively-measurable
strategy
${\boldsymbol \beta}=((\beta_{t}^{i,j})_{i,j \in \ES})_{0 \leq t \leq T}$, the cost functional becomes
\begin{equation}
\label{eq:limit:cost:functional:new}
\begin{split}
{\mathcal J}\bigl({\boldsymbol \beta},{\boldsymbol P}\bigr)
&:= \sum_{i \in \ES}
{\mathbb E}\biggl[ Q_{T}^i g^{i} (P_{T}) 
+ \int_{0}^TQ_{t}^i
 \Bigl(   f^{i} \bigl(t,P_{t} \bigr) + \frac12
 \sum_{j \not = i}
\bigl\vert \beta_{t}^{i,j}  \bigr\vert^2
\Bigr) \ud t
\biggr],
\end{split}
\end{equation}
where 
$(Q_{t})_{0 \le t \le T}$ 
solves
\eqref{eq:weak:sde:final:2}, 
with $Q_0=P_0$ (the latter being equal to some deterministic $p_{0} \in {\mathcal S}_{d-1}$).

\begin{definition}
\label{defmfg}
Given a deterministic initial condition $p_0$, a \emph{solution of the mean field game (with common noise)} is a pair $({\boldsymbol P}, \upalpha)$ 
such that
\begin{itemize}
\item[(i)] ${\boldsymbol P}=(P_{t})_{0 \leq t \leq T}$ is an $\mathcal{S}_{d-1}$-valued process, progressively measurable with respect to $\mathbb{F}^{\boldsymbol W}$, with $p_{0}$ as initial condition, and 
$\upalpha : [0,T] \times \ES \times {\mathcal S}_{d-1} \times \ES \rightarrow \RR$ is a bounded feedback strategy;
\item[(ii)] ${\boldsymbol P}$ and $\upalpha$ satisfy Equation \eqref{eq:weak:sde:final} in the strong sense;
\item[(iii)] ${\mathcal J}\bigl(\upalpha,{\boldsymbol P}\bigr)\leq{\mathcal J} \bigl({\boldsymbol \beta},{\boldsymbol P}\bigr)$ for any admissible open-loop strategy ${\boldsymbol \beta}$.
\end{itemize}
We say that the solution $({\boldsymbol P},\upalpha)$ is unique if given another solution $(\widetilde{\boldsymbol P}, 
 \widetilde{\upalpha})$, we have 
$P_t =\widetilde{P}_t$  for any $t\in [0,T]$, $\mathbb{P}$-a.s., and 
$\upalpha(t,i,P_{t})(j)=\widetilde \upalpha(t,i,P_{t})(j)$
 $dt\otimes\mathbb{P}$-a.e., for each $i,j\in\ES$.

\end{definition}

We recall that the probability space and the Brownian motion ${\boldsymbol W}$ are fixed and then Equations \eqref{eq:weak:sde:final} and \eqref{eq:weak:sde:final:2} have unique strong solutions. The above hence defines {\it strong} mean field game solutions, in the sence that ${\boldsymbol P}$ is adapted to $\mathbb{F}^{\boldsymbol W}$. For a comparison between strong and weak MFG solutions, in the diffusion case, we refer to \cite[Chapter 2]{CarmonaDelarue_book_II}.

\vspace{5pt}
Here is now a meta form of our main statement.

\begin{metatheorem}
\label{main:metathm}
Assume that
the coefficients $f$ and $g$ are \emph{sufficiently regular}.
Then,
for any $\varepsilon \in (0,1)$,
there exists a threshold $\kappa_{0}>0$, 
only depending 
on $\varepsilon$, $\|f\|_{\infty}$, $\| g \|_{\infty}$ and $T$, such that, for any $\kappa \geq \kappa_{0}$
and $\delta \in (0,1/(4\sqrt{d}))$,
and 
for any (deterministic) initial condition $(p_{0,i})_{i \in \ES} \in {\mathcal S}_{d-1}$ with positive entries, the mean field game 
has a unique solution as defined by Definition \ref{defmfg}.
\end{metatheorem}

{The statement is said to be in meta-form since the assumptions on $f$ and $g$ are not clear. 
The definitive version is given in Theorem \ref{main:thm} below
and the proof is completed in  Section \ref{subse:MFG:proof:existence:!}. }

\begin{remark}
\label{rem:meta-thm}
At this stage, it is worth mentioning that our notion of solution, as defined in 
Definition \ref{defmfg}, could be relaxed: Instead of requiring the strategies to be in feedback form (namely, 
in the form $((\upalpha(t,i,P_t)(j))_{i,j \in \ES})_{0 \le t \le T}$), we could allow them to be 
in open-loop form (namely, to be given by more general bounded progressively-measurable processes
$((\alpha_{t}^{i,j})_{i,j \in \ES})_{0 \le t \le T}$). Our claim is that Meta-Theorem \ref{main:metathm}, and in fact  
Theorem \ref{main:thm} as well, extend to this case: The solution given by 
Meta-Theorem \ref{main:metathm} and 
Theorem \ref{main:thm} remains unique within the larger class of open-loop solutions. The proof is exactly the same. 
Actually, 
our choice to use feedback strategies is for convenience only since we feel better to keep, in our main statements, 
  the same framework as the one used in the exposition of the problem.

Moreover, we are confident that our result also extends to random initial conditions, but the proof would certainly require an additional effort since 
the initial conditions should then satisfy suitable integrability properties. To wit, 
an expectation must be added to the right-hand side of both 
\eqref{eq:expfin2}
and 
\eqref{eq:expfin2:bis} when $(p_{0,i})_{i \in \ES}$ becomes random: The resulting expectations might be infinite unless some integrability properties are indeed satisfied.
\end{remark}

\subsection{Overview of the PDE results}
 \label{subsec:new:meta:PDE}

The proof of Meta-Theorem 
\ref{main:metathm}
(or equivalently of the definitive version Theorem \ref{main:thm})
is highly based upon the so-called master equation associated with the mean field game. We refer to
\cite{Cardaliaguet,CardaliaguetDelarueLasryLions,CarmonaDelarue_book_II,Lionscollege1,Lionsvideo} for foundations of the topics for mean field games set on 
${\mathbb R}^d$ and to 
\cite{bay-coh2019, cec-pel2019, BertucciLasryLions,Lionscollege2} for related issues for mean field games with a finite state space. 

\subsubsection{Well-posedness of the master equation}

Generally speaking, the master equation here takes the form of a system of nonlinear parabolic equations driven by a
so-called Kimura operator, the latter being carefully described in Section 
\ref{sec:new:Kimura}. This system is nothing but 
the system 
of second order PDEs 
\eqref{eq:master:equation:introl}
stated on the 
($d-1$)-dimensional simplex.  Although the formulation 
of 
\eqref{eq:master:equation:introl}
looks fine, it is in fact rather abusive since ${\mathcal S}_{d-1}$ has an empty interior in $\RR^d$: In other words, except if the unknown $U^i$ therein is defined on a neighborhood of the simplex, 
the derivatives that appear in the equations are not properly defined. 
Although we just clarify this point in the next section, we feel useful to state 
a preliminary version 
of our main result on the well-posedness of 
\eqref{eq:master:equation:introl}. Very much in the spirit of 
Meta-Theorem 
\ref{main:metathm}, we have it in the form of a meta-statement. 
  \begin{metatheorem}
\label{main:metathm-PDE}
Assume that
the coefficients $f$ and $g$ are \emph{sufficiently regular}.
Then,
for any $\varepsilon \in (0,1)$,
there exists a threshold $\kappa_{0}>0$, 
only depending 
on $\varepsilon$, $\|f\|_{\infty}$, $\| g \|_{\infty}$ and $T$, such that, for any $\kappa \geq \kappa_{0}$
and $\delta \in (0,1/(4\sqrt{d}))$,
the master equation \eqref{eq:master:equation:introl}
has a unique smooth solution $U=(U^1,\cdots,U^d)$.  
\end{metatheorem}
Similar to Meta-Theorem \ref{main:metathm}, 
Meta-Theorem \ref{main:metathm-PDE}
leaves unclear 
 the assumptions on $f$ and $g$ as well as the smoothness of the solution of the master equation. 
The definitive version is given in Theorem \ref{existence:master} below; this includes a clear definition of the 
various derivatives
that appear in 
\eqref{eq:master:equation:introl}. 
We refer to Section 
\ref{subse:master:e!} for the proof.

\subsubsection{Smoothing estimates for PDEs set on the simplex}
Solvability of the master equation 
\eqref{eq:master:equation:introl}
is clearly the main issue in our analysis. 
In fact, the key step in our approach is 
to prove a priori estimates for 
a so-called \textit{linear} version of the master equation, the latter being obtained by 
freezing the nonlinear component $U$ therein. To make it clear, 
the linear analogue of 
each equation in the system
\eqref{eq:master:equation:introl}
may be written in the generic form
\begin{equation}
\label{eq:pde:apriori:1:intro}
\begin{split}
&\partial_{t} u(t,p) + \sum_{j \in \ES} \Bigl( \varphi(p_{j}) + b_{j}(t,p) + p_{j} b_{j}^{\circ}(t,p)  \Bigr) \partial_{p_{j}} u(t,p)
\\ 
&\hspace{15pt} +\frac{\varepsilon^2}{2} \sum_{j,k \in \ES}\bigl(p_j \delta_{jk}-p_{j} p_{k}\bigr) \partial^2_{p_{j} p_{k}} u(t,p)
+ h(t,p) = 0,
\\
&u(T,p) = \ell(p),
\end{split}
\end{equation}
for $(t,p) \in [0,T] \times \mathrm{Int}(\hat{{\mathcal S}}_{d-1})$ ($\mathrm{Int}(\hat{{\mathcal S}}_{d-1})$ being here regarded as a subset of ${\mathcal S}_{d-1}$), $b=(b_{j})_{j \in \ES} : [0,T] \times {\mathcal S}_{d-1} \rightarrow (\RR_{+})^d$, $b^{\circ} =(b_{j}^{\circ})_{j \in \ES} : 
[0,T] \times {\mathcal S}_{d-1} \rightarrow {\mathbb R}^{d}$,
 $h : [0,T] \times {\mathcal S}_{d-1} \rightarrow \RR$ 
 and
 $\ell : {\mathcal S}_{d-1} \rightarrow \RR$ 
 are bounded and satisfy
\begin{equation}
\label{eq:zero:sum:intro}
\begin{split}
&\sum_{j \in \ES} \Bigl( \varphi(p_{j}) + b_{j}(t,p) + p_{j} b_j^{\circ}(t,p) \Bigr)= 0, \quad t \in [0,T], \quad p \in {\mathcal S}_{d-1}, 
\end{split}
\end{equation}
the unknown $u$ in 
\eqref{eq:pde:apriori:1:intro} being here real-valued (in words, it is an equation and not a system of equations). Obviously, the function $\varphi$ is 
the same as in 
\eqref{eq:varphi}
and the constraint 
\eqref{eq:zero:sum:intro}
is reminiscent of 
\eqref{eq:constraint:B:drift}
as it guarantees that the diffusion process associated with 
\eqref{eq:pde:apriori:1:intro} leaves the simplex invariant.

The
technical result below
is the 
core of our paper. 
It provides an a priori H\"older estimate to solutions of 
\eqref{eq:pde:apriori:1:intro} under the rather weak assumption that
$b$ and $b^\circ$ are merely continuous. This is made possible under the constraint that 
$\varphi$ points inward the simplex at the boundary in a sufficiently strong manner, whence our need 
to have
$\kappa$ large enough in Meta-Theorems 
 \ref{main:metathm} 
and \ref{main:metathm-PDE}. 
Differently from the latter two, 
this new statement 
must be regarded as being complete
except for the fact that the notion of derivatives in the PDE 
\eqref{eq:pde:apriori:1:intro} must be clarified.

\begin{theorem}
\label{main:holder}
Assume that $(b_{j})_{j \in \ES}$, $(b_{j}^\circ)_{j \in \ES}$ and $h$ are time-space continuous
and that $\ell$ is Lipschitz continuous (we let $\| \ell\|_{1,\infty}=\| \ell\|_{\infty}+ \sup_{p \not =q}
\vert \ell(p)-\ell(q)\vert/\vert p-q\vert$). Then,
there exists an exponent $\eta \in (0,1)$ such that, 
for any given $\delta \in (0,1/(4\sqrt{d}))$ and $\varepsilon \in (0,1)$, we can find a threshold $\kappa_{0} > 0$, depending 
on $\varepsilon$ and $(\| b_{j} \|_{\infty})_{j \in \ES}$,  such that, for any $\kappa \geq \kappa_{0}$, we can find another constant $C$, only depending on 
$\delta$, $\varepsilon$, $\kappa$, $(\| b_{j} \|_{\infty})_{j \in \ES}$, $(\| b_{j}^\circ \|_{\infty})_{j \in \ES}$, $\| h \|_{\infty}$, $\| \ell \|_{1,\infty}$ and $T$, such that any solution $u \in {\mathcal C}^{1,2}([0,T] \times \textrm{\rm Int}(\hat{\mathcal S}_{d-1}),\RR) \cap {\mathcal C}^{0}([0,T] \times {\mathcal S}_{d-1},\RR)$ of \eqref{eq:pde:apriori:1:intro}
satisfies
\begin{equation*}
\bigl\vert u(t,p) - u(s,q) \vert \leq C \bigl( \vert t-s  \vert^{\eta/2} + \vert p-q \vert^\eta \bigr), 
\quad (s,t) \in [0,T], \ (p,q) \in {\mathcal S}_{d-1}. 
\end{equation*}
Moreover,
$\| u \|_{\infty}$ is less than $\| \ell \|_{\infty} + T \| h \|_{\infty}$.
\end{theorem}
{The proof of the theorem is provided in Section \ref{sec:apriori}.}
\begin{remark} 
\label{rem:thm47}
{\ }
\begin{enumerate}
\item 
The
notation ${\mathcal C}^{1,2}$ in the statement is here understood in the usual sense: $u$ is required to be once continuously differentiable in $t$ and twice continuously differentiable in $p$ on the interior of the simplex, the notion of derivative being fully clarified in Subsections
\ref{subsub:derivatives:simplex} and \ref{subse:weak:solvability}.
As for the notation ${\mathcal C}^0$, it refers to functions that are continuous in $(t,p)$. 
\item 
We stress that we are not aware of any similar a priori H\"older estimate in the literature. There are some papers about the H\"older regularity of elliptic equations with degeneracies near the boundary, but they do not fit our framework (besides the obvious fact that the underlying equations are elliptic whist ours is parabolic): We refer for instance to \cite{GoulaouicShimakura} for a case with a  specific instance of drift that does not cover our needs.
We also emphasize that the H\"older estimate in Theorem \ref{main:holder} does not depend on the modulus of continuity of the coefficients 
$(b_{j})_{j \in \ES}$, $(b_{j}^\circ)_{j \in \ES}$ and $h$. In fact, we here assume the latter to be continuous for convenience only as it suffices for our own purposes. We believe that the result would remain true if 
$(b_{j})_{j \in \ES}$, $(b_{j}^\circ)_{j \in \ES}$ and $h$ were merely bounded and measurable; this would require to adapt accordingly the proof, which consists in mollifying the coefficients, as explained in the 
introduction 
of 
Section \ref{sec:apriori}, see in particular footnote \eqref{foo:mollif}.
\item
On another matter, it is worth noticing that we may trace back explicitly the dependence of 
$\kappa_{0}$ over $\varepsilon$. The key point in the proof is inequality \eqref{eq:condition:kappa0:epsilon}, which shows that 
$\kappa_{0}$ may be taken of the form $\kappa_{0}= \varepsilon^{-2} \kappa_{00}$, for 
$\kappa_{00}$ only depending on
$(\| b_{j} \|_{\infty})_{j \in \ES}$.
The parameter $\eta$ therein is a free parameter that is eventually chosen as $1/2$, see 
the discussion after 
Proposition \ref{prop:induction}.
\end{enumerate}
\end{remark}

\color{black}
\section{From Wright-Fischer spaces to complete statements}
 \label{sec:new:Kimura}
 
 {The purpose of this section is twofold. 
First, we introduce useful material
 for studying the PDEs  
 addressed in the article. By the way, we  
   provide a more rigorous formulation of the two equations 
      \eqref{eq:master:equation:introl} and 
\eqref{eq:pde:apriori:1:intro}, using appropriate systems of derivatives. 
This allows us to obtain 
complete and definitive versions of the two Meta-Theorems
\ref{main:metathm} and \ref{main:metathm-PDE}. In a second step of the section, we provide the proofs of 
Propositions 
\ref{thm:approximation:diffusion:2}, 
\ref{expfin} and
\ref{expfin:2}
that we used in the previous section to formulate our MFG. }

\subsection{Reformulation of the master equation}
\label{subsub:derivatives:simplex} 
{On the road towards a complete version of 
Meta-Theorem
\ref{main:metathm-PDE}, we first provide a more rigorous formulation of the master equation 
\eqref{eq:master:equation:introl} since the latter should be regarded as a PDE stated on a $(d-1)$-dimensional manifold}.

\subsubsection{Local coordinates}
The first way is to reformulate 
the master equation in so-called 
local coordinates, using 
the identification of 
${\mathcal S}_{d-1}$ 
with $\hat{\mathcal S}_{d-1}=\{ (x_{1},\cdots,x_{d-1}) \in (\RR_{+})^{d-1} : \sum_{i \in \llbracket d-1 \rrbracket} x_{i} \leq 1\}$  (see the introduction for the notation).  
For a real-valued function $h$ defined on ${\mathcal S}_{d-1}$, we may indeed define the function
\color{black}
\begin{equation}
\label{eq:check:x}
\begin{split}
\hat{h}^i \bigl(p^{-i}\bigr) &:= h (p)
=  h
\Bigl( p_{1},\cdots,p_{i-1},1- \sum_{k \not = i} p_{k},p_{i+1},\cdots,p_{d} \Bigr),
\\
&\textrm{\rm with}
\ p^{-i} = \Bigl(p_{1},\cdots,p_{i-1},p_{i+1},\cdots,p_{d}\Bigr),
\end{split}
\end{equation}
for $p \in {\mathcal S}_{d-1}$ and hence $p^{-i} \in \hat{\mathcal S}_{d-1}$, and for $i \in \ES$. 
We then say that $h$ is differentiable on ${\mathcal S}_{d-1}$ if 
$\hat{h}^i$ is differentiable on $\hat{\mathcal S}_{d-1}$ for some (and hence for any\footnote{In order to prove the differentiability of $\hat h^j$ for any $j \in \ES \setminus \{i\}$, it suffices to see that (say that $j>i$ to simplify) 
$\hat{h}^j(p^{-j})=\hat{h}^i(p_{1},\cdots,p_{i-1},p_{i+1},\cdots,p_{j-1},1-\sum_{k \not =j}p_{k},p_{j+1},\cdots)$.}) $i \in \ES$.
 In case when $h$
is defined on a ($d$-dimensional) neighborhood of ${\mathcal S}_{d-1}$, we then have 
$\partial_{p_{j}} \hat{h}^i(p^{-i}) = \partial_{p_{j}} h(p) - \partial_{p_{i}}h(p)$, for $j\not =i$.
We then end up with 
\begin{equation}\label{first:der}
\partial_{p_{j}} \hat{h}^i\bigl(p^{-i}\bigr) 
-
\partial_{p_{k}} \hat{h}^i\bigl(p^{-i}\bigr) 
= \partial_{p_{j}} h(p) - \partial_{p_{k}}h(p),
\end{equation}
for any $j,k \in \ES \setminus \{i\}$ and $p \in {\mathcal S}_{d-1}$. 
Similarly, the second order derivative may be written as 
\begin{equation}\label{second:der}
\partial^2_{p_{j} p_{k}} \hat{h}^i\bigl(p^{-i}\bigr) =  \partial_{p_{j} p_{k}}^2 h(p) - \partial_{p_{i}p_{j}}^2 h(p) - \partial^2_{p_{i} p_{k}}h(p) +\partial_{p_{i} p_{i}}^2 h(p),
\end{equation}
and the second order term in 
\eqref{eq:master:equation:introl} with $h=U^i$ (the reader should not make any confusion between $U^i$ and $\hat{h}^i$: $U^i$ is the $i$th coordinate of the solution to the master equation whilst $\hat{h}^i$ is the projection of $h$, whenever the latter is real-valued, onto a real-valued function on $\hat{\mathcal S}_{d-1}$) becomes
\begin{equation}
\label{eq:generator:local:coordinates}
\frac{\varepsilon^2}{2} \sum_{j,k \in \ES}\bigl(p_j  \delta_{jk}- p_{j} p_{k} \bigr) \partial^2_{p_{j} p_{k}} h (p)
= 
\frac{\varepsilon^2}{2} \sum_{  j,k  \not =i} ( p_j \delta_{jk}-  p_{j} p_{k}) \partial^2_{p_{j} p_{k}} \hat{h}^i\bigl(p^{-i}\bigr),
\end{equation}
for $p \in {\mathcal S}_{d-1}$.
In fact, according to the definition 
of the Wright--Fisher spaces in {the forthcoming}
\S
\ref{subse:space:Wright--Fisher}, we will consider functions that are just twice-differentiable on the interior $\textrm{\rm Int}(\hat{\mathcal S}_{d-1})$ of the simplex and for which the second-order derivatives may not extend to the boundary of the simplex.

\subsubsection{Intrinsic derivatives} 
Equivalently,
we may formulate the derivatives in \color{black} 
\eqref{eq:master:equation:introl} in terms of the intrinsic gradient on ${\mathcal S}_{d-1}$, regarded as a $(d-1)$-dimensional manifold. 
Indeed, whenever $h$ is defined on a neighborhood of ${\mathcal S}_{d-1}$, we may denote by $\nabla h = (\partial_{p_{1}}h,\dots, \partial_{p_{d}}h)$ the standard gradient in $\R^d$.
Identifying the tangent space ${\mathcal T}_{d-1}$ to the simplex at a given point $p \in {\mathcal S}_{d-1}$ with the 
orthogonal space to the $d$-dimensional vector $\bm{1}=(1,\cdots,1)$, 
the intrinsic gradient of $h$, seen as a function defined on the simplex, at $p$ identifies with the orthogonal projection of 
$\nabla h$ on ${\mathcal T}_{d-1}$. We denote it by  $\mathfrak{D} h := (\mathfrak{d}_1 h, \dots, \mathfrak{d}_d h)$, that is
\begin{equation*}
\mathfrak{D} h := \nabla h - \frac1d(\bm{1} \cdot \nabla h)\bm{1} \ ; \quad 
\mathfrak{d}_{p_i} h := \partial_{p_{i}} h - \frac1d \sum_{j \in \ES} \partial_{p_{j}} h, \quad i \in \ES, 
\end{equation*}
or, equivalently (which allows to define ${\mathfrak D} h$ when $h$ is just defined on ${\mathcal S}_{d-1}$),
\begin{equation*}
\mathfrak{d}_{p_i}h(p)= - \frac1d \sum_{j \not =i} \partial_{p_{j}} \hat{h}^i\bigl( p^{-i} \bigr), \quad p \in {\mathcal S}_{d-1}.
\end{equation*}
Of course we have $\sum_j \mathfrak{d}_{p_j}h = \bm{1} \cdot\mathfrak{D} h =0$ by construction\footnote{When 
$h$ is just defined on the simplex, the proof is slightly less obvious, but it may be achieved by checking that 
$\partial_{p_{j}} \hat{h}^i(p^{-i}) 
= - \partial_{p_{i}} \hat{h}^j(p^{-j})$, for $i \not =j$. And then,
$\sum_{i \in \ES} \sum_{j \not =i}
\partial_{p_{j}} \hat{h}^i(p^{-i})  = 
- \sum_{i \in \ES} \sum_{j \not =i}
\partial_{p_{i}} \hat{h}^j(p^{-j})$. By Fubini's theorem, the latter is also equal to 
$- \sum_{j \in \ES} \sum_{i \not =j}
\partial_{p_{j}} \hat{h}^i(p^{-i})=
- \sum_{i \in \ES} \sum_{j \not =i}
\partial_{p_{i}} \hat{h}^j(p^{-j})$, from which we deduce that it is indeed equal to $0$.}. And the following holds true
\begin{equation*}
\mathfrak{d}_{p_i} h(p) - \mathfrak{d}_{p_j}h(p) =  \partial_{p_{i}} h(p) - \partial_{p_{j}}h(p), 
\end{equation*}
for $i,j \in \ES$ and $p \in {\mathcal S}_{d-1}$. 
As for the second order derivatives, we have\footnote{The second line follows from the identity 
$\partial_{p_{l}} \hat{h}^{j}(p^{-j})
=  \partial_{p_{l}} \hat{h}^{i}(p^{-i}) - 
 \partial_{p_{j}} \hat{h}^{i}(p^{-i})$ if $l \not =i$ and 
$\partial_{p_{l}} \hat{h}^{j}(p^{-j})
= - 
 \partial_{p_{j}} \hat{h}^{i}(p^{-i})$ if $l=i$, which, in turn, implies  
\begin{equation*}
\begin{split}
{\mathfrak d}_{p_i}({\mathfrak d}_{p_j} h) &=
-\frac1d \sum_{k \not =i} \partial_{p_{k}} \widehat{{\mathfrak d}_{p_j} h}^i\bigl(p^{-i}\bigr)
=
-\frac1d \sum_{k \not =i} \partial_{p_{k}} \biggl( - \frac1d \sum_{l \not = i,j} \Bigl[ \partial_{p_{l}} \hat{h}^i\bigl(p^{-i}\bigr)
- \partial_{p_{j}}\hat{h}^i\bigl(p^{-i}\bigr) \Bigr] + \frac1d  \partial_{p_{j}} \hat{h}^i\bigl(p^{-i}\bigr) \biggr)
\\
&=
-\frac1d \sum_{k \not =i} \partial_{p_{k}} \biggl( - \frac1d \sum_{l \not = i }
 \partial_{p_{l}} \hat{h}^i\bigl(p^{-i}\bigr)
 +   \partial_{p_{j}} \hat{h}^i\bigl(p^{-i}\bigr) \biggr).
\end{split}
\end{equation*}
}
\begin{equation}
\label{eq:mathfrakd:partiald}
\begin{split}
\fd^2_{p_ip_j}h(p) &= \partial_{p_{i} p_{j}}^2 h(p) - \frac1d \sum_{k \in \ES} \bigl(\partial_{p_{i} p_{k}}^2 h(p) + \partial_{p_{j} p_{k}}^2 h(p)\bigr) + \frac{1}{d^2}  \sum_{k,l \in \ES} \partial_{p_{k} p_{l}}^2 h(p)
\\
&= \frac1{d^2} \sum_{k,l \not = i} \partial^2_{p_{k} p_{l}} \hat{h}^i\bigl(p^{-i} \bigr) 
- \frac1{d} \sum_{k \not = i } \partial^2_{p_{k} p_{j}} \hat{h}^i\bigl(p^{-i} \bigr),
\end{split}
\end{equation}
and then the second order term in \eqref{eq:master:equation:introl} 
(with $h=U^i$, see \eqref{eq:generator:local:coordinates}) 
becomes\footnote{The result may be also proved by combining 
\eqref{eq:generator:local:coordinates}
and 
\eqref{eq:mathfrakd:partiald}, hence bypassing the derivatives of $h$ themselves. Indeed, 
\begin{equation*}
\begin{split}
\sum_{  j,k  \in \ES} ( p_j \delta_{jk}-  p_{j} p_{k}) \fd^2_{p_jp_k} h
&= 
- 
\frac1d 
\sum_{  j,k  \in \ES} ( p_j \delta_{jk}-  p_{j} p_{k}) \sum_{l \not =k} \partial^2_{p_{l} p_{j}} \hat{h}^{k}(p^{-k})
\\
&= - 
\frac1d 
\sum_{  j,k  \in \ES} ( p_j \delta_{jk}-  p_{j} p_{k}) \Bigl( 
\sum_{l \not =k,i} \bigl[ 
\partial^2_{p_{l}p_{j}} \hat{h}^{i}(p^{-i})
-
\partial^2_{p_{k}p_{j}} \hat{h}^{i}(p^{-i})\bigr]
- 
\partial^2_{p_{k}p_{j}} \hat{h}^{i}(p^{-i}) \Bigr)
\\
&= - 
\frac1d 
\sum_{  j,k  \in \ES} ( p_j \delta_{jk}-  p_{j} p_{k}) \Bigl( 
\sum_{l \not =i} 
\partial^2_{p_{l}p_{j}} \hat{h}^{i}(p^{-i})
- 
d \cdot \partial^2_{p_{k}p_{j}} \hat{h}^{i}(p^{-i}) \Bigr)
\\
&=  \sum_{  j,k  \in \ES} ( p_j \delta_{jk}-  p_{j} p_{k}) \partial^2_{p_{k}p_{j}} \hat{h}^{i}(p^{-i}).
\end{split}
\end{equation*}

}
\[\sum_{j,k \in \ES}(p_j \delta_{jk}-p_{j} p_{k}) \partial_{p_{j} p_{k}}^2 h(p) = \sum_{ j,k \in \ES}(p_j \delta_{jk}-p_{j} p_{k})\fd^2_{p_jp_k} h,\]
since $\sum_{k \in \ES} (p_j \delta_{jk}-p_{j} p_{k}) =0$ for any $j \in \ES$.
 
\subsubsection{Two equivalent formulations} In the end, the master equation \color{black}
\eqref{eq:master:equation:introl} may be written in two equivalent forms. The first one may be written in terms of the derivatives in \eqref{first:der}--\eqref{second:der}:
\begin{equation}
\label{eq:master:equation:local}
\begin{split}
&\partial_t U^i(t,p)  -\frac{1}{2}\sum_{j=1}^d\bigl(U^i(t,p)-U^j(t,p)\bigr)^2_++ f^i(t,p) 
+
\sum_{j \in \ES} \varphi(p_{j})
\bigl[ U^j(t,p) - U^i(t,p) \bigr]
\\
&\quad
+ 
\sum_{j\not = i} \sum_{k \in \ES} \Bigl( 
p_{k}\bigl[ \varphi(p_{j}) + \bigl(U^k(t,p)- U^j(t,p)\bigr)_+ \bigr]
\\
&\hspace{150pt}-
p_{j}\bigl[ \varphi(p_{k}) + \bigl(U^j(t,p)- U^k(t,p)\bigr)_+ \bigr]
\Bigr)
 \partial_{p_{j}} {\widehat {\{ U^i\}}}{}^{i}\bigl(t,p^{-i}\bigr)
\\
&\quad - \varepsilon^2\sum_{ j\neq i} p_{j}  \partial_{p_{j}}{\widehat {\{ U^i\}}}{}^{i}\bigl(t,p^{-i}\bigr)
+\frac{\varepsilon^2}{2} \sum_{ j,k \not =i}(p_j \delta_{jk}-p_{j} p_{k}) \partial^2_{p_{j} p_{k}} {\widehat {\{ U^i\}}}{}^{i}\bigl(t,p^{-i}\bigr) =0,
\\
&U^i(T,p)= g^i(p),
\end{split}
\end{equation}
for $(t,p) \in [0,T] \times \textrm{\rm Int}(\hat{\mathcal S}_{d-1})$.
Above, the function 
$\widehat {\{ U^i\}}{}^{i}$ is defined on $\hat{\mathcal S}_{d-1}$ in the same way as before, namely
$\widehat {\{ U^i\}}{}^{i}(t,p^{-i})= U^i(t, p_{1},\cdots,p_{i-1},1-\sum_{j\ne i}p_j, p_{i+1},\cdots,p_{d})$, for $p \in {\mathcal S}_{d-1}$. For sure, we could rewrite the equation for $U^i$ in terms of the variable $p^{-j}$ (instead of $p^{-i}$) for another index $j \not =i$, but this would be of little interest for us. In fact, we will make greater use of a
second form 
of \eqref{eq:master:equation:introl} that may be written in terms of the intrinsic derivative:
\begin{equation}
\label{eq:master:equation:intrinsic}
\begin{split}
&\partial_t U^i(t,p)  -\frac{1}{2}\sum_{j=1}^d\bigl(U^i(t,p)-U^j(t,p)\bigr)^2_+ +f^i(t,p) 
+
\sum_{j \in \ES} \varphi(p_{j})
\bigl[ U^j(t,p) -   U^i(t,p) \bigr]
\\
&\qquad + 
\sum_{j,k \in \ES} p_{k}\bigl[ \varphi(p_{j}) + \bigl(U^k(t,p-U^j(t,p)\bigr)_+ \bigr] \left( \fd_{p_{j}} U^i(t,p) - \fd_{p_{k}} U^i(t,p)\right) \\
&\qquad+ \varepsilon^2\sum_{j  \neq i } p_{j} \left( \fd_{p_{i}} U^i(t,p) - \fd_{p_j} U^i(t,p)\right)
+\frac{\varepsilon^2}{2} \sum_{j,k \in \ES}(p_j \delta_{jk}-p_{j} p_{k}) \fd^2_{p_{j} p_{k}} U^i(t,p) =0,\\
&U^i(T,p)= g^i(p),
\end{split}
\end{equation}
for $(t,p) \in [0,T] \times \textrm{\rm Int}(\hat{\mathcal S}_{d-1})$. 

\subsection{Kimura operators}
\label{subse:weak:solvability}
We {now clarify} the choice of the functional spaces for $f$ and $g$ in Meta-Theorems \ref{main:metathm} and 
\ref{main:metathm-PDE} and for $U$ in 
Meta-Theorem \ref{main:metathm-PDE}. Basically, 
we take those spaces from a recent Schauder like theory due to 
Epstein and Mazzeo \cite{EpsteinMazzeo}
for what we called Kimura operators, the latter being operators of the very same structure as the second order generator
of
\eqref{eq:weak:sde:bis}, which we will denote by 
$({\mathcal L}_{t})_{0 \le t \le T}$.



\subsubsection{{Normal form}}
\label{subse:Wright--Fisher:model}
Following 
\eqref{eq:bracket:weak:sde},
we get (at least informally) 
that,  
for any twice differentiable real-valued function $h$ on $\RR^d$,
\begin{equation}
\label{eq:generator:kimura}
\begin{split}
{\mathcal L}_{t}  h(p)  &= 
\sum_{i \in \ES} a_{i}(t,p) \partial_{p_{i}} h(p)
+ 
 \frac{\varepsilon^2}{2} \sum_{i,j \in \ES}  \bigl( p_{i} \delta_{i,j} - {p_{i} p_{j}} \bigr)
\partial^2_{p_{i} p_{j}} h(p).
\end{split} 
\end{equation}
{Obviously, we here recover the same second-order operator  as in the master equation 
\eqref{eq:master:equation:introl}, but as we have just explained, this writing is rather abusive.
For sure, we have introduced in the previous subsection convenient systems of coordinates and of derivatives that permit to get 
a more rigorous form. However we need here a more refined formulation of 
\eqref{eq:generator:kimura} that brings out the degeneracies of the operator and, most of all, that fits the framework of Definition 2.2.1 for Kimura operators in the aforementioned monograph \cite{EpsteinMazzeo}.} In this regard, it is worth noticing that the second-order term in ${\mathcal L}_{t}$ 
is {obviously degenerate when it is acting on functions defined on ${\mathbb R}^d$. This follows from the fact that the matrix 
$(p_{i} \delta_{i,j} - p_{i} p_{j})_{i,j \in \ES}$
has $(1,\cdots,1)$ in its kernel whenever $(p_{1},\cdots,p_{d})$ is in ${\mathcal S}_{d-1}$; this is somehow the price to pay for forcing the solutions to  
\eqref{eq:weak:sde:bis} to stay within  ${\mathcal S}_{d-1}$. More subtly, we are concerned in the rest of this section with the degeneracies of 
${\mathcal L}_{t}$ when it is truly acting on functions defined on the simplex and hence when it is expressed in local coordinates.}

In case when the drift $a$ in 
\eqref{eq:generator:kimura} is zero, 
${\mathcal L}_{t}$ becomes time independent and coincides with the generator of the  $d$-dimensional Wright--Fisher model. 
In case when $a$ is non-zero but is time-independent and satisfies
\begin{equation}
\label{eq:bi:condition:sign}
a_{i}(p) \geq 0 \quad \textrm{if} \ p_{i} =0,
\end{equation}
(which means that $a$ points inward at points $p$ that belong to the boundary of $\hat {\mathcal S}_{d-1}$), the operator 
${\mathcal L}_{t}$ itself becomes a time-homogeneous Kimura diffusion operator. Below, we make an intense use of the recent monograph of
Epstein and Mazzeo \cite{EpsteinMazzeo} on those types of operators, see also 
\cite{ChenStroock,Kimura,Shimakura} for earlier results.
The
key feature is that, under the identification of 
${\mathcal S}_{d-1}$
with $\hat{\mathcal S}_{d-1}$,
we may regard the simplex 
as a $d-1$ dimensional \textit{manifold with corners}, the corners being obtained by intersecting at most $d-1$ of the hyperplanes 
$\{x \in \RR^{d-1} : x_{1}=0\}$, $\dots$, 
$\{x \in \RR^{d-1} : x_{d-1}=0\}$,
$\{x \in \RR^{d-1} : x_{1}+\cdots+x_{d-1}=1\}$ with $\hat{\mathcal S}_{d-1}$ (we then call the codimension of the corner the number of hyperplanes showing up in the intersection). 
Accordingly {and consistently with the analysis performed in Subsection \ref{subsub:derivatives:simplex}}, we can rewrite 
\eqref{eq:generator:kimura}
as an
operator acting on functions from $\hat{\mathcal S}_{d-1}$ to ${\mathbb R}$,  
the resulting operator being then a Kimura diffusion  
operator on $\hat{\mathcal S}_{d-1}$.
{It suffices to}
reformulate 
\eqref{eq:generator:kimura}
in terms of the sole $d-1$ first coordinates $(p_{1},\cdots,p_{d-1})$
or, more generally, in terms of $(p_{i})_{i \in \ES \setminus \{ l\}}$ for any 
given coordinate $l \in \ES$. 
Somehow, choosing the coordinate $l$ amounts to choosing a system of local coordinates and, as we explain below, 
the choice of $l$ is mostly dictated by the position of $(p_{1},\cdots,p_{d})$ inside the simplex. 
For instance, 
whenever all the entries of $p=(p_{1},\cdots,p_{d})$ are positive, meaning that 
$(p_{1},\cdots,p_{d-1})$ belongs to the interior 
of $\hat{\mathcal S}_{d-1}$, the choice of $l$ does not really matter and we may work, for convenience, 
with $l=d$. 
We then rewrite the generator ${\mathcal L}_{t}$, as given by \eqref{eq:generator:kimura},
in the form
\begin{equation}
\label{eq:reduced:L:d-1}
\hat {\mathcal L}_{t}  \hat h(x)  
=   \sum_{i=1}^{d-1} \hat a_{i}(t,x) \partial_{x_{i}} \hat h(x)
+ \frac{\varepsilon^2}{2} \sum_{i,j=1}^{d-1}  \bigl(x_{i} \delta_{i,j} - {x_{i} x_{j}} \bigr)
\partial^2_{x_{i} x_{j}} \hat h(x),
\end{equation}
where now $x \in \hat{\mathcal S}_{d-1}$, $\hat{h}$ is a smooth function 
on $\RR^{d-1}$ and $\hat{a}_{i}(t,x) = a_{i}(t,\check{x})$, with $\check{x}=(x_{1},\cdots,x_{d-1},1-x_{1}-\cdots-x_{d-1})$. 
{Following the computations of Subsection \ref{subsub:derivatives:simplex}, the connection}
between 
\eqref{eq:generator:kimura}
and 
\eqref{eq:reduced:L:d-1} is that 
$\hat {\mathcal L}_{t}  \hat h(x)
={\mathcal L}_{t} \{ h(\check x)\}$, whenever 
$\hat h$ is defined 
from $h$ through the identity 
$\hat h(x)=h(\check x)$.    

Importantly, 
we then notice that, in the interior of $\hat{\mathcal S}_{d-1}$, $\hat {\mathcal L}_t$ is elliptic. Indeed, for any 
$(x_{1},\cdots,x_{d-1}) \in \hat{\mathcal S}_{d-1}$ 
and
$(\xi_{1},\cdots,\xi_{d-1}) \in \RR^{d-1}$,
\begin{equation}
\label{eq:ellipticity}
\begin{split}
\sum_{i,j=1}^{d-1} \xi_{i} \bigl(x_{i} \delta_{i,j} - {x_{i} x_{j}} \bigr) \xi_{j}
= \sum_{i=1}^{d-1} \xi_{i}^2 x_{i} - \biggl( \sum_{i=1}^{d-1} \xi_{i} x_{i} \biggr)^2
&\geq \sum_{i=1}^{d-1} \xi_{i}^2 x_{i} - \sum_{j=1}^{d-1} x_{j} \sum_{i=1}^{d-1} \xi_{i}^2 x_{i}
\\
&= \sum_{i=1}^{d-1} \xi_{i}^2 x_{i}\Bigl( 1 -
\sum_{j=1}^{d-1} x_{j} \Bigr), 
\end{split}
\end{equation}
which suffices to prove ellipticity whenever $x_{1},\cdots,x_{d-1}>0$ and 
$\sum_{j=1}^{d-1} x_{j} <1$.

Take now a corner of codimension $l \in \{1,\cdots,d-1\}$. 
 If the $l$ hyperplanes entering the definition of the corner are of the form
 ${\mathcal H}_{i_{1}}=\{x \in \RR^{d-1} : x_{i_{1}} = 0\}$, $\dots$, ${\mathcal H}_{i_{l}}= \{x \in \RR^{d-1} : x_{i_{l}}=0\}$, 
 for $1 \leq i_{1} < i_{2} < \cdots <i_{l} \leq d-1$, then we can rewrite 
\eqref{eq:reduced:L:d-1} in the form
 \begin{equation}
 \label{eq:canonical:form:kimura}
\begin{split}
\hat {\mathcal L}_{t}  \hat h(x)  
&= \frac{\varepsilon^2}{2} \sum_{i \in I} (1-x_{i})
x_{i}
\partial^2_{x_{i}} \hat h(x)
+ \frac{\varepsilon^2}{2} \sum_{j,k \not \in I}
x_{j} (\delta_{j,k}-x_{k}) 
\partial^2_{x_{j} x_{k}} \hat h(x)
+   \sum_{i \in I} 
\hat a_{i}(t,x) \partial_{x_{i}} \hat h(x)
\\
&\hspace{15pt}
 - \frac{\varepsilon^2}{2} \sum_{i,j \in I : i \not =j}   {x_{i} x_{j}} 
\partial^2_{x_{i} x_{j}} \hat h(x)
 - \varepsilon^2 \sum_{i \in I,j \not \in I} x_{i} x_{j}
\partial^2_{x_{i} x_{j}} \hat h(x)
 +   \sum_{i \not \in I} 
\hat a_{i}(t,x) \partial_{x_{i}} \hat h(x),
\end{split}
\end{equation}
with $I=\{i_{1},\cdots,i_{l}\}$. Here, we observe from 
\eqref{eq:bi:condition:sign} that $\hat{a}_{i}(t,x) \geq 0$ if 
$i \in I$ and $x \in \cap_{j \in I} {\mathcal H}_{j}$. Also, as long as 
$x$ belongs to 
$\cap_{j \in I} {\mathcal H}_{j}$
but
$(x_{j})_{j \not \in I}$ and $1- \sum_{j \not \in I} x_{j}$ remain positive (which is necessarily true in the relative interior of 
$\cap_{j \in I} {\mathcal H}_{j} \cap \hat{\mathcal S}_{d-1}$), then, by the same argument as in \eqref{eq:ellipticity}, the matrix 
$(x_{j} (\delta_{j,k}-x_{k}))_{j,k \not \in I}$ is non degenerate. Hence, up to the intensity factor $\varepsilon$, the above decomposition fits (2.4) in \cite[Definition 2.2.1]{EpsteinMazzeo}, which is of crucial interest for us\footnote{\label{foo:normal:form}The reader may also notice that, in \cite{EpsteinMazzeo}, the operator ${\mathcal L}_{t}$ is passed in a normal form, see Proposition 2.2.3 therein which guarantees that such a normal form indeed exists.
Here the change of variable to get the normal form may be easily {clarified} by adapting the 1d case accordingly, see \cite{EpsteinMazzeo:1d}: It suffices to change 
$x_{i}$ into the new coordinate $\arcsin^2(\sqrt{x_{i}})$. In fact, the normal form in \cite{EpsteinMazzeo} plays a key role in the definition of 
the Wright-Fischer H\"older spaces that we recall in the next paragraph. }. 

Assume now that the corner of codimension $l$ is given by the intersection of the hyperplanes 
$\{x \in \RR^{d-1} : x_{i_{1}} = 0\}$, $\dots$, $\{ x \in \RR^{d-1} : x_{i_{l-1}}=0\}$
and $\{x \in \RR^{d-1} : x_{1} + \cdots + x_{d-1} = 1\}$. In order to recover the same form as in 
\cite[(2.4)]{EpsteinMazzeo} (or equivalently in \eqref{eq:canonical:form:kimura}), we 
perform the following change of variables:
consider 
$(y_{1},\cdots,y_{d-1}):=(p_{1},\cdots,p_{i_{l}-1},p_{i_{l}+1},\cdots,p_{d})$
as 
a new system of local coordinates, for some given index $i_{l} \in \{i_{l-1}+1,\cdots,d-1\}$ (which is indeed possible if 
$i_{l-1}<d-1$; if not, $i_{l}$ must be chosen as the largest index that is different from $i_{1},\cdots,i_{l-1},d$ and hence is lower than $i_{l-1}$, which asks for reordering the coordinates in the change of variables).
For a test function $h$ as before, we then expand ${\mathcal L}_{t} [h(p_{1},\cdots,p_{i_{l}-1},
1-\sum_{j \not= i_{l}}
p_{j},p_{i_{l-1}+1},\cdots,p_{d})]$ as before and check that 
we recover the same structure as in \cite[(2.4)]{EpsteinMazzeo}, but in the new 
coordinates. This demonstrates how to check the setting of 
\cite[Definition 2.2.1]{EpsteinMazzeo}.

\subsubsection{Wright--Fisher spaces}
\label{subse:space:Wright--Fisher}
The rationale for double-checking 
\cite[Definition 2.2.1]{EpsteinMazzeo}
is that we want to use next the 
Schauder theory developed in 
\cite{EpsteinMazzeo}
for Kimura operators. 
This prompts us to introduce here the functional spaces that underpin 
the corresponding Schauder estimates.

For a point $x^{0} \in \hat{\mathcal S}_{d-1}$ in the relative interior of a corner ${\mathcal C}$ of $\hat{\mathcal S}_{d-1}$ of codimension $l \in \{0,\cdots,d-1\}$ (if 
$l=0$, then $x^{0}$ is in the interior of $\hat{\mathcal S}_{d-1}$), we may consider a 
new system of coordinates $(y_{1},\cdots,y_{d-1})$ (of the same type as in the previous paragraph) such that ${\mathcal C} = \{y \in \hat{\mathcal S}_{d-1} : y_{i_{1}}= \cdots = 
y_{i_{l}}=0\}$, for $1 \leq i_{1} < \cdots < i_{l}$. 
Letting 
$I:=\{i_{1},\cdots,i_{l}\}$ and 
denoting by 
$(y^0_{1},\cdots,y^0_{d-1})$ the coordinates of $x^0$ in the new system (for sure $y^0_{i_{j}}=0$ for 
$j=1,\cdots,l$), we may find a $\delta^0 >0$ such that:
\begin{enumerate}
\item 
the closure 
$\overline{\mathcal U}(\delta^0,x^0)$ of ${\mathcal U}(\delta^0,x^0):= \{ y \in (\RR_{+})^{d-1} : \sup_{i=1,\cdots,d-1} \vert y_{i} - y^0_{i} \vert < 
\delta^0 \}$ is included in $\hat{\mathcal S}_{d-1}$,
\item for $y$ in $\overline{\mathcal U}(\delta^0,x^0)$, for $j \not \in I$, $y_{j}>0$,
\item for $y$ in $\overline{\mathcal U}(\delta^0,x^0)$, for $y_{1}+\cdots+y_{d-1}<1-\delta^0$. 
\end{enumerate}
A function $\hat{h}$ defined on $\overline{\mathcal U}(\delta^0,x^0)$ 
is then said to belong to ${\mathscr C}^{\gamma}_{\textrm{\rm WF}}(\overline{\mathcal U}(\delta^0,x^0))$, for some $\gamma \in (0,1)$, if, in the new system of coordinates, $\hat{h}$ is H\"older continuous on $\overline{\mathcal U}(\delta^0,x^0)$ with respect to the distance\footnote{\label{foo:subtlety}
There is a subtlety here: In fact, the distance used in \cite[(1.32), (5.42)]{EpsteinMazzeo} is defined in terms of the coordinates that are used in the normal form of the operator, see footnote \ref{foo:normal:form}. So rigorously, we should not use the variables 
$(y_{1},\cdots,y_{d-1})$ but the variables 
$(\textrm{arcsin}^2(\sqrt{y_{1}}),\cdots,\textrm{arcsin}^2(\sqrt{y_{d-1}}))$
in the definition of the distance. Fortunately, since we have the condition $y_{1}+\cdots+y_{d-1} < 1-\delta^0$, the change of variable $y_{i} \mapsto 
\textrm{arcsin}^2(\sqrt{y_{i}})$ is a smooth diffeomorphism, from which we deduce that the distance 
\eqref{eq:distance:d}
is equivalent to the same distance but with the new variables. And, in fact, once we have made the change of variables, there is another subtlety:
The reader may indeed notice that the distance defined in 
\eqref{eq:distance:d} does not match the distance defined in \cite[(1.32), (5.42)]{EpsteinMazzeo} since, for $j \not \in I$, 
we should consider $\vert y_{j} - y_{j}' \vert$ instead of $\vert \sqrt{y_{j}} - \sqrt{y_{j}'}\vert$. Anyhow, since 
$y_{j}$ and $y_{j}'$ are here required to be away from $0$, the distance used in \cite{EpsteinMazzeo} is equivalent to ours.}
\begin{equation}
\label{eq:distance:d}
d(y,y') := \sum_{i =1}^{d-1}  \bigl\vert \sqrt{ y_{i}} -\sqrt{ y_{i}'} \bigr\vert. 
\end{equation}
We then let
\begin{equation*}
\bigl\| \hat h \bigr\|_{\gamma;{\mathcal U}(\delta^0,x^0)}
:= \sup_{y \in \overline{\mathcal U}(\delta^0,x^0)}
\bigl\vert \hat h(y) \bigr\vert
+ \sup_{y,y' \in \overline{\mathcal U}(\delta^0,x^0)}
\frac{\vert \hat h(y) - \hat h(y') \vert}{d(y,y')^{\gamma}}.
\end{equation*}
Following \cite[Lemma 5.2.5 and Definition 10.1.1]{EpsteinMazzeo}, 
we say that a function $\hat h$ defined on ${\mathcal U}(\delta^0,x^0)$
belongs to ${\mathscr C}^{2+\gamma}_{\textrm{\rm WF}}({\mathcal U}(\delta^0,x^0))$ if, in the new system of coordinates,
\begin{enumerate}
\item $\hat{h}$ is continuously differentiable on
${\mathcal U}(\delta^0,x^0)$ and $\hat{h}$ and its derivatives extend continuously to 
$\overline{\mathcal U}(\delta^0,x^0)$ and the resulting extensions 
 belong to 
${\mathscr C}^{\gamma}_{\textrm{\rm WF}}(\overline{\mathcal U}(\delta^0,x^0))$;
\item The function $\hat{h}$ is twice continuously differentiable on 
${\mathcal U}_{+}(\delta^0,x^0) = {\mathcal U}(\delta^0,x^0) \cap \{ (y_{1},\cdots,y_{d-1}) \in (\RR_{+})^{d-1} : \forall i \in I, y_{i} >0
\}$. Moreover\footnote{Similar to footnote
\ref{foo:subtlety}, the reader should observe that, in 
\cite[Lemma 5.2.5 and Definition 10.1.1]{EpsteinMazzeo}, the two limits in \eqref{eq:derivatives:boundary:WF} are in fact regarded in
terms of the coordinates used in the normal form of the operator. To make it clear, we should here require
$\lim_{\min(z_{i},z_{j}) \rightarrow 0_{+}}
\sqrt{z_{i} z_{j}} \partial^2_{z_{i}z_{j}}
\hat{\ell}(z)=0$ with $\hat{\ell}(z_{1},\cdots,z_{d-1}) = \hat{h}(\sin^2(\sqrt{z_{1}}),\cdots,\sin^2(\sqrt{z_{d-1}}))$  (and similarly for the second limit). It is an easy exercise to check that \eqref{eq:derivatives:boundary:WF} would then follow.}, for any $i,j \in I$ and any $k,l \not \in I$,
\begin{equation}
\label{eq:derivatives:boundary:WF}
\begin{split}
&\lim_{\min(y_{i},y_{j}) \rightarrow 0_{+}}
\sqrt{y_{i} y_{j}} \partial^2_{y_{i}y_{j}}
\hat{h}(y) = 0,
\quad \lim_{y_{i} \rightarrow 0_{+}}
\sqrt{y_{i}} \partial^2_{y_{i}y_{k}}
\hat{h}(y) = 0,
\end{split}
\end{equation}
and the functions
$y \mapsto  
\sqrt{y_{i} y_{j}} \partial^2_{y_{i}y_{j}}
\hat{h}(y)$, 
$y \mapsto  
\sqrt{y_{i}} \partial^2_{y_{i}y_{k}}
\hat{h}(y)$
and
$y \mapsto  
\partial^2_{y_{k}y_{l}}
\hat{h}(y)$
belong to ${\mathscr C}^{\gamma}_{\textrm{\rm WF}}(\overline{\mathcal U}(\delta^0,x^0))$ (meaning in particular that they can be extended by continuity to 
$\overline{\mathcal U}(\delta^0,x^0)$).
\end{enumerate} 
We then call 
\begin{equation*}
\begin{split}
\| \hat{h} \|_{2+\gamma;{\mathcal U}(\delta^0,x^0)}
&:= 
\| \hat{h} \|_{\gamma;{\mathcal U}(\delta^0,x^0)}
+
\sum_{i=1}^{d-1} 
\| \partial_{y_{i}} \hat{h} \|_{\gamma;{\mathcal U}(\delta^0,x^0)}
 +
\sum_{i,j \in I}
\|
\sqrt{y_{i} y_{j}} \partial^2_{y_{i}y_{j}}
\hat{h}
 \|_{\gamma;{\mathcal U}(\delta^0,x^0)}
 \\
&\hspace{15pt} +
 \sum_{k,l \not \in I}
\|
 \partial^2_{y_{k}y_{l}}
\hat{h}
 \|_{\gamma;{\mathcal U}(\delta^0,x^0)}
 +
 \sum_{i \in I}
 \sum_{k  \not \in I}
\|
\sqrt{y_{i}} \partial^2_{y_{i}y_{k}}
\hat{h}
 \|_{\gamma;{\mathcal U}(\delta^0,x^0)}, 
\end{split}
\end{equation*}
where 
$\sqrt{y_{i} y_{j}} \partial^2_{y_{i}y_{j}}
\hat{h}$ is a shorten notation for 
$y \mapsto 
\sqrt{y_{i} y_{j}} \partial^2_{y_{i}y_{j}}
\hat{h}(y)$  (and similarly for the others). 
For a given finite covering $\cup_{i=1}^K {\mathcal U}(\delta^0,x^{0,i})$ of $\hat {\mathcal S}_{d-1}$, a function 
$\hat{h}$ (or equivalently the associated function $h$ defined on ${\mathcal S}_{d-1}$) is said to be in ${\mathscr C}^{2+\gamma}_{\textrm{\rm WF}}({\mathcal S}_{d-1})$ if 
$\hat{h}$ belongs to each ${\mathscr C}^{2+\gamma}_{\textrm{\rm WF}}({\mathcal U}(\delta^0,x^{0,i}))$.
We then let
\begin{equation*}
\| \hat{h} \|_{2+\gamma} := \sum_{i=1}^K \| \hat{h} \|_{2+\gamma;{\mathcal U}(\delta^0,x^{0,i})}. 
\end{equation*} 
We refer to 
 \cite[Chapter 10]{EpsteinMazzeo} for more details. 
 
 A similar definition holds for the space ${\mathscr C}_{\textrm{\rm WF}}^{1+\gamma/2,2+\gamma}([0,T] \times {\mathcal S}_{d-1})$ of functions 
 that are once continuously differentiable in time and twice continuously differentiable in space, with derivatives that are locally $\gamma$-H\"older continuous with respect to the time-space distance (in the local system of coordinates)
\begin{equation}
\label{eq:distance:D}
D\bigl((t,y),(t',y')\bigr) := \vert t-t'\vert^{1/2} + d(y,y').
\end{equation}
To make it clear, a function $\hat{h}$ defined on $[0,T] \times \overline{\mathcal U}(\delta^0,x^0)$ 
is then said to belong to ${\mathscr C}_{\textrm{\rm WF}}^{\gamma/2,\gamma}([0,T] \times \overline{\mathcal U}(\delta^0,x^0))$, for some $\gamma \in (0,1)$, if, in the new system of coordinates, $\hat{h}$ is H\"older continuous on $[0,T] \times \overline{\mathcal U}(\delta^0,x^0)$ with respect to the distance
$D$.
We then let
\begin{equation*}
\bigl\| \hat h \bigr\|_{\gamma/2,\gamma;[0,T] \times {\mathcal U}(\delta^0,x^0)}
:= \sup_{(t,y)  \in [0,T] \times \overline{\mathcal U}(\delta^0,x^0)}
\bigl\vert \hat h(t,y) \bigr\vert
+ \sup_{t,t' \in [0,T],\ y,y' \in \overline{\mathcal U}(\delta^0,x^0)}
\frac{\vert \hat h(t,y) - \hat h(t',y') \vert}{D((t,y),(t',y'))^{\gamma}}.
\end{equation*}
Following \cite[Lemma 5.2.7]{EpsteinMazzeo}, we
say that  
a function $\hat h$ defined on ${\mathcal U}(\delta^0,x^0)$
belongs to the space ${\mathscr C}_{\textrm{\rm WF}}^{1+\gamma/2,2+\gamma}([0,T] \times {\mathcal U}(\delta^0,x^0))$ if, in the new system of coordinates,
\begin{enumerate}
\item $\hat{h}$ is continuously differentiable on
$[0,T] \times {\mathcal U}(\delta^0,x^0)$; $\hat{h}$ and its time and space derivatives extend continuously to 
$[0,T] \times \overline{\mathcal U}(\delta^0,x^0)$ and the resulting extensions 
 belong to 
${\mathscr C}^{\gamma/2,\gamma}_{\textrm{\rm WF}}([0,T] \times \overline{\mathcal U}(\delta^0,x^0))$;
\item The function $\hat{h}$ is twice continuously differentiable in space on 
$[0,T] \times {\mathcal U}_{+}(\delta^0,x^0)$. Moreover, for any $i,j \in I$ and any $k,l \not \in I$,
\begin{equation}
\label{eq:rate:2nd:order}
\begin{split}
&\lim_{\min(y_{i},y_{j}) \rightarrow 0_{+}}
\sqrt{y_{i} y_{j}} \partial^2_{y_{i}y_{j}}
\hat{h}(t,y) = 0,
\quad \lim_{y_{i} \rightarrow 0_{+}}
\sqrt{y_{i}} \partial^2_{y_{i}y_{k}}
\hat{h}(t,y) = 0,
\end{split}
\end{equation}
and the functions
$(t,y) \mapsto  
\sqrt{y_{i} y_{j}} \partial^2_{y_{i}y_{j}}
\hat{h}(t,y)$, 
$(t,y) \mapsto  
\sqrt{y_{i}} \partial^2_{y_{i}y_{k}}
\hat{h}(y)$
and
$(t,y) \mapsto  
\partial^2_{y_{k}y_{l}}
\hat{h}(t,y)$
belong to ${\mathscr C}_{\textrm{\rm WF}}^{\gamma/2,\gamma}([0,T] \times \overline{\mathcal U}(\delta^0,x^0))$.
\end{enumerate} 
We then call 
\begin{equation*}
\begin{split}
\| \hat{h} \|_{1+\gamma/2,2+\gamma;  {\mathcal U}(\delta^0,x^0)}
&:= 
\| \hat{h} \|_{\gamma/2,\gamma;  {\mathcal U}(\delta^0,x^0)}
+
\| \partial_{t} \hat{h} \|_{ \gamma/2, \gamma;  {\mathcal U}(\delta^0,x^0)}
+
\sum_{i=1}^{d-1} 
\| \partial_{y_{i}} \hat{h} \|_{ \gamma/2, \gamma;  {\mathcal U}(\delta^0,x^0)}
\\
&\hspace{15pt} +
\sum_{i,j \in I}
\|
\sqrt{y_{i} y_{j}} \partial^2_{y_{i}y_{j}}
\hat{h} 
 \|_{\gamma/2, \gamma;  {\mathcal U}(\delta^0,x^0)}
 +
 \sum_{k,l \not \in I}
\|
 \partial^2_{y_{k}y_{l}}
\hat{h}
 \|_{ \gamma/2, \gamma;  {\mathcal U}(\delta^0,x^0)}
   \\
&\hspace{15pt}+
 \sum_{i \in I}
 \sum_{k  \not \in I}
\|
\sqrt{y_{i}} \partial^2_{y_{i}y_{k}}
\hat{h}
 \|_{\gamma/2, \gamma;  {\mathcal U}(\delta^0,x^0)}. 
\end{split}
\end{equation*}
For a given finite covering $\cup_{i=1}^K {\mathcal U}(\delta^0,x^{0,i})$ of $\hat{\mathcal S}_{d-1}$, a function 
$\hat{h}$ (or equivalently the associated function $h$ defined on $[0,T] \times {\mathcal S}_{d-1}$) is said to be in
${\mathscr C}_{\textrm{\rm WF}}^{1+\gamma/2,2+\gamma}([0,T] \times {\mathcal S}_{d-1})$ if 
$\hat{h}$ belongs to each ${\mathscr C}_{\textrm{\rm WF}}^{2+\gamma}([0,T] \times {\mathcal U}(\delta^0,x^{0,i}))$.
We then let
\begin{equation*}
\| \hat{h} \|_{1+\gamma/2,2+\gamma} := \sum_{i=1}^K \| \hat{h} \|_{1+\gamma/2,2+\gamma;{\mathcal U}(\delta^0,x^{0,i})}. 
\end{equation*} 
We stress the fact that the finite covering that we use in the sequel is fixed once for all. There is no need to change it.

\begin{remark}
Importantly, 
for $y,y' \in \bar{\mathcal U}(\delta^0,x^0)$, 
and
for a constant $c \geq 1$, $c$ depending on $\delta_0$,
\begin{equation*}
\begin{split} 
\biggl\vert \biggl( 1-\sum_{i=1}^{d-1} y_{i} \biggr)^{1/2} - 
\biggl( 1-\sum_{i=1}^{d-1} y_{i}' \biggr)^{1/2} \biggr\vert
&\leq \frac{c-1}{2}
\biggl\vert 
\sum_{i=1}^{d-1} \bigl( y_{i} - y_{i}'\bigr)
  \biggr\vert
\leq (c-1)
\sum_{i=1}^{d-1} \Bigl\vert \sqrt{y_{i}} - \sqrt{y_{i}'}\Bigr\vert.
\end{split}
\end{equation*}
Recalling that the vector $y=(y_{1},\cdots,y_{d-1})$ (resp. $y'=(y_{1}',\cdots,y_{d-1}')$) stand for the new coordinates of 
an element $x \in \hat{\mathcal S}_{d-1}$ (resp. $x'$) and that $x$ (resp. 
$x'$) itself 
is canonically associated with 
$\check x = (x_{1},\cdots,x_{d-1},1-x_{1}-\cdots-x_{d-1}) \in {\mathcal S}_{d-1}$
(resp. $\check x'$),
we deduce that
\begin{equation*}
c^{-1} 
\sum_{i=1}^{d} \Bigl\vert \sqrt{\check x_{i}} - \sqrt{\check x_{i}'} \Bigr\vert
\leq d(y,y') \leq c 
\sum_{i=1}^{d} \Bigl\vert \sqrt{\check x_{i}} - \sqrt{\check x_{i}'} \Bigr\vert,
\end{equation*} 
which permits to reformulate the modulus of continuity showing up in the H\"older condition of 
the Wright--Fisher space in an intrinsic manner. 
Since the number of neighborhoods of the form 
${\mathcal U}(\delta^0,x^0)$  used to cover $\hat{\mathcal S}_{d-1}$ is finite, 
we can choose the same $c$ for all those neighborhoods. 
\end{remark}


\subsubsection{Back to the master equation} Actually, we must point out that, in order to apply Theorem 10.0.2 in \cite{EpsteinMazzeo} (existence of classical solutions to Kimura type PDEs together with related Schauder's estimates), as we do later,  the master equation should be satisfied also in boundary points, under the appropriate local chart; see (10.1) therein.
\color{black}
This would require to formulate \eqref{eq:master:equation:local} for each local chart used in the construction of the Wright--Fischer spaces, that is, for any projection $p^{-i}$ with $i\in\ES$, since we verified in \eqref{eq:canonical:form:kimura} that those changes of variables make the second order operator fit the setup of \cite{EpsteinMazzeo}. 
In this respect, we observe that there is no hindrance for us to restate \eqref{eq:master:equation:local} in the right system of coordinates. Also, we already noticed that, as long as we look for solutions $U^1,\cdots,U^d$ in the space ${\mathscr C}_{\textrm{\rm WF}}^{1+\gamma/2,2+\gamma}([0,T] \times {\mathcal S}_{d-1})$ for some $\gamma \in (0,1)$, 
the first order derivatives always extend by continuity up to the boundary. Moreover,
even though 
the second order derivatives are defined only in the interior of the simplex
and are allowed to blow-up at the boundary, the rate of explosion of those second order derivatives, as prescribed by 
\eqref{eq:rate:2nd:order},
combine well with 
the degeneracy property of the operator on the boundary. Hence, by a standard continuity argument, it
is enough to require that \eqref{eq:master:equation:intrinsic}, which is written in terms of the intrinsic derivative, holds in $\textrm{\rm Int}(\hat{\mathcal S}_{d-1})$.

\subsection{Complete versions of the main statements}
 
We now provide refined versions of 
Meta-Theorems
\ref{main:metathm}
and 
\ref{main:metathm-PDE}.

\subsubsection{Existence and uniqueness of an MFG equilibrium}
We have indeed all the ingredients to 
clarify Meta-Theorem
\ref{main:metathm}
 and to 
formulate the statement in a rigorous manner. 
We recall that the proof is given in Section \ref{subse:MFG:proof:existence:!}. 

\begin{theorem}
\label{main:thm}
Assume that, for some $\gamma >0$,  each $f^{i}$, for $i \in \ES$, 
belongs to 
${\mathscr C}_{\textrm{\rm WF}}^{\gamma/2,\gamma}([0,T] \times {\mathcal S}_{d-1})$,
and
each $g^{i}$, for $i \in \ES$, belongs to 
${\mathscr C}_{\textrm{\rm WF}}^{2+\gamma}({\mathcal S}_{d-1})$.
Then,
for any $\varepsilon \in (0,1)$,
there exists a threshold $\kappa_{0}>0$, 
only depending 
on $\varepsilon$, $\|f\|_{\infty}$, $\| g \|_{\infty}$ and $T$, such that, for any $\kappa \geq \kappa_{0}$
and $\delta \in (0,1/(4\sqrt{d}))$,
and 
for any (deterministic) initial condition $(p_{0,i})_{i \in \ES} \in {\mathcal S}_{d-1}$ with positive entries, the mean field game  has a unique solution, in the sense of Definition \ref{defmfg}. 
\end{theorem}

\begin{remark}
As already emphasized in Remark 
\ref{rem:meta-thm}, 
we could certainly extend the 
uniqueness
result to the larger class of open-loop strategies and also to the case when the initial condition is random.

As for the assumptions on the coefficients, the key fact is that there is no need for any monotonicity condition in the statement. Still, it would be interesting to see whether the result remains true under lower regularity conditions on the function $g$. 
Assuming $g$ to have two H\"older continuous derivatives (as we do here) is quite convenient since it allows to find a solution to the
master equation (see the next section) that remains smooth up to the boundary at time $T$. More effort would be needed to allow for more general (and hence less regular) terminal costs; accordingly, it would require to address with care the rate at which the derivatives of the solution to the master equation would blow up at terminal time. We leave this problem for future work. 
\end{remark}

\subsubsection{Existence and uniqueness of a solution to the master equation}

Here is now the rigorous version of 
Meta-Theorem
\ref{main:metathm-PDE}, the proof of which is given in 
Section 
\ref{subse:master:e!}.

\begin{theorem}
\label{existence:master}
Assume that, for some $\gamma \in (0,1/2]$,  each $f_{i}$, for $i \in \ES$, 
belongs to 
${\mathscr C}_{\textrm{\rm WF}}^{\gamma/2,\gamma}([0,T] \times {\mathcal S}_{d-1})$,
and
each $g_{i}$, for $i \in \ES$, belongs to 
${\mathscr C}_{\textrm{\rm WF}}^{2+\gamma}({\mathcal S}_{d-1})$.
Then,
for any $\varepsilon \in (0,1)$,
there exist a universal exponent $\eta \in (0,1)$ (hence independent of $\varepsilon$)
and a threshold $\kappa_{0}>0$, 
only depending 
on $\varepsilon$, $\|f\|_{\infty}$, $\| g \|_{\infty}$ and $T$, such that, for any $\kappa \geq \kappa_{0}$
and $\delta \in (0,1/(4\sqrt{d}))$, the master equation 
\eqref{eq:master:equation:local}--\eqref{eq:master:equation:intrinsic}
has a solution in  
$[{\mathscr C}_{\textrm{\rm WF}}^{1+\gamma'/2,2+\gamma'}([0,T] \times {\mathcal S}_{d-1})]^d$, 
for $\gamma' = \min(\gamma,\eta)/2$.  

Moreover, the solution is unique in the sense that 
$(U^1,\cdots,U^d)$ 
is the unique tuple
such that, for any 
$i \in \ES$, $U^i$ belongs to the Wright--Fisher space 
${\mathscr C}_{\textrm{\rm WF}}^{1+\gamma'/2,2+\gamma'}([0,T] \times {\mathcal S}_{d-1})$ 
for some $\gamma' >0$, and
\eqref{eq:master:equation:local}--\eqref{eq:master:equation:intrinsic}
hold
at any $(t,p) \in [0,T] \times \textrm{\rm Int}(\hat{\mathcal S}_{d-1})$.
\end{theorem}

\begin{remark}
We do not provide any new version of Theorem \ref{main:holder}. 
The reader 
will find in 
\eqref{eq:pde:apriori:1}
a more suitable version of equation
\eqref{eq:pde:apriori:1:intro}, when expressed in intrinsic derivatives. 
We let her/him adapt accordingly the definition of the space 
 ${\mathcal C}^{1,2}([0,T] \times \textrm{\rm Int}(\hat{\mathcal S}_{d-1}),\RR)$ that is used in the statement. 
\end{remark}

\color{black}

\subsection{Proofs of auxiliary results from Section 2}
\label{subse:proofs}

We now prove {the auxiliary results stated in 
\S \ref{subsub:repelled:from:boundary}}. 
\subsubsection{Proof of Proposition \ref{thm:approximation:diffusion:2}}
The proof holds in two steps. We give a sketch of it only. 
\vskip 4pt

\textit{First Step.} By \eqref{eq:sum=1}, it suffices to solve the equation for $(P_{t}^1,\cdots,P_{t}^{d-1})_{0 \leq t \leq T}$. 
As long as the latter stays in the interior of $\hat{\mathcal S}_{d-1}$, the equation satisfied by the process is 
non-degenerate, see 
\eqref{eq:ellipticity}. Moreover, the diffusion matrix is Lipschitz away from the boundary of the simplex and the drift is bounded. 
Therefore, by a standard localization argument, we can easily adapt the strong existence and uniqueness result 
of Veretennikov \cite{Veretennikov} (see Remark 3 therein for the case when the initial condition is random and Remark 4
therein for the case when the state variable and the underpinning Brownian motion do not have the same dimension) and then deduce that, up until one coordinate (including 
the $d$th coordinate, as given by $(P_{t}^d = 1- (P_{t}^1+\cdots+P_{t}^{d-1}))_{0 \leq t \leq T}$) reaches a given 
positive threshold $\upepsilon$, a (strong) solution exists and is pathwise unique. 
By letting $\upepsilon$ tend to $0$, we deduce that there exists a solution up to the first time it reaches the boundary of the simplex (or, equivalently, one of the coordinates vanishes) and this  solution is pathwise unique (once again, up to the first time it reaches the boundary). 
\vskip 4pt

\textit{Second Step.}
The goal in this step is to prove that, for $\kappa \geq \varepsilon^2/2$, the solution of  
\eqref{eq:weak:sde:final} until {the first time} it reaches the boundary of $\hat{\mathcal S}_{d-1}$
stays in fact away from the boundary of $\hat{\mathcal S}_{d-1}$. {As a consequence, the} solution of 
\eqref{eq:weak:sde:final} {is well defined} on the entire $[0,T]$. 
To do so, we come back to 
\eqref{eq:pti:marginal}, namely, we write the dynamics of the $i$th coordinate (for $i=1,\cdots,d-1$) in the form 
\begin{equation*}
\ud P_{t}^i = 
  a_{i}\bigl(t,(P^1_{t},\cdots,P_{t}^d)\bigr) \ud t + 
 \varepsilon \sqrt{  P_{t}^i ( 1-P_{t}^i)}
\ud \widetilde W_{t}^i,
\end{equation*}
with $a_{i}$ as in 
\eqref{eq:bi:new}.  
The above holds true up until the first time $\tau :=\inf\{ t \in [0,T] : (P^1_{t},\cdots,P_{t}^{d-1}) \in \partial \hat{\mathcal S}_{d-1} \} \wedge T$. 
Then, 
using the fact that $\upalpha$ is bounded and $\upalpha(t,j,x)(i) \geq 0$ for $i \not = j$ in 
\eqref{eq:bi:new}, 
we can easily compare the process $(P_{t}^i)_{0 \leq t \leq \tau}$ with the solution of the equation
\begin{equation*}
\ud \overline P_{t}^i = 
\Bigl(    \varphi(\overline P_{t}^i) - C \overline P_{t}^i \Bigr) \ud t + 
 \varepsilon \sqrt{   (\overline P_{t}^i)_{+} ( 1-\overline P_{t}^i)_{+}}
\ud \widetilde W_{t}^i,
\end{equation*}
with $\bar p_{0}^i = p_{0}^i$ as initial condition, for a constant $C \geq 0$.  
Above, $(\cdot)_{+}$ stands for the positive part. 
By \cite[Chapter 5, Proposition 2.13]{KaratzasShreve}, the above equation has a unique strong solution. 
Letting $\bar \tau^i :=\inf\{ t \in [0,T] : \overline P_{t}^i \in \{0,1\} \} \wedge T$ and choosing $C$ large enough, we have 
$\overline P_{t}^i \leq P_{t}^i$, for all $t \leq \tau \wedge \bar \tau^i$. 
We then apply Feller's test (see \cite[Chapter 5, Proposition 5.22]{KaratzasShreve}) to 
$(\overline{P}_{t}^i)_{0 \leq t \leq T}$ (the reader may notice that the fact that the initial condition is  random 
is not a hindrance since it belongs to $(0,1)$ with probability 1). The natural scale
(see \cite[Chapter 5, (5.42)]{KaratzasShreve})
 is here given by
\begin{equation*} 
\Phi(r) := \int_{\delta}^r \exp \Bigl(- 2 \int_{\delta}^s  \frac{ \varphi(u) - Cu}{\varepsilon^2 u (1-u)} \ud u \Bigr) \ud s, \quad r \in (0,1). 
\end{equation*}
For $r \in (0, \delta]$,
\begin{equation*} 
\begin{split}
- \Phi(r) &\geq  \int_r^{\delta} \exp \Bigl( 2  \frac{ \kappa}{\varepsilon^2} \ln( \frac{\delta}{s} ) - \frac{2 C \delta}{\varepsilon (1-\delta)}\Bigr) \ud s,
\end{split}
\end{equation*}
from which we get that $\Phi(0+)=-\infty$ if $2 \kappa/\varepsilon^2 \geq 1$. We deduce that, if the latter is true, 
$(\overline P_{t}^i)_{0 \leq t \leq T}$ does not touch $0$. By comparison, 
we deduce that
$( P_{t}^i)_{0 \leq t \leq \tau}$ cannot touch $0$ before it touches $1$, that is $P_{\tau}^i =1$ if the set $\{t \in [0,\tau] : 
P_{t}^i \in \{0,1\}\}$ is not empty.
This holds true for $i =1,\cdots,d-1$
, but 
by 
choosing another system of coordinates, we get the same result for 
the coordinate $i=d$. 
Assume now that we can find some coordinate $i \in \ES$
such that 
$P_{\tau}^i \in \{0,1\}$,
in which case our analysis says that  
$P_{\tau}^i = 1$.
Since $\sum_{j=1}^d P_{\tau}^j=1$, we deduce that 
$P_{\tau}^j = 0$ for all $j \in \ES \setminus \{i\}$, which is a contradiction with our analysis. 
So, the conclusion is that, at time $\tau$, we must have $P_{\tau}^i \in (0,1)$
for all $i \in \ES$. That is, $\tau=T$ and the process $(P_{t}=(P_{t}^1,\cdots,P_{t}^d))_{0 \leq t \leq T}$ remains in the ($(d-1)$-dimensional) interior of 
${\mathcal S}_{d-1}$.  \qed

\subsubsection{Proof of
Proposition \ref{expfin}}

As in the proof of Proposition \ref{thm:approximation:diffusion:2}, 
we write the equation for $(P_{t}^i)_{0 \leq t \leq T}$ (for a given $i \in \ES$, $i$ being possibly equal to $d$) in the form 
\begin{equation*}
\ud P_{t}^i = 
  a_{i}(t,P_{t}) \ud t + 
 \varepsilon \sqrt{  P_{t}^i ( 1-P_{t}^i)}
\ud \widetilde W_{t}^i,
\end{equation*}
with $a_{i}$ as in 
\eqref{eq:bi:new}. Then, we get, 
by It\^o's formula (recall that the left-hand side below is well-defined since 
$(P_{t}^i)_{0 \leq t \leq T}$ does not vanish),
\begin{equation}
\label{ito:log}
\begin{split}
\ud \bigl[ \ln P_{t}^i \bigr] &=  
\sum_{j \in \ES} \biggl[ \frac{P^j_{t}}{P^i_{t}}\Bigl(\varphi(P^i_t)  + \upalpha(t,j,P_{t})(i) \Bigr)  
-\Bigl(\varphi(P^j_t)+ \upalpha(t,i,P_{t})(j) \Bigr)
\biggr]
\ud t
 - \frac{\varepsilon^2}{2} \frac{1- P_{t}^i}{P_{t}^i} \ud t
\\
&\hspace{15pt}
+   \varepsilon \sqrt{\frac{1-P_{t}^i}{P_{t}^i}} \ud \widetilde W_{t}^{i}.
\end{split}
\end{equation}
For a constant 
 $C$ depending on the same parameters as those quoted in the statement,  
we can lower bound the drift
in \eqref{ito:log} as follows  (using the definition of $\varphi$ in \eqref{eq:varphi} together with the fact that 
$\upalpha(t,j,P_{t})(i)$ is non-negative if $j \not = i$)
\begin{equation*}
\begin{split}
&\sum_{j \in \ES} \biggl[ \frac{P^j_t}{P^i_t}\Bigl(\varphi(P^i_t)  + \upalpha(t,j,P_{t})(i) \Bigr) 
- \Bigl(\varphi(P^j_t)+ \upalpha(t,i,P_{t})(j) \Bigr) \biggr]
 - \frac{\varepsilon^2}{2} \frac{1- P_{t}^i}{P_{t}^i} 
 \\
&\hspace{5pt} \geq 
 \frac{1-P^i_t}{P_{t}^i}  \kappa  
{\mathbf 1}_{\{P_{t}^i \leq \delta\}} -\frac{\varepsilon^2}{2}\frac{1-P^i_t}{P_{t}^i} - C. 
\end{split}
\end{equation*}
Allowing the value of $C$ to change from line to line and recalling that $\kappa \geq \varepsilon^2/2$, we get 
\begin{equation*}
\begin{split}
&\sum_{j \in \ES} \biggl[ \frac{P^j_t}{P^i_t}\Bigl(\varphi(P^i_t)  + \upalpha(t,j,P_{t})(i) \Bigr) 
- \Bigl(\varphi(P^j_t)+ \upalpha(t,i,P_{t})(j) \Bigr) \biggr]
 - \frac{\varepsilon^2}{2} \frac{1- P_{t}^i}{P_{t}^i} 
 \\
&\geq  \frac{1-P_{t}^i}{P^i_t}\Bigl(\kappa - \frac{\varepsilon^2}{2}\Bigr)  -C.  
\end{split}
\end{equation*}
Hence, integrating \eqref{ito:log} from $0$ to some stopping time $\tau$ (with values in $[0,T]$),
adding and subtracting the compensator $(\lambda \varepsilon^2/2)\int_{0}^{\tau} (1-P_{t}^i)/P_{t}^i \ud t$,
 multiplying by $\lambda$ and then taking the exponential, we get
\begin{equation}
\label{eq:E:Pt:eta:2}
\begin{split}
&(P_{ \tau}^i)^{\lambda} \exp \biggl( - \lambda \varepsilon \int_{0}^{ \tau}  \sqrt{\frac{1-P_{t}^i}{P_{t}^i}}
 \ud \widetilde W_{t}^{i} - \frac{\lambda^2 \varepsilon^2}2 
\int_{0}^{\tau} \frac{1-P_{t}^i}{P_{t}^i} \ud t 
\biggr)
\\
&\geq (p_{0,i})^{\lambda}
\exp \biggl(   \lambda   \Bigl[ \kappa - \frac{\varepsilon^2 (1+\lambda)}{2} \Bigr] 
\int_{0}^{\tau} \frac{1}{P_{t}^i} \ud t  -C\biggr).
\end{split}
\end{equation}
Choosing $\tau=\inf\{ t \in [0,T] : P_{t}^i \leq \upepsilon\}\wedge T$, for $\upepsilon >0$ as small as needed, the left-hand side has conditional expectation less than 1.
So, taking expectation and letting $\upepsilon$ tend to $0$, we deduce that 
\begin{equation*}
\E\left[ \exp \biggl(   \lambda 
\Bigl[ \kappa - \frac{\varepsilon^2 (1+\lambda)}{2} \Bigr] 
\int_{0}^{T} \frac{1}{P_{t}^i} \ud t \biggr)\right] \leq C 
(p_{0,i})^{-\lambda} .
\end{equation*}
The bound \eqref{eq:expfin2} easily follows.

{It then remains to prove \eqref{eq:expfin2:bis}. To do so, we come back to \eqref{eq:E:Pt:eta:2}. Using the fact that 
$\gamma$ is positive and choosing $\tau = t \in [0,T]$, we rewrite it in the form 
\begin{equation*}
\begin{split}
&(P_{t}^i)^{-\lambda} \leq (p_{0,i})^{-\lambda}
 \exp \biggl(C - \lambda \varepsilon \int_{0}^{t}  \sqrt{\frac{1-P_{s}^i}{P_{s}^i}}
 \ud \widetilde W_{t}^{i} - \frac{\lambda^2 \varepsilon^2}2 
\int_{0}^{t} \frac{1-P_{s}^i}{P_{s}^i} \ud s 
\biggr).
\end{split}
\end{equation*}
Taking expectations on both sides, we easily complete the proof of \eqref{eq:expfin2:bis}.} \qed

\subsubsection{Proof of
Proposition 
\ref{expfin:2}}
For each $i \in \ES$,
we call $({\mathcal E}_{t}^i)_{0 \le t \le T}$
the Dol\'eans--Dade exponential
\begin{equation*}
{\mathcal E}_{t}^i
:= \exp 
\biggl(
 \varepsilon  \sum_{j \in \ES} \int_{0}^t \sqrt{\frac{P_{s}^j}{P_{s}^i}} 
 \ud  \overline W_{s}^{i,j}  
 - \frac{\varepsilon^2}{2}
 \int_{0}^t \frac{1-P_{s}^i}{P_{s}^i}
 \ud s \biggr), \quad t \in [0,T].
\end{equation*}
Then, $(Q_{t}^i)_{0 \leq t \leq T}$ is a solution to 
\eqref{eq:weak:sde:final:2:b} if and only if
\begin{equation*}
\begin{split} 
\ud \bigl[ \bigl({\mathcal E}_{t}^i \bigr)^{-1} Q_{t}^i  \bigr] &=   \sum_{j \in \ES} \bigl({\mathcal E}_{t}^i \bigr)^{-1} \Bigl( Q_{t}^j \bigl( \varphi(P_{t}^i) + \upbeta(t,j,P_{t})(i) \bigr) - Q_{t}^i 
\bigl( \varphi(P_{t}^j) 
+  \upbeta(t,i,P_{t})(j) \bigr)
\Bigr) \ud t, 
\end{split}
\end{equation*}
which may be rewritten in the form
\begin{equation}
\label{eq:linear:eq}
\begin{split} 
&\ud  \widetilde Q_{t}^i 
=   \sum_{j \in \ES}  \Bigl(
\bigl({\mathcal E}_{t}^i \bigr)^{-1}
{\mathcal E}_{t}^j 
 \widetilde Q_{t}^j  \bigl( \varphi(P_{t}^i) + \upbeta(t,j,P_{t})(i) \bigr) - 
 \widetilde Q_{t}^i 
\bigl( \varphi(P_{t}^j) 
+  \upbeta(t,i,P_{t})(j) \bigr)
\Bigr) \ud t, 
\end{split}
\end{equation}
for $t \in [0,T]$, 
with  
$(\widetilde Q_{0}^i=q_{0,i})_{i \in \ES} \in {\mathcal S}_{d-1}$ 
as initial condition
and
under the change of variable
\begin{equation}
\label{eq:change:variable}
\widetilde Q_{t}^i := \bigl({\mathcal E}_{t}^i \bigr)^{-1} Q_{t}^i, \quad t \in [0,T]. 
\end{equation}
Obviously, 
\eqref{eq:linear:eq}
has a unique pathwise solution. It is continuous and adapted to the filtration 
${\mathbb F}^{\boldsymbol W}$. 
Since $\upbeta(t,j,P_{t})(i) \geq 0$ for $j \not =i$, it is pretty easy to check that all the coordinates 
remain (strictly) positive. 

Given the solution to 
\eqref{eq:linear:eq}, we may reconstruct 
$((Q_{t}^i)_{0 \leq t \leq T})_{i \in \ES}$ from the change of variable \eqref{eq:change:variable}. 
Then, taking the power $l$ in \eqref{eq:weak:sde:final:2:b}, for an exponent $l \geq 1$, we get
\begin{equation}
\label{eq:moment:ell}
\begin{split}
\ud \bigl( Q_{t}^i \bigr)^l &=  l \sum_{j \in \ES} \bigl( Q_{t}^i \bigr)^{l-1}\Bigl( Q_{t}^j \bigl( \varphi(P_{t}^i) + \upbeta(t,j,P_{t})(i) \bigr) - Q_{t}^i 
\bigl( \varphi(P_{t}^j) 
+  \upbeta(t,i,P_{t})(j) \bigr)
\Bigr)  \ud t
\\
&\hspace{15pt}+ \frac{l (l-1)}{2} \varepsilon^2 \bigl( Q_{t}^i \bigr)^{l} \frac{1-P_{t}^i}{P_{t}^i} \ud t
+  \varepsilon  {l}
\bigl( Q_{t}^i \bigr)^l \sum_{j \in \ES}
\sqrt{\frac{P_{t}^j}{P_{t}^i}}  
\ud \overline W_{t}^{i,j}, 
\end{split}
\end{equation}
for $t \in [0,T]$. As a result,
we can find a constant $C$, only depending on $l$, $\kappa$ and on the supremum norm of 
$\upbeta$, such that 
\begin{equation*}
\begin{split}
&\ud \biggl[\sum_{i \in \ES} \bigl( Q_{t}^i \bigr)^l
 \biggr]
\leq  \biggl[ C +
  \varepsilon^2
  \frac{l (l-1)}2  
\sum_{j \in \ES}
\frac{1-P_{t}^j}{P_{t}^j}
 \biggr] \cdot 
\biggl[\sum_{i \in \ES} \bigl( Q_{t}^i \bigr)^l
 \biggr] \ud t
+
\ud m_{t}, 
\end{split}
\end{equation*}
where $(m_{t})_{0 \le t \le T}$ is a local martingale\footnote{\label{foot:notation:geq}Here and throughout, the notation 
$\ud X^1_{t} \geq \ud X^2_{t}$, $t \in [0,T]$, for two stochastic processes  
$((X_{t}^i)_{0 \le t \le T})_{i=1,2}$, is understood as $(X_{t}^2-X_{t}^1)_{0 \le t \le T}$
is a non-decreasing process.}.
By a standard localization argument, we end up with 
\begin{equation*}
\sup_{0 \leq t \leq T}
{\mathbb E}
\biggl[ 
\biggl(\sum_{i \in \ES} \bigl( Q_{t}^i \bigr)^l
 \biggr)
 \exp\biggl( -   \varepsilon^2
    \frac{l (l-1)}2  \int_{0}^T 
  \sum_{j \in \ES}
\frac{1-P_{s}^j}{P_{s}^j}
\ud s \biggr)
 \biggr] \leq C,
\end{equation*}
for a new value of $C$. And then, applying Cauchy-Schwarz inequality and invoking the above inequality with $2l$ instead of $l$,
we get 
\begin{equation*}
\sup_{0 \leq t \leq T}
{\mathbb E}
\biggl[ 
\sum_{i \in \ES} \bigl( Q_{t}^i \bigr)^l
 \biggr] 
 \leq C 
 \sup_{0 \leq t \leq T}
 {\mathbb E} \biggl[
 \exp\biggl(   \varepsilon^2
    l (2l-1)  \int_{0}^T 
  \sum_{j \in \ES}
\frac{1-P_{t}^j}{P_{t}^j}
\ud t \biggr)
 \biggr]^{1/2}.
\end{equation*}
Take $l=8$ and choose {$\gamma \geq 60\varepsilon^2$} and {$\lambda= 2d$} in the statement of Proposition 
\ref{expfin} ({which is indeed possible since $\kappa-(1+\lambda) \varepsilon^2/2 \geq (61+d) \varepsilon^2 - (1/2+d) \varepsilon^2 \geq 60 \varepsilon^2$}). Then, {by H\"older's inequality (in order to handle the sum inside the exponential)}, the right-hand side is upper bounded. 
Returning to 
\eqref{eq:moment:ell}, 
{invoking Burkholder--Davis--Gundy inequalities,  
\eqref{eq:expfin2:bis} and the bound 
$\sup_{0 \leq t \leq T}{\mathbb E}[ (Q_{t}^i)^{4}/P_{t}^i] \leq 
\sup_{0 \leq t \leq T}{\mathbb E}[ (Q_{t}^i)^{8}]^{1/2}
\sup_{0 \leq t \leq T}{\mathbb E}[ (P_{t}^i)^{-2}]^{1/2}$, for $i \in \ES$, we deduce that 
$\sup_{0 \leq t \leq T} \vert Q_{t}^i \vert$ has a finite fourth moment
for each $i \in \ES$}.
Equality 
\eqref{eq:mean:q:1} is easily proved by summing over 
$i \in \ES$ in 
 \eqref{eq:weak:sde:final:2:b}.
\qed

\section{From the MFG system to the master equation}\label{sec:MFG-master}

{This section is dedicated to the proofs of 
Theorems 
\ref{main:thm}
and 
\ref{existence:master}, taken for granted the statement of 
Theorem 
\ref{main:holder}. 
Throughout the section, we assume that the condition {$\kappa \geq (61+d) \varepsilon^2$} is in force.}

\subsection{MFG system}
With the optimization problem driven by the cost functional 
\eqref{eq:limit:cost:functional}
and the state equation 
\eqref{eq:weak:sde:final:2} (within the environment 
\eqref{eq:weak:sde:final}), we may associate a value function. Obviously, we may expect this value function to solve a 
stochastic (because of the common noise) variant of the usual Hamilton--Jacobi--Bellman equation for a stochastic optimal control problem on 
a discrete state space. 
The combination of this Stochastic Hamilton--Jacobi--Bellman (SHJB) equation 
with the equation 
\eqref{eq:weak:sde:final} for the environment will lead us to a relevant version of the so-called MFG system
(which is a key tool in the standard theory of mean field games, see for instance 
the  references \cite{Cardaliaguet,CardaliaguetDelarueLasryLions,CarmonaDelarue_book_II,Lionscollege2,Lionscollege1,Lionsvideo}). 

\subsubsection{Formulation of the system}
In order to proceed, we recall \eqref{eq:limit:cost:functional:new}. 
Importantly, $(P_{t})_{0 \le t \le T}$ therein
is regarded as a stochastic environment. Typically, it is the solution of an equation 
of the form 
\eqref{eq:weak:sde:final}. In any case, it is a continuous ${\mathcal S}_{d-1}$-valued process
that is  
progressively-measurable
with respect to the filtration ${\mathbb F}^{\boldsymbol W}$
and that satisfies the conclusion of Proposition \ref{expfin}, see \eqref{eq:expfin2}
and 
\eqref{eq:expfin2:bis}. 
In particular, it remains away from the boundary of the simplex. 

The related value function at time $t \in [0,T]$ is defined as 
\begin{equation}
 \label{mfgvaluef}
\begin{split}
&u^l\Bigl(t,(P_{s})_{t \le s \le T} \Bigr):=\underset{(\beta_{s})_{t \le s \le T}}{\textrm{\rm ess inf}} \ {\mathcal J}^l\Bigl(t,(\beta_{s})_{t \le s \le T},(P_{s})_{t \le s \le T}\Bigr), \quad l \in \ES, 
\\
&{\mathcal J}^l\Bigl(t,(\beta_{s})_{t \le s \le T},(P_{s})_{t \le s \le T} \Bigr)
\\
&\hspace{15pt} :=\sum_{i \in \ES}
{\mathbb E}\biggl[ Q_{T}^i[t,l] g^i(P_{T}) 
+ \int_{t}^T \Bigl( Q_{s}^i[t,l] \bigl[ f^i \bigl(s,P_{s} \bigr) + 
 \frac12 
 \sum_{j \not = i}
\bigl\vert \beta_{t}^{i,j} \bigr\vert^2
\bigr] \Bigr) \ud s \,  \Big| \, \mathcal{F}_t^{\boldsymbol W}\biggr], \quad l \in \ES, 
\end{split}
\end{equation}
whereas $(Q^i_s[t,l])_{t \leq s \leq T}$ is the solution to \eqref{eq:weak:sde:final:2} when the initial time is $t\in[0,T)$ and the initial distribution is $Q_t^i[t,l]=\delta_{l,i}$, for $i \in \ES$.
Importantly, 
the value function 
is random: Stochasticity accounts for the fact that the cost functionals $f$ and $g$ in the optimal control problem depend upon 
the environment $(P_{s})_{0 \leq s \leq T}$, which is random itself.
    Hence the corresponding HJB equation is a backward stochastic HJB equation (SHJB) that takes the form
    of a system of backward SDEs indexed by $i \in \ES$:
\begin{equation}
\label{shjbmfg}
\begin{split}
&\displaystyle \ud u_{t}^i = - \Bigl( 
\sum_{j \in \ES} \varphi(P^j_{t})
\bigl[ u_{t}^j - u_{t}^i \bigr]
+ H^i(u_{t}) + f^i(t,P_{t}) \Bigr) \ud t   - \frac{\varepsilon}{\sqrt{2}} \sum_{j \in \ES : j \not = i} 
\sqrt{\frac{{P_{t}^j}}{{P_{t}^i}}} \bigl( \nu_{t}^{i,i,j} - \nu_{t}^{i,j,i} \bigr) \ud t
\\
&\hspace{30pt}+ \sum_{j,k \in \ES : j \not = k} \nu_{t}^{i,j,k} \ud W_{t}^{j,k}, 
\\
&u^i_T=g^i(P_T),
\end{split}
\end{equation}
where $H^i$ is the Hamiltonian
\begin{equation}
\label{eq:Ha}
H^i(y) := -\frac12 \sum _{j \in \ES}(y_{i}-y_{j})_+^2,
\quad y=(y_{j})_{j \in \ES}.
\end{equation}
It is worth emphasizing that, in
the equation  
\eqref{shjbmfg}, the unknown is the larger family of processes 
$((u_{t}^i)_{i \in \ES},(\nu^{i,j,k}_{t})_{i,j,k \in \ES : j \not = k})_{0 \leq t \leq T}$, 
which are required to be progressively measurable with respect to 
${\mathbb F}^{\boldsymbol W}$. 
This is a standard fact in the theory of backward SDEs
and the role of 
the processes $((\nu^{i,j,k}_{t})_{i,j,k \in \ES : j \not = k})_{0 \leq t \leq T}$
is precisely to force the solution of the stochastic HJB equation 
to be non-anticipative. 
The reason why we here choose indices $(i,j,k)$ with $j \not =k$ is quite clear: 
there are no noises of the form $((W^{j,j}_{t})_{0 \le t \le T})_{j \in \ES}$ in the forward equation. 

\subsubsection{Verification argument}
\label{subver}
The 
following verification argument clarifies the 
connection between 
 \eqref{mfgvaluef}
 and \eqref{shjbmfg}.
 
 \begin{lemma}
\label{lem:verification:1}
For an environment ${\boldsymbol P}=(P_{t})_{0 \le t \le T}$ as 
 before (satisfying in particular \eqref{eq:expfin2}
 and 
\eqref{eq:expfin2:bis}), assume that there exists a solution  
 $((u_{t}^i)_{i \in \ES},(\nu^{i,j,k}_{t})_{i,j,k \in \ES : j \not =k})_{0 \leq t \leq T}$
 to 
\eqref{shjbmfg} such that 
 $((u_{t}^i)_{i \in \ES})_{0 \le t \le T}$ is bounded (by a deterministic constant) and 
 \begin{equation*}
\sum_{i,j,k \in \ES : j \not = k}{\mathbb E} \int_{0}^T \vert \nu^{i,j,k}_{t} \vert^2 \ud t < \infty.
\end{equation*}
 For $t \in [0,T]$,
let  $(\beta_{s}^\star)^{i,j}: = (u_{s}^i - u_{s}^j)_{+}$ for $i \not =j$ and $s \in [t,T]$.  
Then, 
${\mathcal J}^l\bigl(t,(\beta_{s}^\star)_{t \leq s \leq T},(P_{s})_{t \le s \le T} \bigr)=u_{t}^l$ and, for any other (bounded) strategy,
say $(\beta_{s})_{t \leq s \leq T}$, such that 
\begin{equation*}
\sum_{i,j \in \ES : i \not =j} \int_{t}^T {\mathbb P} \bigl( \beta_{s}^{i,j} \not = (\beta_{s}^\star)^{i,j} \bigr) \ud s >0, 
\end{equation*}
the cost 
${\mathcal J}^l\bigl(t,(\beta_{s})_{t \leq s \leq T},(P_{s})_{t \le s \le T} \bigr)$
is strictly higher than $u_{t}^l$.
 \end{lemma}
 
In words, the above says that  $(((\beta_{s}^\star)^{i,j} = (u^i_s-u^j_s)_+)_{i,j \in \ES : i \not = j})_{t \le s \le T}$
is the unique optimal control. 
In fact, the solvability of 
the equation 
\eqref{shjbmfg} is guaranteed by the following lemma. 

\begin{lemma}
\label{lem:verification:2}
For ${\boldsymbol P}=(P_{t})_{0 \le t \le T}$ as 
before (satisfying \eqref{eq:expfin2} and 
\eqref{eq:expfin2:bis}), 
\eqref{shjbmfg} has a unique (progressively-measurable) solution  
 $((u_{t}^i)_{i \in \ES},(\nu^{i,j,k}_{t})_{i,j,k \in \ES : j \not = k})_{0 \leq t \leq T}$
 such that 
 $((u_{t}^i)_{i \in \ES})_{0 \le t\le T}$
 is almost surely bounded by a deterministic constant and  
$((\nu^{i,j,k}_{t})_{i,j,k \in \ES : j \not = k})_{0 \leq t \leq T}$ 
 satisfies
 \begin{equation*}
\begin{split}
&\sum_{i,j,k \in \ES : j \not = k} {\mathbb E} \biggl[ 
\int_{0}^T
 \exp \biggl( \varepsilon^2
\sum_{l \in \ES} \int_{0}^t \frac1{P_{s}^l}
\ud s
\biggr)\vert \nu_{t}^{i,j,k} \vert^2
\ud t \biggr] <\infty.
\end{split}
\end{equation*}
Abusively, such a solution is said to be bounded.
\end{lemma}
The two lemmas will be proved in Subsection 
\ref{se:4:proofs} below. 
For the time being, we observe, by combining the two of them, 
that, for a given 
${\boldsymbol P}=(P_{t})_{0 \le t \le T}$ satisfying 
\eqref{eq:expfin2} and 
\eqref{eq:expfin2:bis},  
the solution to the optimal control problem
\eqref{eq:limit:cost:functional:new} 
is entirely described by the SHJB equation 
\eqref{shjbmfg}, as it suffices to solve 
the forward equation 
\eqref{eq:weak:sde:final:2}
with
$((\beta^{i,j}_{t} = (u^i_{ t}-u^j_{t})_+)_{i,j \in \ES : i \not =j})_{
0 \le t \le T}$
therein. 
Now, Definition \ref{defmfg} implies that 
an environment ${\boldsymbol P}$ is a solution to 
the MFG in hand if and only if it solves the forward equation in the 
forward-backward system:
\begin{align}
&\ud P_{t}^i =  \sum_{j \in \ES} \Bigl( P_{t}^j \bigl( \varphi(P_{t}^i) + (u^j_t-u^i_t)_+ \bigr) - P_{t}^i \bigl( \varphi(P_{t}^j) 
+ (u^i_t-u^j_t)_+ \bigr)
\Bigr) \ud t
  + {\varepsilon}  
\sum_{j \in \ES}
\sqrt{P_{t}^i P_{t}^j}  
\ud \overline W_{t}^{i,j}, \nonumber
\\
&\displaystyle \ud u_{t}^i = - \Bigl( 
\sum_{j \in \ES} \varphi(P^j_{t})
\bigl[ u_{t}^j - u_{t}^i \bigr]
+ H^i(u_{t}) + f^i(t,P_{t}) \Bigr) \ud t 
\label{eq:MFG:system:epsilon}
\\
&\hspace{30pt} - \frac{\varepsilon}{\sqrt{2}} \sum_{j \in \ES : j \not = i} 
\sqrt{\frac{{P_{t}^j}}{{P_{t}^i}}} \bigl( \nu_{t}^{i,i,j} - \nu_{t}^{i,j,i} \bigr) \ud t
+ \sum_{j,k \in \ES : j \not = k } \nu_{t}^{i,j,k} \ud W_{t}^{j,k}, \nonumber
\end{align}
with $(P_{0}^i = p_{0,i})_{i \in \ES} \in {\mathcal S}_{d-1}$ as deterministic initial condition for the forward equation and 
$(u^i_T=g^i(P_T))_{i \in \ES}$ as terminal boundary condition for the backward equation. 
System \eqref{eq:MFG:system:epsilon} is the (stochastic) MFG system that characterizes the solutions of the MFG
described in Definition \ref{defmfg}. 
Hence, proving Theorem 
\ref{main:thm} is here the same as proving that \eqref{eq:MFG:system:epsilon}
is uniquely solvable (within the space of processes that satisfy the conditions described in 
Proposition 
\ref{expfin:2}
and 
Lemma 
\ref{lem:verification:2}).

\subsubsection{Proofs of Lemmas \ref{lem:verification:1} and \ref{lem:verification:2}}
\label{se:4:proofs}
\begin{proof}[Proof of Lemma \ref{lem:verification:1}.] 
Call $((Q_{s}^i)_{i \in \ES})_{t \le s \le T}$ the solution to \eqref{eq:weak:sde:final:2}
with $Q_{t}^i = \delta_{i,l}$ for some $l \in \ES$
and expand
\begin{equation}
\label{eq:4}
\begin{split}
&\ud \biggl( \sum_{i \in \ES} Q_{s}^i u_{s}^i + \int_{t}^s \sum_{i \in \ES}  Q_{r}^i \bigl( f^i(r,P_{r}) + \frac12 \sum_{j \not = i} \vert \beta^{i,j}_{r} \vert^2
\bigr) \ud r \biggr)
\\
&= -  \sum_{i \in \ES} Q_{s}^i \Bigl( 
\sum_{j \in \ES} \varphi(P^j_{s})
\bigl[ u_{s}^j - u_{s}^i \bigr]
+ H^i(u_{s})   \Bigr) \ud s
- 
\frac{\varepsilon}{\sqrt{2}}
\sum_{i \in \ES} 
Q_{s}^i
 \sum_{j \not =i} 
\sqrt{\frac{{P_{s}^j}}{{P_{s}^i}}} \bigl( \nu_{s}^{i,i,j} - \nu_{s}^{i,j,i} \bigr) \ud s
\\
&\hspace{15pt} 
+
\sum_{i \in \ES} u_{s}^i
\sum_{j \in \ES} \Bigl( Q_s^j \bigl( \varphi(P_{s}^i) + \beta_{s}^{j,i} \bigr) - Q_{s}^i 
\bigl( \varphi(P_{s}^j) 
+  \beta_{s}^{i,j} \bigr)
\Bigr) \ud s
 + \frac12 
\sum_{i \in \ES}
Q_{s}^i
\sum_{j \not = i}
\vert \beta^{i,j}_{s} \vert^2 \ud s
\\
&\hspace{15pt} +
 \frac{{\varepsilon}}{\sqrt 2} 
\sum_{i \in \ES}
Q_{s}^i \sum_{j \in \ES}
\sqrt{\frac{P_{s}^j}{P_{s}^i}}  
d \bigl[ W_{s}^{i,j} - W_{s}^{j,i} \bigr]
\cdot 
\sum_{j,k \in \ES : j \not = k} \nu_{s}^{i,j,k} \ud W_{s}^{j,k}
  +
\ud m_{s},
\end{split}
\end{equation}
where $(m_{s})_{t \leq s \leq T}$ is a uniformly integrable martingale. 
On the last line, the dot in the first term is used to compute the underlying bracket. On the second line,
\begin{equation*}
\begin{split}
\sum_{i \in \ES} u_{s}^i \sum_{j \in \ES} \Bigl( Q_{s}^j \varphi (P_{s}^i) -
Q_{s}^i \varphi (P_{s}^j) \Bigr) =
\sum_{i \in \ES}  \sum_{j \in \ES} Q_{s}^j \varphi (P_{s}^i) \bigl( u_{s}^i - u_{s}^j \bigr)
= \sum_{i \in \ES}  \sum_{j \in \ES} Q_{s}^i \varphi (P_{s}^j) \bigl( u_{s}^j - u_{s}^i \bigr),
\end{split}
\end{equation*}
which cancels out with the first term on the first line. Moreover,
\begin{equation*}
\begin{split}
&-\sum_{i \in \ES} Q_{s}^i   H^i(u_{s})   
+ \sum_{i,j \in \ES} u_{s}^i \bigl( Q_{s}^j \beta_{s}^{j,i} - Q_{s}^i \beta_{s}^{i,j} \bigr) 
+ \frac12 \sum_{i \in \ES}
Q_{s}^i
\sum_{j \not = i}
\vert \beta_{s}^{i,j} \vert^2 
\\
&=\frac12 
\sum_{i \in \ES} Q_{s}^i   
\sum_{j \not =i}
( u_{s}^i - u_{s}^j )_{+}^2
- \sum_{i \in\ES} 
Q_{s}^i
\sum_{j \not = i}
 \beta_{s}^{i,j}
\bigl( u_{s}^i  - u_{s}^j \bigr)
+ \frac12 \sum_{i \in \ES}
Q_{s}^i
\sum_{j \not = i}
\vert \beta_{s}^{i,j} \vert^2 
\\
&\geq 
\frac12 
\sum_{i \in \ES} Q_{s}^i   
\sum_{j \not =i}
( u_{s}^i - u_{s}^j )_{+}^2
- \sum_{i \in \ES} 
Q_{s}^i
\sum_{j \not = i}
 \beta_{s}^{i,j}
\bigl( u_{s}^i  - u_{s}^j \bigr)_{+}
+ \frac12 \sum_{i \in \ES}
Q_{s}^i
\sum_{j \not = i}
\vert \beta_{s}^{i,j} \vert^2
\\
&= \frac12 
\sum_{i \in \ES} Q_{s}^i   
\sum_{j \not =i}
\bigl\vert 
\beta_{s}^{i,j}
-
( u_{s}^i - u_{s}^j )_{+}
\bigr\vert^2,
\end{split}
\end{equation*}
the inequality being in fact an equality if $\beta \equiv \beta^\star$.

 It remains to compute the bracket on the last line of \eqref{eq:4}. We get
 \begin{equation*}
 \begin{split}
&  \frac{{\varepsilon}}{\sqrt 2} 
Q_{s}^i \sum_{j \in \ES}
\sqrt{\frac{P_{s}^j}{P_{s}^i}}  
\ud \bigl[ W_{s}^{i,j} - W_{s}^{j,i} \bigr]
\cdot 
\sum_{j,k \in \ES : j \not =k} \nu_{s}^{i,j,k} \ud W_{s}^{j,k}
 = 
  \frac{{\varepsilon}}{\sqrt 2} 
Q_{s}^i \sum_{j \in \ES : j \not =i}
\sqrt{\frac{P_{s}^j}{P_{s}^i}} 
\Bigl(
\nu_{s}^{i,i,j}
- 
\nu_{s}^{i,j,i} 
\Bigr) \ud s,
\end{split}
 \end{equation*}
which cancels out 
with the last term on the first line of \eqref{eq:4}. 

Integrating from $t$ to $T$ and taking conditional expectation in \eqref{eq:4}, we then deduce that 
\begin{equation}
\label{eq:4:0}
\begin{split}
&\sum_{i \in \ES} Q_{t}^i u_{t}^i 
+
 \frac12 
{\mathbb E} 
\biggl[ \sum_{i \in \ES}
\int_{t}^T Q_{s}^i   
\sum_{j \not =i}
\bigl\vert 
\beta_{s}^{i,j}
-
( u_{s}^i - u_{s}^j )_{+}
\bigr\vert^2 \ud s \, \vert \, {\mathcal F}_{t}^{\boldsymbol W}
\biggr]
\\
&\leq 
{\mathbb E} \biggl[ \sum_{i \in \ES} Q_{T}^i u_{T}^i +\int_{t}^T \sum_{i \in \ES}
Q_{s}^i \bigl( f^i(s,P_{s}) + \frac12 \sum_{j \not = i} \vert \beta^{i,j}_{s} \vert^2
\bigr) \ud s \, \vert \, {\mathcal F}_{t}^{\boldsymbol W} \biggr],
\end{split}
\end{equation}
the inequality being an equality if $\beta \equiv \beta^\star$.
Recalling that 
$Q_{t}^i = \delta_{i,l}$, this is what we want. 
\end{proof}

\begin{proof}[Proof of Lemma \ref{lem:verification:2}.] 
\textit{First Step.}
The first step of the proof is to consider a truncated version of  
\eqref{shjbmfg}. Hence, for a given constant $c>0$, we consider the equation
\begin{equation}
\label{shjbmfg:truncated}
\begin{split}
&\displaystyle \ud u_{t}^i = - \Bigl( 
\sum_{j \in \ES} \varphi(P^j_{t})
\bigl[ u_{t}^j - u_{t}^i \bigr]
+ H^i_{c}(u_{t}) + f^i(t,P_{t}) \Bigr) \ud t 
- \frac{\varepsilon}{\sqrt{2}} \sum_{j \not =i} 
\sqrt{\frac{{P_{t}^j}}{{P_{t}^i}}} \bigl( \nu_{t}^{i,i,j} - \nu_{t}^{i,j,i} \bigr) \ud t
\\
&\hspace{30pt} 
+ \sum_{j \not = k} \nu_{t}^{i,j,k} \ud W_{t}^{j,k}, 
\\
&u^i_T=g^i(P_T),
\end{split}
\end{equation}
where $H^i_{c}$ stands for the truncated Hamiltonian
\begin{equation*}
H^i_{c}(y)  := -\frac12 \sum _{j \in \ES} \bigl[ 
(y_{i}-y_{j})_+^2 {\mathbf 1}_{\{y_{i} - y_{j} \leq c\}}
+ \bigl(    2
c (y_{i} - y_{j}) - c^2 \bigr) {\mathbf 1}_{\{y_{i} - y_{j} > c\}}
\bigr],
\quad y=(y_{j})_{j \in \ES}.
\end{equation*}
Then, \eqref{shjbmfg:truncated} is a backward equation with a time dependent driver that is Lipschitz continuous with respect to 
the entries $(u^{i}_{t})_{i \in \ES}$ and
$(\nu^{i,j,k}_{t})_{i,j,k \in \ES : j \not =k}$, the Lipschitz constant with respect to 
the entries
$(u^{i}_{t})_{i \in \ES}$  being bounded by a deterministic constant $C$ (possibly depending on 
$c$) and the Lipschitz constant with respect to 
$(\nu_{t}^{i,j,k})_{i,j,k \in \ES : j \not = k}$ being bounded by
\begin{equation*}
c_{t} := \frac{\varepsilon}{\sqrt{2}}
 \biggl[ \sum_{i \in \ES} \frac{1}{P_{t}^i} \biggr]^{1/2},
\end{equation*} 
in the sense that (using the fact that the driver is linear in 
$(\nu_{t}^{i,j,k})_{i,j,k \in \ES :j \not = k}$)
\begin{equation*}
\begin{split}
\frac{\varepsilon}{\sqrt{2}}
\biggl( \sum_{i \in \ES}
\biggl[
\sum_{j \not =i} \sqrt{\frac{P_{t}^j}{P_{t}^i}}
\bigl( \nu_{t}^{i,i,j} - \nu_{t}^{i,j,i} \bigr)
\biggr]^{2}
\biggr)^{1/2} 
&\leq
\frac{\varepsilon}{\sqrt{2}}
\biggl(
\sum_{i \in \ES}
\biggl[
\frac{1}{P_{t}^i}
\sum_{j \not =i} 
\bigl( \nu_{t}^{i,i,j} - \nu_{t}^{i,j,i} \bigr)^2
\biggr] \biggr)^{1/2} 
\\
&\leq 
\frac{\varepsilon}{\sqrt{2}}
\biggl[ \sum_{i \in \ES}
\frac{1}{P_{t}^i}
\biggr]^{1/2}
\biggl[
\sum_{i,j \in \ES : j \not =i} 
\bigl( \nu_{t}^{i,i,j} - \nu_{t}^{i,j,i} \bigr)^2
\biggr]^{1/2} 
\\
&= c_{t}
\biggl[
\sum_{i,j \in \ES : j \not =i} 
\bigl( \nu_{t}^{i,i,j} - \nu_{t}^{i,j,i} \bigr)^2
\biggr]^{1/2}.
\end{split}
\end{equation*}
By {Proposition
\ref{expfin} (with $\lambda=2d-1$ and $\gamma \geq 60\varepsilon^2$)
and from the condition $\kappa \geq (61+d) \varepsilon^2$} together with H\"older's inequality, we notice that 
${\mathbb E}
[\exp ( 2 \varepsilon^2 
\sum_{l \in \ES} \int_{0}^T (1/{P_{s}^l})
ds)] < \infty$. Then, by \cite[Theorem 2.1 (i)]{GashiLi} (with, using the notations therein, $\gamma \equiv 1$, $\beta_{1}$ a positive constant, 
$c_{1}$ a non-negative constant, $\beta_{2}=2$ and $c_{2}(t)= c_{t}$), there exists a unique solution to 
\eqref{shjbmfg:truncated} satisfying
\begin{equation*}
\begin{split}
&\sum_{i \in \ES} {\mathbb E} \biggl[ 
\sup_{0 \leq t \leq T}
\biggl( \exp \biggl(  \varepsilon^2 
\sum_{l \in \ES} \int_{0}^t \frac1{P_{s}^l}
\ud s
\biggr)\vert u_{t}^i \vert^2
\biggr) \biggr] <\infty, 
\\
&\sum_{i,j,k \in \ES : j \not = k} {\mathbb E} \biggl[ 
\int_{0}^T
 \exp \biggl(  \varepsilon^2 
\sum_{l \in \ES} \int_{0}^t \frac1{P_{s}^l}
\ud s
\biggr)
\Bigl[ 
\vert \nu_{t}^{i,j,k} \vert^2
+
\Bigl( 1+ 
\sum_{l \in \ES}
\int_{0}^t \frac1{P_{s}^l}
\ud s
\Bigr) 
\vert u_{t}^i \vert^2
\Bigr]
\ud t \biggr] <\infty.
\end{split}
\end{equation*}

\textit{Second Step.}
We now prove that we can find a bound for the solution that is independent of $c$.
To do so, we follow the proof of Lemma
\ref{lem:verification:1}, noticing that the Hamiltonian $H_{c}$ introduced in the first step is associated 
with the same cost functional ${\mathcal J}$ as in 
\eqref{eq:limit:cost:functional:new}
except that the processes $(({\beta}^{i,j}_{t})_{0 \le t \le T})_{i,j \in \ES : i \not = j}$
therein are required to be bounded by $c$, and similarly 
for 
${\mathcal J}^l$ 
in 
 \eqref{mfgvaluef}. In particular,
 $u_{t}^l$ defined in the first step satisfies
 $u_{t}^l=\textrm{\rm ess inf}_{\boldsymbol \beta :  \boldsymbol \beta^{i,j}  \leq c }  
{\mathcal J}^l(t, (\beta_{s})_{t \leq s \leq T}, (P_{s})_{t \leq s \leq T})$.
Call now $((Q_{s}^i)_{i \in \ES})_{t \le s \le T}$ the solution to \eqref{eq:weak:sde:final:2}
with $Q_{t}^i = \delta_{i,l}$ for some $l \in \ES$ and $\beta \equiv 0$. Then, 
by \eqref{eq:4:0} (but with the solution 
$((u_{t}^i)_{i \in \ES})_{0 \le t \le T}$
to 
\eqref{shjbmfg:truncated} and so with the new Hamiltonian)
\begin{equation*}
u_{t}^l \leq 
{\mathbb E} \biggl[ \sum_{i \in \ES} Q_{T}^i u_{T}^i +\int_{t}^T \sum_{i \in \ES}
Q_{s}^i \bigl( f^i(s,P_{s}) + \frac12 \sum_{j \not = i} \vert \beta^{i,j}_{s} \vert^2
\bigr) \ud s \, \vert \, {\mathcal F}_{t}^{\boldsymbol W} \biggr].
\end{equation*}
Here, $\beta \equiv 0$ and $u_{T}^i = g^i(P_{T})$, which provides an upper bound 
for $((u_{t}^l)_{l \in \ES})_{0 \leq t \le T}$, by using \eqref{eq:mean:q:1} and the $L^\infty$ bounds on $f$ and $g$.
Importantly, the upper bound is independent of $c$.
In order to obtain a lower bound, we 
call $((Q_{s}^i)_{i \in \ES})_{t \le s \le T}$ the solution to \eqref{eq:weak:sde:final:2}
with $Q_{t}^i = \delta_{i,l}$ for some $l \in \ES$ 
, given an open-loop strategy ${\boldsymbol \beta}$ whose coordinates are bounded by $c$. Using again the bounds on $f$ and $g$, we get
\[
\begin{split}
{\mathcal J}^l\Bigl(t, (\beta_{s})_{t \leq s \leq T}, (P_{s})_{t \leq s \leq T}\Bigr)&=
{\mathbb E} \biggl[ \sum_{i \in \ES} Q_{T}^i u_{T}^i +\int_{t}^T \sum_{i \in \ES}
Q_{s}^i \bigl( f^i(s,P_{s}) + \frac12 \sum_{j \not = i} \vert \beta^{i,j}_{s} \vert^2
\bigr) \ud s \, \vert \, {\mathcal F}_{t}^{\boldsymbol W} \biggr]\\
&\geq{\mathbb E} \biggl[ -\sum_{i \in \ES} Q_{T}^i ||g^i||_\infty -\int_{t}^T \sum_{i \in \ES}
Q_{s}^i || f^i||_\infty \ud s \, \vert \, {\mathcal F}_{t}^{\boldsymbol W} \biggr]
\geq -C_0,
\end{split}
\]
by \eqref{eq:mean:q:1}, for a constant $C_0$ independent of $c$ and ${\boldsymbol \beta}$; 
so $u_{t}^l \geq -C_0$.


In the end, we may find a constant $C_{0}$ such that, whatever the value of 
$c$ in \eqref{shjbmfg:truncated}, 
the solution is bounded by $C_{0}$. We deduce that, whenever $c \geq 2 C_{0}$, 
the solution of 
\eqref{shjbmfg:truncated}
is also a solution 
of the backward equation 
\eqref{shjbmfg}.
This proves the existence of a bounded solution 
 to 
\eqref{shjbmfg}.
As for uniqueness, it suffices to notice that a bounded solution to
\eqref{shjbmfg} is also a solution to 
\eqref{shjbmfg:truncated}, but for a large enough $c$ inside. Hence, we get that any bounded solution to 
\eqref{shjbmfg} is bounded by $C_{0}$, which shows uniqueness.
\end{proof}

\subsection{Proof of Theorem \ref{main:thm}: Existence and uniqueness of an MFG equilibrium}
\label{subse:MFG:proof:existence:!}
The system \eqref{eq:MFG:system:epsilon}
is what we call a forward-backward stochastic differential equation. But, differently from most of the cases that have been addressed so far in the literature (see for instance \cite[Chapter 3]{CarmonaDelarue_book_I}), solutions to the forward equation are here 
regarded as processes with values in ${\mathcal S}_{d-1}$. This  requires a special treatment, {which is the main rationale of our paper. Whatever the setting,  a standard strategy for solving forward-backward stochastic differential equations (at least in the so-called Markovian case, see for instance 
\cite{MaProtterYong,Delarue02})
is to regard the system formed by the two forward and backward equations as the characteristics of 
a system of parabolic second order PDEs. In our framework, this system of PDEs is precisely 
the master equation
\eqref{eq:master:equation:local}--\eqref{eq:master:equation:intrinsic}.
Therein, the unknown is the tuple of functions $(U^i)_{i \in \ES}$, each $U^i$ standing for a real valued function defined on $[0,T] \times {\mathcal S}_{d-1}$. Accordingly, the master equation is formally obtained by imposing $u^i_t=U^i(t,P_t)$ and expanding $U$ using It\^o formula; this is basically what we do later for proving Theorem \ref{lem:existence:mfg:from:master}.}

{The precise connection between 
the master equation
and the 
MFG system 
\eqref{eq:MFG:system:epsilon} is  given by the following statement, which is in fact a general feature of MFGs: Once the master equation is known to have a unique smooth solution, the MFG should be uniquely solvable.}

\begin{theorem}
\label{lem:existence:mfg:from:master}
Assume that there exists a $d$-tuple $(U^1,\cdots,U^d)$ of real valued functions 
defined on $[0,T] \times {\mathcal S}_{d-1}$ such that, for any 
$i \in \ES$, $U^i$ belongs to the Wright--Fisher space 
${\mathscr C}_{\textrm{\rm WF}}^{1+\gamma'/2,2+\gamma'}([0,T] \times {\mathcal S}_{d-1})$ 
for some $\gamma' >0$ (see 
\S \ref{subse:space:Wright--Fisher}
for the definition), and for any $(t,p) \in [0,T] \times \textrm{\rm Int}(\hat{\mathcal S}_{d-1})$, 
equation \eqref{eq:master:equation:local}
holds at $(t,p)$.
Then, 
for any (deterministic) initial condition $(p_{0,i})_{i \in \ES} \in {\mathcal S}_{d-1}$, 
with $p_{0,i}>0$ for any $i \in \ES$, the 
MFG system 
\eqref{eq:MFG:system:epsilon} has a unique solution
$(P_{t}=(P_{t}^i)_{i \in \ES},u_{t}=(u_{t}^i)_{i \in \ES},\nu_{t}=(\nu_{t}^{i,j,k})_{i,j,k \in \ES : j \not =k})_{0 \leq t \leq T}$ in the class 
of ${\mathbb F}^{\boldsymbol W}$-progressively-measurable processes $(\widetilde P_{t}=(\widetilde P_{t}^i)_{i \in \ES},\widetilde u_{t}=(\widetilde u_{t}^i)_{i \in \ES},\widetilde \nu_{t}=(\widetilde \nu_{t}^{i,j,k})_{i,j,k \in \ES : j \not =k})_{0 \leq t \leq T}$ such that 
$(\widetilde P_{t})_{0 \leq t \leq T}$ is continuous 
and takes values in ${\mathcal S}_{d-1}$, $(\widetilde u_{t})_{0 \leq t \leq T}$ is continuous and 
is bounded by a deterministic constant
and 
$(\widetilde \nu_{t})_{0 \leq t \leq T}$ satisfies 
${\mathbb E}[ \int_{0}^T 
 \exp (  \varepsilon^2
\sum_{l \in \ES} \int_{0}^t (1/{P_{s}^l})
\ud s)
\vert \widetilde \nu_{t} \vert^2 \ud t ] < \infty$. The solution satisfies,
$\ud{\mathbb P}$ almost surely, for all $t \in [0,T]$ and all $i \in \ES$, $u^i_t= U^i(t,P_t)$,
and, 
$\ud{\mathbb P} \otimes \ud t$ almost everywhere,
for all $i,j,k \in \ES$, 
with $j \not =k$, 
$\nu_t^{i,j,k} = V^{i,j,k}(t,P_t)$, where
\[
V^{i,j,k}(t,p):= \frac{\varepsilon}{\sqrt{2}} \left({\mathfrak d}_{p_j}U^i(t,p)-{\mathfrak d}_{p_k} U^i(t,p)\right) \sqrt{p_j p_k}.
\] 
\end{theorem}


We notice that Theorem \ref{main:thm} from the above theorem, thanks to what we discussed in \S \ref{subver}.
%

\begin{proof} 
The proof is inspired by  
\cite{MaProtterYong}, but the fact that the master equation is set on the simplex makes it more difficult. 
Also, we recall that  $\kappa \geq (61+d) \varepsilon^2$.
\vskip 4pt

\textit{First Step.}
In order to prove the existence of a solution, one may first solve the SDE
\begin{align}
&\ud P_{t}^i =  \sum_{j \in \ES} \Bigl( P_{t}^j \bigl[ \varphi(P_{t}^i) + \bigl(U^j(t,P_t)-U^i(t,P_{t})\bigr)_+ \bigr] - P_{t}^i \bigl[ \varphi(P_{t}^j) 
+ 
 \bigl(U^i(t,P_t)-U^j(t,P_{t})\bigr)_+
 \bigr]
\Bigr) \ud t \nonumber
\\
&\hspace{30pt}
  + {\varepsilon} 
\sum_{j \in \ES : j \not = i}
\sqrt{P_{t}^i P_{t}^j}  
\ud \overline W_{t}^{i,j}, \quad t \in [0,T],
\label{eq:SDE:epsilon}
\end{align}
with $p_{0}=(p_{0,i})_{i \in \ES}$ as initial condition. 
Solvability is a mere consequence of 
Proposition 
\ref{thm:approximation:diffusion:2}. 

Then, it suffices to let
$u^i_t:= U^i(t,P_t)$
and 
$\nu_t^{i,j,k} := V^{i,j,k}(t,P_t)$,
for $t \in [0,T]$ and $i,j,k \in \ES$, with $j \not =k$. 
By It\^o's formula (see the next step if needed for the details), we may easily expand $(u_{t}^i)_{0 \leq t \leq T}$
and check that it solves the backward equation in 
\eqref{eq:MFG:system:epsilon}. (The fact that the second-order derivatives are just defined on the interior of simplex is not a hindrance since $(P_{t})_{0 \le t \le T}$ does not touch the boundary, see Proposition 
\ref{thm:approximation:diffusion:2}.)
Obviously, the processes 
$((u^i_{t})_{0 \le t \le T})_{i \in \ES}$
and
$((\nu^{i,j,k}_{t})_{0 \le t \le T})_{i,j,k \in \ES : j \not = k}$
are bounded and hence satisfy the required growth conditions. 
\vspace{4pt}

\textit{Second Step.}
Consider now another solution, say 
$$\Bigl(\bigl((\widetilde P_{t}^i)_{0 \leq t \leq T}\bigr)_{i \in \ES},
\bigl((\widetilde u_{t}^i)_{0 \leq t \leq T}\bigr)_{i \in \ES},\bigl((\widetilde \nu_{t}^{i,j,k})_{0 \leq t \leq T}\bigr)_{i,j,k \in \ES : j \not = k}\Bigr)$$
to \eqref{eq:MFG:system:epsilon}. 
%
%
%
%
We denote $\widetilde U^i_t := U^i(t,\widetilde P_t)$, 
$\partial_{p_{j}}\widetilde U^i_t := {\mathfrak d}_{p_j} U^i(t,\widetilde P_t)$,
$\partial_{p_{j}p_{k}}^2 \widetilde U^i_t := {\mathfrak d}_{p_jp_k}^2 U^i (t,\widetilde P_t)$,
for $t \in [0,T]$ and $i,j,k \in \ES$, $j \not = k$. Thanks to the fact that $U^i \in \mathcal{C}^{1,2}([0,T] \times \mathrm{Int}(\hat{\mathcal S}_{d-1}))$, 
we can apply It\^o's formula to $(\widehat {\{ U^i\}}{}^{i}(t,\widetilde P_{t}))_{0 \le t \le T}$ (which obviously coincides with $(\widetilde U^i_{t})_{0 \le t \le T}$).
We get (the computations of the various intrinsic derivatives that appear in the expansion are similar to those in
\eqref{eq:master:equation:intrinsic})
\begin{align*}
 \ud \widetilde{U}^i_t 
&=  \bigg\{  \partial_t \widetilde{U}^i_t +\frac{\varepsilon^2}{2} \sum_{j,k \in \ES}\bigl(\widetilde P^j_t \delta_{jk}- \widetilde P^j_t \widetilde P^k_t\bigr) 
 \partial^2_{p_{j} p_{k}} \widetilde{U}^i_t \\
&\qquad+ \sum_{j,k \in \ES} \widetilde P^k_t \left[\varphi\bigl(\widetilde P^j_t\bigr) +(\widetilde U^k_t-\widetilde U^j_t)_+ \right]\left( \partial_{p_{j}} \widetilde U^i_t - \partial_{p_{k}} \widetilde U^i_t  \right)\bigg\}\ud t
\\ 
&\qquad + \frac{\varepsilon}{\sqrt{2}} \sum_{j, k \in \ES : j \not =k}  
	\sqrt{\widetilde P_{t}^j \widetilde P_{t}^k}  \left( \partial_{p_{j}} \widetilde U^i_t - \partial_{p_{k}} \widetilde U^i_t
	 \right) \ud W_{t}^{j,k} \\
	&= -\bigg\{  H^i\bigl(\widetilde{U}_t\bigr) 
	+ f^i\bigl(t,\widetilde P_{t}\bigr) 
	+
\sum_{j \in \ES} \varphi\bigl(\widetilde P_{t}^j\bigr)
\bigl[ \widetilde U_{t}^j - \widetilde U_{t}^i \bigr]
	\\
	&\qquad +\varepsilon^2\sum_{j \in \ES} \widetilde P^j_t \left( \partial_{p_{i}} \widetilde U^i_t - \partial_{p_{j}} \widetilde U^i_t 
	 \right)\bigg\}\ud t + \frac{\varepsilon}{\sqrt{2}} \sum_{j, k \in \ES : j \not =k}  
	\sqrt{\widetilde P_{t}^j \widetilde P_{t}^k}  \left( \partial_{p_{j}} \widetilde U^i_t - \partial_{p_{k}} \widetilde U^i_t
	  \right) \ud W_{t}^{j,k}, 
\end{align*}
where in the last equality we used the equation \eqref{eq:master:equation:intrinsic} satisfied by $U$.
This prompts us to let 
\begin{equation*}
\widetilde V_{t}^{i,j,k} = \frac{\varepsilon}{\sqrt 2} \sqrt{\widetilde P_{t}^j \widetilde P_{t}^k}  \left( \partial_{p_{j}} \widetilde U^i_t - \partial_{p_{k}} \widetilde U^i_t
	  \right), \quad t \in [0,T], \quad i,j,k \in \ES, \quad j \not =k. 
\end{equation*}
Subtracting the equation satisfied by $((\widetilde u_{t}^i)_{i \in \ES})_{0 \le t \le T}$, we get 
\begin{align*}
\ud \bigl( 
\widetilde{U}^i_t 
- \widetilde u_{t}^i \bigr) 
&= -\Bigg\{  H^i\bigl(\widetilde{U}_t\bigr) - H^i\bigl( \widetilde u_{t} \bigr) 
+ \sum_{j \in \ES} \varphi\bigl(\widetilde P_{t}^j\bigr)
\bigl[ \widetilde U_{t}^j - \widetilde u_{t}^j- \bigl( \widetilde U_{t}^i - \widetilde u_{t}^i \bigr) \bigr]
	\\
	&\qquad + \frac{\varepsilon}{\sqrt 2}\sum_{j \in \ES} \sqrt{\frac{\widetilde P^j_t}{\widetilde P^i_t}} \Bigl( 
	\widetilde V_{t}^{i,i,j} - \widetilde \nu_{t}^{i,i,j}-
	\bigl( \widetilde V_{t}^{i,j,i} - \widetilde \nu_{t}^{i,j,i}
	\bigr)
		 \Bigr)\Bigg\}\ud t
		 \\
&\qquad		  +   \sum_{j, k \in \ES : j \not = k}  
\bigl( \widetilde{V}_{t}^{i,j,k} - \widetilde \nu_{t}^{i,j,k} \bigr)
 \ud W_{t}^{j,k}, \quad t \in [0,T], \quad i \in \ES.
\end{align*}
Consider now
\begin{equation*}
e_{t}:= \exp \biggl( \varepsilon^2 \int_{0}^t \sum_{j \in \ES} \frac{1}{\widetilde P^j_{s}} \ud s \biggr), \quad t \in [0,T]. 
\end{equation*}
Then, by It\^o's formula, 
we obtain, for any $t \in [0,T]$,  
\begin{align}
 &e_{t}
 |\widetilde U^i_t - \widetilde{u}^i_t|^2    
 + \int_t^T \varepsilon^2 e_{s}
 |\widetilde U^i_s - \widetilde{u}^i_s|^2 
  \biggl(\sum_{j \in \ES} \frac{1}{ \widetilde P^j_s}\biggr) \ud s
 +  \int_t^T e_{s} \sum_{j, k \in \ES : j \not =k}  
  \left| 
 \widetilde V_{s}^{i,j,k} - \widetilde \nu_{s}^{i,j,k}  \right|^2   \ud s \nonumber
 \\
	&= 2 \int_t^T e_{s} \bigl( \widetilde U^i_s - \widetilde{u}^i_s\bigr)
	 \Bigg\{ 
	 H^i\bigl(\widetilde U_s\bigr) -H^i\bigl(\widetilde{u}_s\bigr) 
+ \sum_{j\in \ES} \varphi\bigl(\widetilde P_{s}^j\bigr)
\bigl[ \widetilde U_{s}^j - \widetilde u_{s}^j- \bigl( \widetilde U_{s}^i - \widetilde u_{s}^i \bigr) \bigr] \nonumber
	\\
	&\qquad + 
	  \frac{\varepsilon}{\sqrt 2}\sum_{j \in \ES} \sqrt{\frac{\widetilde P^j_s}{\widetilde P^i_s}} \Bigl( 
	\widetilde V_{s}^{i,i,j} - \widetilde \nu_{s}^{i,i,j}-
	\bigl( \widetilde V_{s}^{i,j,i} - \widetilde \nu_{s}^{i,j,i}
	\bigr)
		 \Bigr)
	 \Bigg\} \ud s \label{eq:e_{t}}
	 \\
&\quad	 + 2 \int_t^T e_{s}
\bigl( 
\widetilde{U}^i_t 
- \widetilde u_{t}^i \bigr)
	   \sum_{j, k \in \ES : j \not = k}  
\bigl( \widetilde{V}_{s}^{i,j,k} - \widetilde \nu_{s}^{i,j,k} \bigr)
 \ud W_{s}^{j,k}. \nonumber
\end{align} 
By Proposition 
\ref{expfin} (together with Remark \ref{foot:extension:alpha:open loop}),
with $\lambda=2d-1$, 
$\gamma \geq 60 \varepsilon^2$ and $\kappa \geq (61+d) \varepsilon^2$ therein, 
and by H\"older's inequality, ${\mathbb E}[e_{T}^4]$
is finite. 
Since 
$((\widetilde U_{t}^i)_{i \in \ES})_{0 \le t \le T}$
and 
$((\widetilde u_{t}^i)_{i \in \ES})_{0 \le t \le T}$
are bounded (by deterministic constants)
and
$((\widetilde V_{t}^{i,j,k})_{i,j,k \in \ES : j \not = k})_{0 \le t \le T}$
and 
$((\widetilde u_{t}^{i,j,k})_{i \in \ES : j \not = k})_{0 \le t \le T}$
are square-integrable, all the terms in the right-hand side have integrable
sup norm (over $t \in [0,T]$); as for the last term in the right-hand side, the latter 
follows from Burkholder--Davis--Gundy inequalities. 
Also, we can treat the difference 
$H^i(\widetilde U_s) -H^i(\widetilde{u}_s) $
as a Lipschitz difference, since $U$ and $\tilde{u}$ are bounded. Hence, taking expectations and applying Young's inequality, we can find a constant $C$ such that 
\begingroup
\allowdisplaybreaks
\begin{align*}
&\E\biggl[ e_{t}
 |\widetilde U^i_t - \widetilde{u}^i_t|^2    
 + \int_t^T  \varepsilon^2 e_{s}
 |\widetilde U^i_s - \widetilde{u}^i_s|^2 
  \biggl(\sum_{j \in \ES} \frac{1}{\widetilde P^j_s}\biggr) \ud s
 +  \int_t^T e_{s} \sum_{j, k \in \ES : j \not = k}  
  \left| 
 \widetilde V_{s}^{i,j,k} - \widetilde \nu_{s}^{i,j,k}  \right|^2   \ud s  \biggr]
 \\
	&\leq C \sum_{j \in \ES} \E\biggl[ \int_t^T e_{s} |\widetilde U_s^{j} - \widetilde{u}_s^{j}|^2   \ud s\biggr] 
	+ \varepsilon^2
	 \E \biggl[\int_t^T \frac{e_{s}}{\widetilde P^i_s}|\widetilde U^i_s - \widetilde{u}^i_s|^2  \ud s\biggr]
	   \\
&\hspace{15pt}	   +   {\mathbb E} \biggl[\int_t^T e_{s} \sum_{j, k \in \ES : j \not = k}  
  \left| 
 \widetilde V_{s}^{i,j,k} - \widetilde \nu_{s}^{i,j,k}  \right|^2   \ud s\biggr].
\end{align*}
\endgroup
 We obtain, 
 \[
 \sum_{j \in \ES}
\E\left[ e_{t} |\widetilde U_t^{j} - \widetilde{u}_t^{j}|^{{2}} \right] 
 \leq C  \sum_{j \in \ES} \int_t^T \E\left[ e_{s} |\widetilde U_s^{j} - \widetilde{u}_s^{j}|^2   \right] \ud s,
 \]
 and thus Gronwall's lemma yields, for any $i \in \ES$ and any $t\in [0,T]$,
 \[
\mathbb{P}\left( \widetilde u_{t}^i = \widetilde{U}^i_{t}=U^i(t,\widetilde P_t)\right)=1. 
 \]
This permits to identify $(\widetilde P_{t})_{0 \le t \le T}$ with the solution of
\eqref{eq:SDE:epsilon}. It is then pretty straightforward to show that 
$(\widetilde u_{t}^i)_{0 \leq t \leq T}$
coincides with 
$( u_{t}^i)_{0 \leq t \leq T}$, for each $i \in \ES$,
and then that $(\widetilde \nu_{t}^{i,j,k})_{0 \leq t \leq T}$
coincides with 
$( \nu_{t}^{i,j,k})_{0 \leq t \leq T}$, for each $i,j,k \in \ES$, $j \not = k$.
\end{proof}
\subsection{Proof of Theorem \ref{existence:master}: Unique solvability of the master equation}
\label{subse:master:e!}

{In order to establish the unique solvability of the master equation 
\eqref{eq:master:equation:local}--\eqref{eq:master:equation:intrinsic},}
the first point is to observe that it may be rewritten in a somewhat generic form. Indeed, for a given 
coordinate $i \in \ES$, we may let
\begin{equation}
\label{eq:B:F}
\begin{split}
&B^i_{j}(t,p,y) := \varphi(p_{j}) + \sum_{k \in \ES} p_{k} \bigl( y_{k} - y_{j} \bigr)_{+}
- p_{j} \sum_{k \in \ES} \bigl[ \varphi(p_{k}) + \bigl( y_{j} - y_{k} \bigr)_{+} \bigr]
+  {\varepsilon^2 \bigl( \delta_{i,j} -  p_{j}\bigr)},
\\
&F^i(t,p,y) := H^i(y) + f^i(t,p) + \sum_{k \in \ES} \varphi(p_{j})\bigl( y_{j} - y_{i} \bigr),
\end{split}
\end{equation}
where $t \in [0,T]$, $p \in {\mathcal S}_{d-1}$ and $y=(y_{k})_{k \in \ES} \in \RR^d$.
{Recalling  
\eqref{eq:Ha}}, we may rewrite 
\eqref{eq:master:equation:intrinsic} as
\begin{equation}
\label{eq:master:equation:intrinsic:reformulation}
\begin{split}
&\partial_t U^i(t,p) + F^i\bigl(t,p,U(t,p)\bigr) 
+
\sum_{j \in \ES} B^i_{j}\bigl(t,p,U(t,p)\bigr) \fd_{p_{j}} U^i(t,p) 
\\
&\hspace{30pt}+\frac{\varepsilon^2}{2} \sum_{j,k \in \ES}(p_j \delta_{jk}-p_{j} p_{k}) \fd^2_{p_{j} p_{k}} U^i(t,p) =0,\\
&U^i(T,p)= g^i(p),
\end{split}
\end{equation}
for $t \in [0,T]$ and $p \in \textrm{\rm Int}( {\mathcal S}_{d-1})$, 
with the shorten notation $U(t,p)=(U^i(t,p))_{i \in \ES}$. For sure, we could write 
\eqref{eq:master:equation:local} in a similar form. 
In fact, what really matters is that 
\begin{equation*}
\sum_{j \in \ES} B^i_{j}(t,p,y) =0,
\end{equation*}
for any $i \in \ES$, $t \in [0,T]$, $p \in \textrm{\rm Int}(\hat {\mathcal S}_{d-1})$
and $y =(y_{i})_{i \in \ES} \in {\mathbb R}^d$, and that 
$B^i_{j}(t,p,y) >0$  
whenever $p_{j}=0$.

In the sequel, solvability of \eqref{eq:master:equation:intrinsic:reformulation} is addressed  in several steps. The first one
is to address the solvability of the linear version of 
\eqref{eq:master:equation:intrinsic:reformulation} obtained by freezing the nonlinear component 
$U$ in $B^i$ and $F^i$; as we make it clear below, this mostly follows from the earlier results of \cite{EpsteinMazzeo}. 
The second one is to prove a priori estimates for the solutions to the latter linear version 
independently of the nonlinear component $U$ that is frozen in the coefficients $B^i$ and $F^i$;
{this is where we invoke Theorem \ref{main:holder}}. The last step is to deduce the existence of a classical solution to
the master equation by means of Schauder's fixed point theorem. 

\subsubsection{Linear version}
{We first establish a preliminary solvability result for the linear version 
\eqref{eq:pde:apriori:1:intro}, keeping in mind that it should be
reformulated in intrinsic derivatives as follows:}
\begin{equation}
\label{eq:pde:apriori:1}
\begin{split}
&\partial_{t} u(t,p) + \sum_{j \in \ES} \Bigl( \varphi(p_{j}) + b_{j}(t,p) + p_{j} b_{j}^{\circ}(t,p)  \Bigr) \fd_{p_{j}} u(t,p)
\\ 
&\hspace{15pt} +\frac{\varepsilon^2}{2} \sum_{j,k \in \ES}\bigl(p_j \delta_{jk}-p_{j} p_{k}\bigr) \fd^2_{p_{j} p_{k}} u(t,p)
+ h(t,p) = 0,
\\
&u(T,p) = \ell(p),
\end{split}
\end{equation}
for $(t,p) \in [0,T] \times \mathrm{Int}(\hat{{\mathcal S}}_{d-1})$ {($\mathrm{Int}(\hat{{\mathcal S}}_{d-1})$ being here regarded as a subset of ${\mathcal S}_{d-1}$)}, 
{where we recall that}
$b=(b_{j})_{j \in \ES} : [0,T] \times {\mathcal S}_{d-1} \rightarrow (\RR_{+})^d$, $b^{\circ} =(b_{j}^{\circ})_{j \in \ES} : 
[0,T] \times {\mathcal S}_{d-1} \rightarrow {\mathbb R}^{d}$,
 $h : [0,T] \times {\mathcal S}_{d-1} \rightarrow \RR$ 
 and
 $\ell : {\mathcal S}_{d-1} \rightarrow \RR$ 
 are bounded and satisfy
\begin{equation}
\label{eq:zero:sum}
\begin{split}
&\sum_{j \in \ES} \Bigl( \varphi(p_{j}) + b_{j}(t,p) + p_{j} b_j^{\circ}(t,p) \Bigr)= 0, \quad t \in [0,T], \quad p \in {\mathcal S}_{d-1}. 
\end{split}
\end{equation}
Our solvability result is 
\begin{lemma}
\label{lem:4.6}
Assume that the functions $(b_{j})_{j \in \ES}$ and $(b^{\circ}_{j})_{j \in \ES}$ 
are in 
${\mathscr C}_{\textrm{\rm WF}}^{\eta/2,\eta}([0,T] \times {\mathcal S}_{d-1})$, 
that $h$ is in 
${\mathscr C}_{\textrm{\rm WF}}^{\eta/4,\eta/2}([0,T] \times {\mathcal S}_{d-1})$
and $\ell$  
is 
in 
${\mathscr C}_{\textrm{\rm WF}}^{2+\eta/2}({\mathcal S}_{d-1})$, for some 
$\eta \in (0,1)$. Then, 
equation 
\eqref{eq:pde:apriori:1} has a unique classical solution in the space 
${\mathscr C}_{\textrm{\rm WF}}^{1+\eta/4,2+\eta/2}([0,T] \times {\mathcal S}_{d-1})$.
\end{lemma}

\begin{proof}
For a fixed $t_{0} \in [0,T]$, we rewrite 
\eqref{eq:pde:apriori:1}
in the form 
\begin{align}
&\partial_{t} u(t,p) +
\sum_{j \in \ES} 
\Bigl( \varphi(p_{j}) + b_{j}(t_{0},p) + p_{j} b_{j}^{\circ}(t_{0},p) \Bigr) \fd_{p_{j}} u(t,p) 
 +\frac{\varepsilon^2}{2} \sum_{ j,k \in \ES}\bigl(p_j \delta_{jk}-p_{j} p_{k}\bigr) \fd^2_{p_{j} p_{k}} u(t,p)
 \nonumber
\\
&\hspace{15pt}
+ 
\sum_{j \in \ES}
\Bigl( \bigl(b_{j}(t,p) + p_{j} b_{j}^{\circ}(t,p) \bigr)
- 
\bigl(b_{j}(t_{0},p) + p_{j} b_{j}^{\circ}(t_{0},p)\bigr) \Bigr)  \fd_{p_{j}} u(t,p) 
+ h(t,p) = 0,
\label{eq:pde:apriori:2}
\\
&u(T,p) = \ell(p), \nonumber
\end{align}
for $(t,p) \in [0,T] \times \textrm{\rm Int}(\hat{\mathcal S}_{d-1})$. 
Our first goal is to solve the equation on $[t_{0},T] \times \mathrm{Int}(\hat{\mathcal S}_{d-1})$ provided 
$t_{0}$ is chosen close enough to $T$. 

In order to solve the above equation, we define the following mapping. For a vector-valued
function $w=(w_{j})_{j \in \ES} : [0,T] \times {\mathcal S}_{d-1} \rightarrow \RR^d$ whose components are 
in ${\mathscr C}_{\textrm{\rm WF}}^{\eta/4,\eta/2}([0,T] \times {\mathcal S}_{d-1})$, 
 we call 
$v$ the solution of the equation 
\begin{equation*}
\begin{split}
&\partial_{t} v(t,p) + 
\sum_{j \in \ES} 
\Bigl( \varphi(p_{j}) + b_{j}(t_{0},p) + p_{j} b_{j}^{\circ}(t_{0},p) \Bigr) \fd_{p_{j}} v(t,p) 
 +\frac{\varepsilon^2}{2} \sum_{j,k \in \ES}(p_j \delta_{jk}-p_{j} p_{k}) \fd^2_{p_{j} p_{k}} v
\\
&\hspace{15pt}
+
\sum_{j \in \ES}
\Bigl( \bigl(b_{j}(t,p) + p_{j} b_{j}^{\circ}(t,p) \bigr)
- 
\bigl(b_{j}(t_{0},p) + p_{j} b_{j}^{\circ}(t_{0},p)\bigr) \Bigr)  w_{j}(t,p) 
+ h(t,p) = 0,
\\
&v(T,p) = \ell(p),
\end{split}
\end{equation*}
for $(t,p) \in [t_{0},T] \times \textrm{\rm Int}(\hat{\mathcal S}_{d-1})$, the solution being known, by \cite[Theorem 10.0.2]{EpsteinMazzeo}, to exist and to satisfy
\begin{equation}
\label{schau}
\| v \|_{1+\eta/4,2+\eta/2;[t_{0},T])} \leq C \Bigl( \| \ell \|_{2+\eta/2}
+ \|  
W \|_{\eta/4,\eta/2;[t_{0},T]}
+
\|  
h \|_{\eta/4,\eta/2;[t_{0},T]} \Bigr),
\end{equation} 
where we added the notation $[t_{0},T]$ 
in the Wright--Fisher norm in order to emphasize the fact that the underlying domain is 
$[t_{0},T] \times {\mathcal S}_{d-1}$
and not $[0,T] \times {\mathcal S}_{d-1}$, and with
\begin{equation*}
W(t,p) :=
\sum_{j \in \ES}
\Bigl( \bigl(b_{j}(t,p) + p_{j} b_{j}^{\circ}(t,p) \bigr)
- 
\bigl(b_{j}(t_{0},p) + p_{j} b_{j}^{\circ}(t_{0},p)\bigr) \Bigr)  w_{j}(t,p).
\end{equation*}
Clearly, we can find a universal constant $c>0$ such that 
\begin{equation*}
\begin{split}
&\bigl\|  W \bigr\|_{\eta/4,\eta/2;[t_{0},T]}
\leq 
c 
\sum_{j \in \ES}
\bigl\| 
B^{\circ}_{j}  \bigr\|_{\eta/4,\eta/2;[t_{0},T]}
 \| 
 w_{j}  \|_{\eta/4,\eta/2;[t_{0},T]},
\end{split}
\end{equation*}
with
\begin{equation*}
B_{j}^{\circ}(t,p) : = b_{j}(t,p) + p_{j}b_{j}^{\circ}(t,p)
- 
\Bigl( b_{j}(t_{0},p) + p_{j} b_{j}^{\circ}(t_{0},p)\Bigr), \quad (t,p) \in [t_{0},T] \times {\mathcal S}_{d-1}. 
\end{equation*}
Now, we can find a constant $C$, only depending on the Wright--Fisher norm of $b=(b_{j})_{j \in \ES}$ such that, for any 
$s,t \in [t_{0},T]$ and any $p,q \in {\mathcal S}_{d-1}$,  
\begin{equation*}
\begin{split}
&\bigl\vert b(t,p) - b(t_{0},p) - \bigl( b(s,q) - b(t_{0},q) \bigr) \bigr\vert
\\
&= 
\bigl\vert  b(t,p) - b(t_{0},p) - \bigl( b(s,q) - b(t_{0},q) \bigr) \bigr\vert^{1/2}
\bigl\vert b(t,p) - b(t_{0},p) - \bigl( b(s,q) - b(t_{0},q) \bigr)
 \bigr\vert^{1/2}
\\
&\leq C \Bigl( 
\bigl\vert b(t,p) - b(t_{0},p) \bigr\vert + \bigl\vert b(s,q) - b(t_{0},q) \bigr\vert
\Bigr)^{1/2} 
\Bigl(
\bigl\vert b(t,p) -  b(s,q) \bigr\vert + \bigl\vert b(t_{0},p)- b(t_{0},q)  \bigr\vert
\Bigr)^{1/2}
\\
&\leq C (T-t_{0})^{\eta/4}
\Bigl( \vert t- s \vert^{\eta/4} + \vert \sqrt{p}-\sqrt{q} \vert^{\eta/2} \Bigr),
\end{split}
\end{equation*}
which shows that the Wright--Fisher H\"older norm (of exponents $(\eta/4,\eta/2)$) of $b-b(t_{0},\cdot)$ is small with $T-t_{0}$ (it is easy to see that sup norm is small
with $T-t_{0}$). Proceeding in a similar way with the other functions entering the definition of 
$B^{\circ}$, 
we deduce that the Wright--Fisher H\"older norm (of exponent $\eta/2$) of $B^{\circ}$ is small with $T-t_{0}$. 

Therefore, Schauder's estimates \eqref{schau} imply that, for $T-t_{0}$ small enough 
\begin{equation*}
\| \fd_{p} v \|_{\eta/4,\eta/2;[t_{0},T]} \leq C \Bigl( 
\| \ell \|_{2+\eta/2;[t_{0},T]}
+
\|  
h \|_{\eta/4,\eta/2;[t_{0},T]}
\Bigr) + \frac12
 \bigl\| 
 w \bigr\|_{\eta/4,\eta/2;[t_{0},T]},
\end{equation*} 
for a constant $C$ which is independent of $w$ and $t_0$. 
This shows in particular that 
$\| \fd_{p} v \|_{\eta/4,\eta/2;[t_{0},T]} \leq 2 C 
(\| \ell \|_{2+\eta,[t_{0},T]} + 
\|  
h \|_{\eta/4,\eta/2;[t_{0},T]})
$ whenever 
$\| w \|_{\eta/4,\eta/2;[t_{0},T]} \leq 
2 C 
(\| \ell \|_{2+\eta,[t_{0},T]} + 
\|  
h \|_{\eta/4,\eta/2;[t_{0},T]})
$. In particular, the map $w \mapsto v$ preserves 
a closed ball of $[{\mathscr C}_{\textrm{\rm WF}}^{\eta/4,\eta/2}([t_{0},T] \times {\mathcal S}_{d-1})]^d$. By linearity, the map $w \mapsto v$ is obviously continuous from
$[{\mathscr C}_{\textrm{\rm WF}}^{\eta'/4,\eta'/2}([t_{0},T] \times {\mathcal S}_{d-1})]^d$ into itself, for any $\eta' \in (0,\eta]$. By Schauder's theorem (regarding 
any closed ball of 
$[{\mathscr C}_{\textrm{\rm WF}}^{\eta/4,\eta/2}([t_{0},T] \times {\mathcal S}_{d-1})]^d$
as a compact subset of 
$[{\mathscr C}_{\textrm{\rm WF}}^{\eta'/4,\eta'/2}([t_{0},T] \times {\mathcal S}_{d-1})]^d$, for $\eta' \in (0,\eta)$), we deduce that there exists a solution $v$ to 
\eqref{eq:pde:apriori:2} (and hence to 
\eqref{eq:pde:apriori:1})
in ${\mathscr C}_{\textrm{\rm WF}}^{1+\eta/4,2+\eta/2}([t_{0},T] \times {\mathcal S}_{d-1})$ (so on 
$[t_{0},T] \times {\mathcal S}_{d-1}$).
By iterating in time, we deduce that
there exists a solution 
to 
\eqref{eq:pde:apriori:1}
on the entire 
$[{0},T] \times {\mathcal S}_{d-1}$
in the space 
${\mathscr C}_{\textrm{\rm WF}}^{1+\eta/4,2+\eta/2}([{0},T] \times {\mathcal S}_{d-1})$.

Uniqueness follows 
from a straightforward application of Kolmogorov representation formula, see Proposition 
\ref{prop:ito:formula} if needed.
\end{proof}

\subsubsection{Fixed point argument via Schauder's theorem}

Here is now the last step of our proof of Theorem 
\ref{existence:master}, which strongly relies on 
Theorem 
\ref{main:holder}.

\begin{proof}[Proof of Theorem \ref{existence:master}.]
{ \ }
\vskip5pt

\noindent {\textit{Existence.} The proof of existence holds in two steps.} 
\vskip 4pt

\textit{First Step.}
We first consider the following nonlinear variant of \eqref{eq:pde:apriori:1}:
\begin{equation}
\label{eq:pde:apriori:3}
\begin{split}
&\partial_{t} U^i(t,p) + 
\sum_{j \in \ES}
\Bigl(  \varphi(p_{j}) + b^i_{j}\bigl(t,p,U(t,p)\bigr) + p_{j} b_j^\circ \bigl(t,p,U(t,p)\bigr)\Bigr)  \fd_{p_{j}} U^i (t,p) 
\\
&\hspace{15pt} +\frac{\varepsilon^2}{2} \sum_{ j,k \in \ES}\bigl(p_j \delta_{jk}-p_{j} p_{k}\bigr) \fd^2_{p_{j} p_{k}} U^i(t,p)
+ h^i\bigl(t,p,U(t,p)\bigr) = 0,
\\
&U^i(T,p) = g^i(p),
\end{split}
\end{equation}
where for any $i \in \ES$, 
$b^i = (b^i_{j})_{j \in \ES} : [0,T] \times {\mathcal S}_{d-1} \times \RR^d \rightarrow (\RR_{+})^d$, $b^\circ=(b^\circ_{j})_{j \in \ES} : 
[0,T] \times {\mathcal S}_{d-1} \times \RR^d \rightarrow {\mathbb R}^{d}$
and
$h^i : [0,T] \times {\mathcal S}_{d-1} \times \RR^d \rightarrow \RR$. 
We are going to prove the existence of 
a solution 
$U=(U^1,\cdots,U^d)$
to 
\eqref{eq:pde:apriori:3}
whenever, for some constant $C_{0} \geq 0$, the functions 
$(b^i)_{i \in \ES}$, 
$b^{\circ}$
and
$(h^i)_{i \in \ES}$
are bounded by $C_{0}$ and satisfy the following regularity properties
\begin{equation}
\label{eq:continuity:drift:coefficients}
\begin{split}
&\bigl\vert 
b^i(t,p,y) - b^i(s,q,z)
\bigr\vert
+
\bigl\vert 
b^\circ(t,p,y) - b^\circ(s,q,z)
\bigr\vert
+
\bigl\vert
h^i(t,p,y) - h^i(s,q,z)
\bigr\vert
\\
&\hspace{15pt} \leq C_{0} \bigl( \vert t -s \vert^{\gamma/2}
+ \vert \sqrt{p}
- \sqrt{q} \vert^{\gamma} + \vert y - z \vert
\bigr),
\end{split}
\end{equation} 
for $i \in \ES$, $s,t \in [0,T]$, $p,q \in {\mathcal S}_{d-1}$ and 
$y,z \in {\mathbb R}^d$.
Without any loss of generality, we can assume that 
$\sup_{i \in \ES}\| g^i \|_{1,\infty} \leq C_{0}$. 
Existence of a classical solution to 
\eqref{eq:pde:apriori:3}
is then proved by a new application of Schauder's fixed point theorem. 
To do so, we call $\eta$ and $C$
the exponent and the constant from 
Theorem \ref{main:holder} 
 when $\| b \|_{\infty}$, $\|b^\circ \|_{\infty}$, $\| h \|_{\infty}$ and $\| \ell \|_{1,\infty}$ are less than $C_{0}$. 
 We then take an input function $V=(V^1,\cdots,V^d) \in [{\mathscr C}_{\textrm{\rm WF}}^{\eta/2,\eta}([0,T] 
 \times {\mathcal S}_{d-1})]^d$ such that, for each $i \in \ES$, 
 $\| V^i \|_{\eta/2,\eta} \leq C$. 
 By Lemma
\ref{lem:4.6} with $\eta$ therein being replaced by $\min(\eta,\gamma)$, we can solve 
\eqref{eq:pde:apriori:3} for each $i \in \ES$ when, in the nonlinear terms, $U$ is replaced by $V$. 
We call $U=(U^1,\cdots,U^d)$ the solution. It belongs to 
$[{\mathscr C}_{\textrm{\rm WF}}^{1+\gamma'/2,2+\gamma'}([0,T] 
 \times {\mathcal S}_{d-1})]^d$. By 
Theorem \ref{main:holder}, 
it also satisfies 
 $\| U^i \|_{\eta/2,\eta} \leq C$,
 for each $i \in \ES$.
 Revisiting if needed the proof of Lemma 
\ref{lem:4.6}, there is no difficulty in proving that the resulting map 
$V \mapsto U$ is continuous from 
$[{\mathscr C}_{\textrm{\rm WF}}^{\eta'/2,\eta'}([0,T] 
 \times {\mathcal S}_{d-1})]^d$
 into itself, for any $\eta' \in (0,\eta)$. 
 This permits to apply Schauder's theorem. 
\vskip 4pt

\textit{Second Step.}
The goal now is to  choose 
$(b^i)_{i \in \ES}$, 
$b^\circ$
and 
$(h^i)_{i \in \ES}$
(and hence $C_{0}$ as well)
accordingly so that the solution to 
\eqref{eq:pde:apriori:3} is in fact 
a solution to the master equation 
\eqref{eq:master:equation:intrinsic:reformulation}.
In order to proceed, we follow the same idea as in the proof of 
Lemma 
\ref{lem:verification:2}
and recall the truncated Hamiltonian 
\begin{equation*}
H^i_{c}(y)  = -\frac12 \sum _{j \in \ES} \bigl[ 
(y_{i}-y_{j})_+^2 {\mathbf 1}_{\{y_{i} - y_{j} \leq c\}}
+ \bigl(    2
c (y_{i} - y_{j}) - c^2 \bigr) {\mathbf 1}_{\{y_{i} - y_{j} > c\}}
\bigr],
\quad y=(y_{j})_{j \in \ES},
\end{equation*}
for a constant $c$ to be fixed later. 
Also, for another constant $\Gamma$, the value of which will be also fixed later on, we call
$\psi_{\Gamma}$ the function 
\begin{equation*}
\psi_{\Gamma}(r) :=
\left\{
\begin{array}{ll} 
r, &\quad \textrm{\rm if} \ \vert r \vert \leq \Gamma,
\\
\Gamma \textrm{\rm sign}(r),
&\quad \textrm{\rm if} \ \vert r \vert \geq \Gamma,
\end{array}
\right. \quad r \in {\mathbb R}. 
\end{equation*}
Given these notations, we let
(compare with \eqref{eq:B:F})
\begin{align}
&b^i_{j}(t,p,y) := \sum_{k \in \ES} p_{k} \min \Bigl( c, \bigl( y_{k} - y_{j} \bigr)_{+} \Bigr) 
+ \varepsilon^2  \delta_{i,j} , \quad 
i,j \in \ES \nonumber,
\\
&b^\circ_{j}(t,p,y) := 
-  \sum_{k \in \ES} \Bigl[ \varphi(p_{k}) +  \min \Bigl( c,  \bigl( y_{j} - y_{k} \bigr)_{+}\Bigr) \Bigr] -\varepsilon^2, 
\quad j \in \ES, \label{eq:B:F:2}
\\
&h^i(t,p,y) := \psi_{\Gamma}\biggl( H_{c}^i(y) + f^i(t,p) + \sum_{j \in \ES}  \varphi(p_{j}) \psi_{\Gamma}\bigl( y_{j} - y_{i} \bigr) \biggr), \quad i \in \ES,
\nonumber
\end{align}
for $(t,p,y) \in [0,T] \times {\mathcal S}_{d-1} \times \RR^d$. For a given value of $c$, we can choose $\Gamma$ (hence depending on $c$) such that all the above coefficients are bounded by 
$\Gamma$.  
Moreover, the coefficients satisfy 
\eqref{eq:continuity:drift:coefficients} for a suitable choice of $C_{0}$ therein (notice in this regard that this is the specific interest of the second occurence of $\psi_{\Gamma}$ to force the whole term to be jointly Lipschitz in $(p,y)$). 
By the first step, there exists a classical solution, say $U=(U^1,\cdots,U^d)$, to \eqref{eq:pde:apriori:3}. 
As in the proof of Theorem 
\ref{lem:existence:mfg:from:master}, 
we can represent $U$ through a forward-backward stochastic differential equation. 
Following 
\eqref{eq:MFG:system:epsilon}, the backward equation writes 
\begin{equation}
\label{eq:backward:just:used:once}
\begin{split}
&\displaystyle \ud u_{t}^i = - h^i(t,P_{t},u_{t})  \ud t  - \frac{\varepsilon}{\sqrt{2}} \sum_{j \in \ES} 
\sqrt{\frac{{P_{t}^j}}{{P_{t}^i}}} \bigl( \nu_{t}^{i,i,j} - \nu_{t}^{i,j,i} \bigr) \ud t
+ \sum_{j,k \in \ES} \nu_{t}^{i,j,k} \ud W_{t}^{j,k},
\end{split}
\end{equation}
with $u_{T}^i = g^i(P_{T})$ as terminal boundary condition, 
where $(P_{t})_{0 \le t \le T}$ is the solution to the corresponding forward equation (but there is no need to write it down). 
The key point here is to observe that $h^i$ is at most of linear growth in $u$, uniformly in $(t,p)$, 
the constant in the linear growth depending on 
$c$ but not on $\Gamma$. 
By considering the drifted Brownian motions $((W^{\prime,i,j}_{t}=W^{i,j}_{t} -\int_{0}^t 
(\varepsilon \sqrt{P_{s}^j})/(\sqrt{2 P_{s}^i}) \ud s)_{0 \le t \le T})_{j \in \ES: j \not = i}$ and $((W^{\prime,j,i}_{t}=W^{j,i}_{t} +\int_{0}^t 
(\varepsilon \sqrt{P_{s}^j})/(\sqrt{2 P_{s}^i}) \ud s)_{0 \le t \le T})_{j \in \ES : j\not = i}$, for a given value of $i \in \ES$, we can apply Girsanov theorem to get rid of the second term in the equation for $(u^i_{t})_{0 \le t \le T}$ in \eqref{eq:backward:just:used:once}, the application of Girsanov theorem being here made licit by Proposition \ref{expfin}. We easily deduce that, with $U$ as in  \eqref{eq:backward:just:used:once}, 
$\| U(t,\cdot)\|_{\infty} \leq C (1+ \int_{t}^T \| U(s,\cdot)\|_{\infty} \ud s)$ and then deduce that 
$U$ and hence $(u_{t})_{0 \le t \le T}$ are bounded by a constant $C$ that depends on $c$ but not on $\Gamma$. 
In particular, it makes sense to choose $\Gamma$ large enough such that, for 
$y=(y_{j})_{j \in \ES}$
with 
$\vert y \vert \leq C$, $h^i(t,p,y)$ in 
\eqref{eq:B:F:2} 
is also equal to 
\begin{equation*}
h^i(t,p,y) =  H_{c}^i(y) + f^i(t,p) + \sum_{j \in \ES} \varphi(p_{j})\bigl( y_{j} - y_{i} \bigr).
\end{equation*}
It says that the backward equation 
\eqref{eq:backward:just:used:once}
identifies with the backward equation 
\eqref{shjbmfg:truncated}
in the proof of Lemma 
\ref{lem:verification:2}. But the point in the second step of the proof of Lemma 
\ref{lem:verification:2} is precisely to show that 
the solution to 
\eqref{eq:backward:just:used:once}
can be bounded independently of $c$. In words, we can find a constant $C_{1}$, independent of $c$ (and of course 
of $\Gamma$) such that $U=(U^i)_{i \in \ES}$ is bounded by $C_{1}$.
Then, choosing $c \geq 2C_{1}$ (and $\Gamma$ large enough as before), we have that, for any 
$y=(y_j)_{j \in \ES}$ with 
$\vert y \vert \leq C_{1}$, and any $(t,p) \in [0,T] \times {\mathcal S}_{d-1}$
and $i,j \in \ES$,
\begin{equation*}
B_{j}^i(t,p,y) = 
  \varphi(p_{j} )+ b^i_{j}(t,p,y) + b^{\circ}_{j}(t,p,y) p_{j}, 
\quad 
F^i(t,p,y)
=h^i(t,p,y),
\end{equation*}
with 
$B^i$ and $F^i$ as in 
\eqref{eq:B:F}. 
This shows that $U=(U^1,\cdots,U^d)$ solves the master equation \eqref{eq:master:equation:intrinsic:reformulation}. 
\vskip 5pt

{\noindent \textit{Uniqueness.} By Theorem 
\ref{lem:existence:mfg:from:master}, we know
that, 
for any initial condition $p_{0}=(p_{0,i})_{i \in \ES} \in \mathrm{Int}({\mathcal S}_{d-1})$, 
the system 
\eqref{eq:MFG:system:epsilon}
has a unique solution. Hence, for any two solutions $U$ and $U'$ to the master equation, one has 
$U(0,p_{0})=U'(0,p_{0})$. Since $p_{0}$ is arbitrary, we get that 
$U(0,\cdot)$ and $U'(0,\cdot)$ coincide on $\textrm{\rm Int}(\hat{\mathcal S}_{d-1})$. By continuity, they coincide up to the boundary. 
Here, the initial time is arbitrary and we can replace the initial time $0$ by any other initial time 
$t \in (0,T)$. }
\end{proof}

\section{Proof of the a priori H\"older estimate}\label{sec:apriori}
The proof of Theorem 
\ref{main:holder}
is the core of the paper. The main ingredient is a coupling estimate for the diffusion process associated with 
the linear equation \eqref{eq:pde:apriori:1}, see the statement of Proposition 
\ref{prop:coupling:2}.
Whilst this approach is mostly inspired by earlier coupling arguments used to prove regularity of various classes of harmonic functions ({see for instance \cite{ChenLi,Cranston,Eberle,LuoWang}}), 
we here need a tailored version
that fits the specificities of Kimura operators. 
In short, 
the coupling estimate we obtain
below does not suffice to conclude directly in full generality. In fact, it just permits to derive the required H\"older estimate in the case $d=2$. 
In the higher dimensional setting, we need an additional argument that uses induction on the dimension of the state space 
to pass from the coupling estimate to the H\"older bound; see 
{Remark \ref{rem:about:induction} for a first account and
Subsection 
\ref{subse:induction:d} for more details}.  {In short,}
the rationale for this additional induction argument is that the coupling estimate
obtained in Proposition 
\ref{prop:coupling:2} blows up near the boundary, except when $d=2$. 
As for the induction argument itself, it is based on a conditioning property that is proper to Kimura type operators: Roughly speaking, the last $d-m$ 
coordinates of the diffusion process 
associated with 
the linear equation \eqref{eq:pde:apriori:1} (see also \eqref{eq:pde:apriori:1:intro}) behave, conditional on the first $m$ coordinates, 
as a diffusion associated with 
a linear equation of the same type as \eqref{eq:pde:apriori:1} but in dimension $d-m$ instead of $d$, {see Proposition 
\ref{prop:representation} for the complete statement. 
It is worth mentioning that our induction argument is inspired by the work \cite{albanesemangino}. Therein, the authors prove a gradient estimate for simpler and more regular forms of drifts by iterating on the dimension of the state space. Differently from ours, their approach is purely deterministic: As a result, the conditioning principle exposed in Proposition 
\ref{prop:representation} manifests implicitly in \cite{albanesemangino} through the form of the underlying Kimura operators.} 

Throughout the section, we are given coefficients 
$b=(b_{i})_{i \in \ES}$, $b^\circ=(b^\circ_{i})_{i \in \ES}$, $h$ and $\ell$ as in the statement of 
Theorem 
\ref{main:holder}. Then, the aforementioned diffusion process associated with 
\eqref{eq:pde:apriori:1} is given by the following statement. 

\begin{proposition}
\label{prop:ito:formula}
Consider
$\varphi$ as in  
\eqref{eq:varphi}
with $\delta \in (0,1)$ and $\kappa \geq \varepsilon^2/2$, for $\varepsilon >0$. Then, the stochastic differential equation 
\begin{equation}
\label{eq:P}
\ud P_{s}^i = \Bigl( \varphi(P_{s}^i) +  b_{i}(s,P_{s}) + P_{s}^i b_{i}^{\circ}(s,P_{s}) \Bigr) \ud t + \varepsilon \sqrt{P_{s}^i} \sum_{j \in \ES} \sqrt{P_{s}^j} \ud \overline{W}_{t}^{i,j}, \quad s \in [t,T], \  i \in \ES, 
\end{equation}
is uniquely solvable for any initial time in $t\in [0,T]$ and any (possibly random) initial condition in $\mathrm{Int}({\mathcal S}_{d-1})$. Moreover, the coordinates of the solution remain almost surely strictly positive. 
In particular, for any $(t,p) \in [0,T] \times \textrm{\rm Int}(\hat{\mathcal S}_{d-1})$,
for any $[t,T]$-valued stopping time $\tau$ (with respect to the filtration ${\mathbb F}^{\boldsymbol W}$), 
any function $u$ as in the statement of Theorem 
\ref{main:holder} 
(with 
\eqref{eq:pde:apriori:1:intro} therein
being replaced by the more rigorous version \eqref{eq:pde:apriori:1})
admits the representation
\begin{equation}
\label{eq:Kolmogorov:u:5.1}
u(t,p) = {\mathbb E} \biggl[ u\bigl( \tau,P_{\tau}^{t,p} \bigr) + \int_{t}^{\tau} h(s,P_{s}^{t,p}) \ud s \biggr],
\end{equation}
where $P^{t,p}$ is the $d$-dimensional process whose dynamics are given by \eqref{eq:P} and starts from $p$ at time $t$. 
In particular, the $L^\infty$ bound in Theorem 
\ref{main:holder} holds true. 
\end{proposition}

\begin{proof}
Strong existence and uniqueness may be proven in Proposition 
\ref{thm:approximation:diffusion:2} (Equation \eqref{eq:P} is slightly more general than the equation handled in 
Proposition 
\ref{thm:approximation:diffusion:2}, but the proof works in the same way). Representation of $u$ is a straightforward consequence of 
It\^o's formula (as in the proof of Theorem
\ref{lem:existence:mfg:from:master}). 
\end{proof}

It is worth observing that, taking $\tau=T$ in 
\eqref{eq:Kolmogorov:u:5.1}, $u$ has the (standard) representation: 
\begin{equation}
\label{eq:Kolmogorov:u:5.1:b}
u(t,p) = {\mathbb E} \biggl[ \ell\bigl(P_{T}^{t,p} \bigr) + \int_{t}^{T} h(s,P_{s}^{t,p}) \ud s \biggr],
\quad (t,p) \in [0,T] \times \textrm{\rm Int}\bigl(\hat{\mathcal S}_{d-1}\bigr), 
\end{equation}
which is of course very useful to us. Indeed, using a standard mollification argument
(taking benefit of the fact that the coefficients 
$(b_{i})_{i \in \ES}$ and $(b_{i}^\circ)_{i \in \ES}$
are continuous), we can easily approximate the coefficients
$(b_{i})_{i \in \ES}$ and $(b_{i}^\circ)_{i \in \ES}$ for the sup norm by sequences of coefficients 
$((b_{i}^n)_{i \in \ES})_{n \geq 1}$ and $((b_{i}^{\circ,n})_{i \in \ES})_{n \geq 1}$
that are time-space continuous and Lipschitz continuous in the space variable (uniformly in the time variable).
Hence, if we prove that 
$u$ in
\eqref{eq:Kolmogorov:u:5.1:b} 
satisfies the H\"older estimate stated in Theorem 
\ref{main:holder} for coefficients $(b_{i})_{i \in \ES}$ and $(b_{i}^\circ)_{i \in \ES}$
that are Lipschitz continuous in space (uniformly in time), we can 
deduce that the same holds when 
$(b_{i})_{i \in \ES}$ and $(b_{i}^\circ)_{i \in \ES}$
are merely continuous
by passing to the limit along the aforementioned mollification\footnote{\label{foo:mollif} As noticed in Remark \ref{rem:thm47}, 
we may think of adapting the argument when $(b_{i})_{i \in \ES}$ and $(b_{i}^\circ)_{i \in \ES}$ are just measurable. In such a case, 
we could no longer approximate 
them in sup norm and we should work instead with some $L^p$ norm. This should require 
to control 
the mean occupation measure of the 
 process ${\boldsymbol P}^{t,p}$ 
in terms of this  $L^p$ norm. We believe that this is possible thanks to the ellipticity property of the operator inside the domain.}.

In other words, we may assume for our purpose that $(b_{i})_{i \in \ES}$ and $(b_{i}^\circ)_{i \in \ES}$ are Lipschitz continuous in space, uniformly in time, provided that we prove that the resulting H\"older estimate does not depend on the Lipschitz constants of $(b_{i})_{i \in \ES}$ and $(b_{i}^\circ)_{i \in \ES}$.

Throughout the section, we assume that, as in the statement of Theorem \ref{main:holder},  $\varepsilon$ is in $(0,1)$.

\subsection{Preliminary results on coupling and conditioning}

\subsubsection{Conditioning on the $m$ first coordinates}
The core of the analysis is based upon the probabilistic representation 
\eqref{eq:Kolmogorov:u:5.1}
and in turn on the properties of the process 
${\boldsymbol P}=(P_{t}^1,\cdots,P_{t}^d)_{0 \leq t \leq T}$
solving equation
\eqref{eq:P}. 

As we already alluded to a few lines before, our general strategy relies on 
an induction argument based upon the dimension of 
the 
state variable. This is precisely the goal of this paragraph to clarify the way we 
may reduce dimension inductively. 
General speaking, the arguments is based on a conditioning argument. 

In order to make it clear, we rewrite 
\eqref{eq:P}, but using (at least for the sole purpose of the statement 
of Proposition \ref{prop:representation} right below) the letter $X_{t}$ instead of $P_{t}$ for the unknown:
\begin{equation}
\label{eq:Xti:2}
\ud X_{t}^i = \Bigl( \varphi(X_{t}^i) +  b_{i}(t,X_{t}) + X_{t}^i b_{i}^{\circ}(t,X_{t}) \Bigr) \ud t + \varepsilon \sqrt{X_{t}^i} \sum_{j \in \ES} \sqrt{X_{t}^j} \ud \overline{W}_{t}^{i,j}, 
\quad t \in [t_{0},T], \ i \in \ES, 
\end{equation}
for a given initial time $t_{0}$. 
Our rationale to change $P_{t}$ into $X_{t}$
is motivated by the fact that we feel better to keep the letter $P_{t}$ for the new state variable 
once the dimension has been reduced. 
The objective is then to write the law of $(X_{t})_{t_{0} \leq t \le T}$ in the form 
\begin{equation}
\label{eq:identity:law}
\begin{split}
&\bigl(X_{t}\bigr)_{t_{0} \le t \le T}
\overset{\text{law}}{=}
\Bigl(P^{\circ,1}_{t},\cdots,P^{\circ,m}_{t}, 
\varsigma^2(P_{t}^{\circ}) P_{t}^{1},
\cdots,
\varsigma^2(P_{t}^{\circ}) P_{t}^{d-m}
\Bigr)_{t_{0} \le t \le T},
\end{split}
\end{equation}
where ${\boldsymbol P}^\circ=(P^\circ_{t}=(P_{t}^{\circ,1},\cdots,P_{t}^{\circ,m}))_{t_{0} \le t \le T}$ and ${\boldsymbol P}=(P_{t}=(P_{t}^1,\cdots,P_{t}^{d-m}))_{t_{0} \le t \le T}$ are new stochastic processes taking respectively values within the set 
$\hat {\mathcal S}_{m} = \{ (p^\circ_{1},\cdots,p^\circ_{m}) \in ({\mathbb R}_{+})^m : 
\sum_{i=1}^m p^\circ_{i} \leq 1 \}$
and 
${\mathcal S}_{d-m-1} = \{(p_{1},\cdots,p_{d-m}) \in ({\mathbb R}_{+})^{d-m} :
\sum_{i=1}^{d-m} p_{i} =1 \}$. Above,   
 $\varsigma$ is given by 
\begin{equation*}
\varsigma(p^\circ) := {\sqrt{1-( p_{1}^\circ +\cdots+p_{m}^\circ})}, \quad (p^\circ_{1},\cdots,p^\circ_{m}) \in \hat{\mathcal S}_{m}.
\end{equation*}

\begin{proposition}
\label{prop:representation}
Given coefficients 
$(b_{i})_{i \in \ES}$
and
$(b_{i}^\circ)_{i \in \ES}$
as in the statement of Theorem 
\ref{main:holder}, 
there exist new coefficients
\begin{itemize}
\item 
$(\widetilde b_{i})_{i \in \ESdm}$, with values in $({\mathbb R}_{+})^{d-m}$, 
 that are bounded by a constant that only depends on 
$(\| b_{i} \|_{\infty})_{i \in \ES}$,
\item 
$(\widetilde b_{i}^\circ)_{i \in \ESdm}$, 
with values in 
${\mathbb R}^{d-m}$, 
 that are bounded by a constant that only depends on 
$(\| b_{i} \|_{\infty})_{i \in \ES}$
and 
$(\| b_{i}^\circ \|_{\infty})_{i \in \ES}$,
\end{itemize}
 and that are Lipschitz continuous in space uniformly in time such that
\begin{itemize}
\item 
whenever $\delta \in (0,1)$ and $\kappa \geq \varepsilon^2/2$,
\item
for any family of antisymmetric Brownian motions 
$\overline{\boldsymbol W}^\circ=(\overline W^\circ_{t}=(\overline W^{\circ,i,j}_{t})_{i,j \in \ES : i \not = j})_{0 \leq t \leq T}$
of dimension $d(d-1)/2$ that is independent of 
$\overline{\boldsymbol W}=(\overline W_{t}=(\overline W^{i,j}_{t})_{i,j \in \ES: i \not =j})_{0 \leq t \leq T}$,
\item
for any given initial condition $(t_{0},p^\circ,p) \in [0,T] \times \hat{\mathcal S}_{m} \times {\mathcal S}_{d-m-1}$, 
\end{itemize}
the system
\begin{equation}
\label{eq:new:system:prop:5:2}
\begin{split}
&\ud P_{t}^i = \varsigma^{-2}(P_{t}^\circ) \Bigl( \varphi\bigl(\varsigma^2(P_{t}^\circ) P_{t}^i \bigr) + \widetilde b_{i}\bigl(t,P_{t}^\circ,P_{t}\bigr)  + P_{t}^i \widetilde b_i^{\circ}(t,P_{t}^\circ,P_{t}) \Bigr) \ud t 
\\
&\hspace{5pt}+ \varepsilon \varsigma^{-1}(P_{t}^\circ) \sum_{j \in \ESdm : j \not = i} \sqrt{P_{t}^i P_{t}^j} \ud \overline W_{t}^{i,j}, 
\quad i \in \ESdm,
\\
&\ud P_{t}^{\circ,i} = 
\Bigl(  \varphi\bigl( P_{t}^{\circ,i} \bigr) + b_{i}\bigl(t,(P_{t}^\circ,\varsigma^2(P_{t}^\circ)P_{t})\bigr)  + P_{t}^i  b_i^{\circ}\bigl(t,P_{t}^\circ,
\varsigma^2(P_{t}^\circ)
P_{t}
\bigr) 
\Bigr)  \ud t 
\\
&\hspace{5pt}
+\varepsilon \sum_{j \in \ESm : j \not = i}
\sqrt{P^{\circ,i}_{t} {P}_{t}^{\circ,j}}
\ud \overline{W}^{\circ,i,j}_{t} 
+
\varepsilon \varsigma(P_{t}^\circ) 
 \sqrt{P_{t}^{\circ,i}}
 \sum_{j \in \ESdm}
\sqrt{{P}_{t}^{j}}
\ud \overline{W}^{\circ,i,m+j}_{t}, \quad i \in \ESm,
\end{split}
\end{equation}
for $t \in [t_{0},T]$,
with $(P^\circ_{t_{0}},P_{t_{0}})=(p^\circ,p)$, 
has a unique  {strong} solution, which satisfies 
the identity in law 
\eqref{eq:identity:law}
whenever 
\eqref{eq:Xti:2}
is initialized from $(p^\circ,\varsigma^2(p^\circ)p)$ at time $t_{0}$. 
\end{proposition}

The proof of 
Proposition 
\ref{prop:representation}
is deferred to Subsection \ref{subsubse:proof:proposition 5.2}.
Throughout, we denote by ${\mathbb F}^{{\boldsymbol W}^\circ,{\boldsymbol W}}=({\mathcal F}_{t}^{{\boldsymbol W}^\circ,{\boldsymbol W}})_{0 \leq t \leq T}$ 
the augmented filtration
generated by 
${\boldsymbol W}^\circ=((W_{t}^{\circ,i,j})_{0 \leq t \leq T})_{i,j \in \ES : i \not = j}$
and
${\boldsymbol W}=((W_{t}^{i,j})_{0 \leq t \leq T})_{i,j \in \ES : i \not = j}$, the latter two being implicitly understood as 
two independent collections of Brownian motions such that 
$\overline{\boldsymbol W}^{\circ,i,j}=( {\boldsymbol W}^{\circ,i,j}- {\boldsymbol W}^{\circ,j,i})/\sqrt{2}$ and
$\overline{\boldsymbol W}^{i,j}=( {\boldsymbol W}^{i,j}- {\boldsymbol W}^{j,i})/\sqrt{2}$, 
when
$\overline{\boldsymbol W}^{\circ}$
and
$\overline{\boldsymbol W}$
are as in the statement of Proposition 
\ref{prop:representation}.

\subsubsection{Main coupling estimate}

\begin{proposition}
\label{prop:coupling:2}
Assume that $\delta$ in 
\eqref{eq:varphi}
belongs to $(0,1/(4\sqrt{d}))$ Moreover, 
take two initial conditions $(t_{0},p^{\circ},p)$ and $(t_{0},q^\circ,q)$ in $[0,T] \times \hat{\mathcal S}_{m} \times {\mathcal S}_{d-m-1}$, with $m \in \{1,\cdots,d-1\}$, such that $\vert p^\circ \vert_{1} := p^\circ_{1}+\cdots+p^\circ_{m} \leq 1/2$, $\vert q^\circ \vert_{1} \leq 1/2$, and $\vert p -q \vert < \delta^2/(64 \sqrt{d})$. On the (filtered) probability space carrying $\overline{\boldsymbol W}^\circ$ and $\overline{\boldsymbol W}$, call $(P_t^\circ=(P_{t}^{\circ,1},\cdots,P_{t}^{\circ,m}),P_{t}=(P_{t}^1,\cdots,P_{t}^{d-m}))_{t_{0} \leq t \leq T}$ the solution to 
\eqref{eq:new:system:prop:5:2} with $(t_{0},p^\circ,p)$ as initial condition. 

Then, 
for any $\eta \in (0,1)$, 
there exists a threshold $\kappa_{0} \geq 2$, only depending on $\eta$, $\varepsilon$ and $(\| b_{i} \|_{\infty})_{i \in \ES}$
(but not on $\delta$)
such that, for any $\kappa \geq \kappa_{0}$, we can find a constant $C$, depending on $\delta$, 
$\varepsilon$, 
$\kappa$, $\eta$, $(\|b_{i}\|_{\infty})_{i \in \ES}$, 
$(\|b_{i}^{\circ}\|_{\infty})_{i \in \ES}$
and $T$ such that, provided that $\vert p-q \vert^{1/3} \leq T-t_{0}$, there exists an adapted process
$(Q_{t}^\circ=(Q_{t}^{\circ,1},\cdots,Q_{t}^{\circ,m}),Q_{t}=(Q_{t}^1,\cdots,Q_{t}^{d-m}))_{t_{0} \le t \le T}$ that has the same law as the solution to 
\eqref{eq:new:system:prop:5:2} with $(t_{0},q^\circ,q)$ as initial condition and for which the following property holds true.

If we call
\begin{equation*}
\widetilde{P}_{t} := \Bigl( \sqrt{P_{t}^1},\cdots, \sqrt{ P_{t}^{d-m}}  \Bigr), 
\quad
\widetilde{Q}_{t} := \Bigl( \sqrt{ Q_{t}^1},\cdots,\sqrt{Q_{t}^{d-m}} \Bigr), 
\quad t \in [t_{0},T],
\end{equation*}
and
\begin{equation}
\label{eq:stopping:times}
\begin{split}
&\varrho := \inf \bigl\{ s \geq t_{0} : \bigl\vert P_{s}^\circ - Q_{s}^\circ \bigr\vert > 
\bigl\vert \widetilde P_{s} - \widetilde Q_{s} \bigr\vert \bigr\}, \quad 
\rho: = \inf \bigl\{ s \geq t_{0} : \bigl\vert P_{s}^\circ \bigr\vert_{1} \geq 3/4\}, 
\\ 
&\sigma := \inf\bigl\{ s \geq t_{0} : \bigl\vert \widetilde P_{s}
- \widetilde Q_{s} \bigr\vert \geq \delta/4 \bigr\},
\quad
\tau := \inf \bigl\{ s \geq t_{0} : \bigl\vert \widetilde P_{s}
- \widetilde Q_{s} \bigr\vert =0 \bigr\},
\end{split}
\end{equation} 
then 
\begin{equation}
\label{eq:induction:coupling}
{\mathbb P} \bigl( \bigl\{ \varpi_{S} < \tau \wedge \varrho \wedge \rho \bigr\} \bigr) 
  \leq C  \frac{\vert p -q \vert^{1/12}}{\min_{i \in \ESdm}(\max(p_i,q_i))^{\eta}}, 
\end{equation}
where $\varpi_{S}:=\varpi \wedge S$, with $\varpi: = \rho \wedge \varrho \wedge \sigma \wedge \tau$ and $S:= t_{0}+\vert p - q \vert^{1/3}$.
\end{proposition}

The proof of 
Proposition \ref{prop:coupling:2}
is deferred to 
Subsection
\ref{subse:proof:coupling:property}. 

{\begin{remark}
\label{rem:about:induction}
We now explain the difficulty when the dimension $d$ is greater than or equal to 3 and the reason why we need an induction argument to derive the required H\"older estimate.

A naive way to proceed is indeed to choose $m=0$ in the above statement. In such a case, the 
process $(P^1,\cdots,P^{d-m})$ in 
\eqref{eq:new:system:prop:5:2} coincides with the solution 
$(P^1,\cdots,P^d)$ in 
\eqref{eq:P}. In other words, 
Proposition 
\ref{prop:coupling:2} with $m=0$ reads as a coupling estimate for the 
solution to \eqref{eq:P}. 

Let us now see what 
 the right hand side of 
\eqref{eq:induction:coupling} becomes
when $m=0$. 
Up to an obvious change of coordinates, we then may assume that 
$\min_{i \in \ES}(\max(p_{i},q_{i}))=\max(p_{1},q_{1})$, in which case the right-hand side of 
\eqref{eq:induction:coupling} 
writes 
$\vert p-q \vert^{1/12}/\max(p_{1},q_{1})^{\eta}$,
both $p=(p_{1},\cdots,p_{d})$ and $q=(q_{1},\cdots,q_{d})$ being now regarded as 
$d$-dimensional vectors. The point is then to upper bound 
$\vert p-q \vert^{1/12}/\max(p_{1},q_{1})^{\eta}$. When $d=2$, this is pretty easy because 
\begin{equation*}
\vert p-q \vert = \sqrt{ \vert p_{1}- q_{1} \vert^2 + \vert p_{2} - q_{2} \vert^2}=
\sqrt{ \vert p_{1} - q_{1} \vert^2 + \vert (1-p_{1}) - (1-q_{1}) \vert^2} \leq \sqrt{2} \vert p_{1} - q_{1} \vert,
\end{equation*}
and then we get 
\begin{equation*}
\frac{\vert p-q \vert^{1/12}}{\max(p_{1},q_{1})^{\eta}}
\leq \sqrt{2} \vert p-q \vert^{1/12-\eta}.
\end{equation*}
Unfortunately, this argument no longer works when $d \geq 3$ since, in that case, 
one of the entries
$(\vert p_{i} - q_{i} \vert)_{i=2,\cdots,d}$ may be much larger than $\vert p_{1} - q_{1} \vert$. 
\end{remark}

}

\subsection{Derivation of the H\"older estimate and proof of Theorem 
\ref{main:holder}}
\label{subse:induction:d}
We now explain how to derive the H\"older estimate in Theorem 
\ref{main:holder}
from 
Proposition \ref{prop:coupling:2}. 
As we already alluded to, it relies on an additional iteration on the dimension, which is in turn 
inspired by 
earlier PDE results on Kimura diffusions, see for instance \cite{albanesemangino}. 
The induction assumption takes the following form. 

Take $h$, $\ell$ and $u$ as in \eqref{eq:pde:apriori:1}. For a given $m \in \ES$, call ${\mathscr P}_{m}$ the following property: For any $\varepsilon,\eta \in (0,1)$ and $\delta \in  (0,1/(4\sqrt{d}))$, there exist a threshold $\kappa_{0}$, only depending 
on $\varepsilon$, $\eta$, $m$ and $(\| b_{i} \|_{\infty})_{i \in \ES}$, 
and 
 an exponent $\alpha \in (0,1)$, only depending on $\eta$ and $m$, 
such that, for any $\kappa \geq \kappa_{0}$, we can find a constant $C$, depending on $\delta$, $\varepsilon$, $\kappa$, $\eta$, $(\|b_{i}\|_{\infty})_{i \in \ES}$, 
$(\|b_i^\circ \|_{\infty})_{i \in \ES}$, $\|h\|_{\infty}$, 
$\|\ell \|_{1,\infty}$
and $T$ such that, for any $p=(p_1,\cdots,p_d)$ and $q=(q_1,\cdots,q_d)$ in 
${\mathcal S}_{d-1}$, 
\begin{equation}
\label{eq:induction}
\vert u(t,p_1,\cdots,p_d) - u(t,q_1,\cdots,q_d) \vert \leq C \frac{\vert p - q \vert^{\alpha}}{\max(p,q)_{(m)}^{\eta \alpha}},
\end{equation}
where $\max(p,q)_{(m)}$ denotes the $m$th element in the increasing reordering of 
$$\max(p,q)=\bigl(\max(p_{1},q_{1}),\cdots,\max(p_{d},q_{d})\bigr).$$ 

We then have the following two propositions:

\begin{proposition}
\label{prop:induction:step:0}
Within the framework of Theorem \ref{main:holder}, 
${\mathscr P}_{1}$ holds true. 
\end{proposition}

\begin{proposition}
\label{prop:induction}
Within the framework of Theorem \ref{main:holder}, assume that there exists an integer $m \in \{1,\cdots,d-1\}$ such that 
${\mathscr P}_{m}$ holds true. Then, ${\mathscr P}_{m+1}$ holds true. 
\end{proposition}

Notice that Proposition 
\ref{prop:induction}
implies Theorem \ref{main:holder}: It suffices to choose $\eta=1/2$
in ${\mathscr P}_{d}$, noticing that $\max(p,q)_{(d)}$ is necessarily greater than $1/d$. Below, we directly prove Proposition \ref{prop:induction}. The proof of Proposition \ref{prop:induction:step:0} is completely similar (somehow, everything works as if we had a property ${\mathscr P}_{0}$). 
 
\begin{proof}
For 
some $m \in \{1,\cdots,d-1\}$
and
some $\eta \in (0,1/4)$, we consider $\kappa_{0}$ as being the maximum 
of $\kappa_{0}$ given by Proposition
 \ref{prop:coupling:2}
with $\eta$ replaced by $\eta/12$ therein
and of $\kappa_{0}$ given by 
${\mathscr P}_{m}$ with $\eta$ therein. 
Also, we consider $\alpha$ given by ${\mathscr P}_{m}$ and $\kappa \geq \kappa_{0}$.
We then assume that ${\mathscr P}_{m}$ holds true and we take $p,q \in {\mathcal S}_{d-1}$ together with $t_{0}\in [0,T]$. 
Without any loss of generality, we may assume that 
\begin{equation}
\label{eq:1/2md}
\max(p_1,q_1) \leq \max(p_2,q_2) \leq \cdots \leq \max(p_m,q_m) \leq \frac{1}{2m d}. 
\end{equation}
Observe that if the last inequality is not satisfied, the bound \eqref{eq:induction} at rank $m+1$ (at $t_{0}$ instead of $t$) is a straightforward consequence of the bound at rank $m$ with $(2md)^{\eta \alpha} C$ instead of $C$ as constant. 
For sure, we may also assume 
that
\begin{equation}
\label{eq:increasing:order}
\max(p_1,q_1) \leq \max(p_2,q_2) \leq \cdots \leq \max(p_d,q_d),
\end{equation} 
in which case $\max(p_d,q_d)$ is the largest element in the sequence 
$\max(p_i,q_i)$. In particular, at least $p_d$ or $q_d$ is above $1/d$ (since one of the two elements  dominates all the other elements in the family $(p_{1},\cdots,p_{d},q_{1},\cdots,q_{d})$). Hence, we may assume that $\min(p_d,q_d) \geq 1/(2d)$. Again, the proof is over if not since $\vert p-q \vert$ is then necessarily larger than $1/(2d)$: Tuning $C$ accordingly, \eqref{eq:induction} follows from the fact that $u$ is bounded, see Proposition 
\ref{prop:ito:formula}. By the same argument, we may assume that $\vert p-q\vert <  \delta^2/(128 d^{3/2})$. 
\vskip 5pt

\textit{First Step.}
Clearly, 
\begin{equation}
\label{eq:decomposition}
\begin{split}
&\vert u(t_{0},p_1,\cdots,p_d) - u(t_{0},q_1,\cdots,q_d) \vert
\\
&\leq \sum_{i \in \ESm}
\Bigl(
\bigl\vert u\bigl(t_{0},q_{1},\cdots,q_{i-1},p_i,\cdots,p_{d-1},1 - q_1 - \cdots q_{i-1} - p_i - \cdots- p_{d-1}\bigr) 
\\
&\hspace{45pt}- u(t_{0},q_1,\cdots,q_i,p_{i+1},\cdots,1 - q_1 - \cdots q_{i} - p_{i+1} - \cdots - p_{d-1}) \bigr\vert \Bigr)
\\
&\hspace{15pt} + 
\bigl\vert u\bigl(t_{0},q_{1},\cdots,q_{m},p_{m+1},\cdots,p_{d-1},1 - q_1 - \cdots q_{m} - p_{m+1} - \cdots- p_{d-1}\bigr)
\phantom{\Bigr)}
\\
&\hspace{45pt}- u(t_{0},q_1,\cdots,q_m,q_{m+1},\cdots) \bigr\vert,
\end{split}
\end{equation}
with the obvious convention that 
$(q_{1},\cdots,q_{i-1},p_i,\cdots,p_{d-1},1 - q_1 - \cdots q_{i-1} - p_i - \cdots- p_{d-1})=(p_{1},\cdots,p_{d})$
when $i=1$. 
Notice from 
\eqref{eq:1/2md} 
and
from the bound $\min(p_d,q_d) \geq 1/(2d)$
that, for $i \in \ESm$, $q_1 + \cdots + q_i \leq 1/(2d)$ and $p_i + \cdots + p_{d-1} \leq 1 - 1/(2d)$, which fully justifies 
the fact that all the entries above are non-negative.
Obviously, by the induction assumption, for any $i \in \ESm$, 
\begin{equation*}
\begin{split}
&\sum_{i \in \ESm}
\Bigl(
\bigl\vert u\bigl(t_{0},q_{1},\cdots,q_{i-1},p_i,\cdots,p_{d-1},1 - q_1 - \cdots q_{i-1} - p_i - \cdots- p_{d-1}\bigr) 
\\
&\hspace{15pt}- u(t_{0},q_1,\cdots,q_i,p_{i+1},\cdots,1 - q_1 - \cdots q_{i} - p_{i+1} - \cdots - p_{d-1}) \bigr\vert \Bigr)
\\
&\leq C \sum_{i \in \ESm} \frac{\vert p_{i} - q_{i}\vert^{\alpha}}{\max(p,q)^{\eta \alpha}_{(m)}}
\leq C \vert p-q \vert^{(1-\eta) \alpha}, \phantom{\Bigr)}
\end{split}
\end{equation*}
where we used the fact that 
$\max(p,q)_{(m)}
=\max(p_{m},q_{m})$
and where we modified the value of $C$ in the last term. 
The conclusion is that, in 
\eqref{eq:decomposition}, we can focus on the last term. Equivalently,  we can assume that $p_i=q_i$, for $i=1,\cdots,m$, 
provided we replace 
\eqref{eq:increasing:order}
by (which is weaker, but which is the right assumption here since there is no way to compare properly the last coordinates in 
the last term of \eqref{eq:decomposition})
\begin{equation}
\label{eq:increasing:order:b}
\max_{i \in \ESm} p_{i} = 
\max_{i \in \ESm} q_{i} \leq \max(p_{m+1},q_{m+1}) \leq \cdots \leq \max(p_{d-1},q_{d-1}),
\end{equation} 
with $q_{d} \geq 1/(2d)$.
We now invoke Proposition
\ref{prop:ito:formula} to represent $u(t,p_{1},\cdots,p_{d})$ and $u(t,q_{1},\cdots,q_{d})$ through the respective solutions to 
\eqref{eq:P} together with Proposition  
\ref{prop:representation} above which 
provides another representation for the process used 
in the Kolmogorov formula
\eqref{eq:Kolmogorov:u:5.1}. 
In particular, we can find  
$(P^{\circ,1},\cdots,P^{\circ,m},P^1,\cdots,P^{d-m})$ as in the statement of Proposition
\ref{prop:coupling:2}
such that the tuple
$(P^{\circ,1},\cdots,P^{\circ,m},\varsigma^2(P^\circ) P^{1},\cdots,\varsigma^2(P^\circ) P^{d-m})$ has the same law
as 
the solution to 
\eqref{eq:P}
when starting from 
$p$ at time $t_{0}$,
and, in a similar manner, 
$(Q^{\circ,1},\cdots,Q^{\circ,m},Q^1,\cdots,Q^{d-m})$ such that the tuple
$(Q^{\circ,1},\cdots,Q^{\circ,m},\varsigma^2(Q^\circ) Q^{1},\cdots,\varsigma^2(Q^\circ) Q^{d-m})$ has the same law
as  
the solution to 
\eqref{eq:P}
when starting from 
$q$ at time $t_{0}$. In particular, we have
$(P^{\circ,1}_{t_{0}},\cdots,P^{\circ,m}_{t_{0}})
=
(Q^{\circ,1}_{t_{0}},\cdots,Q^{\circ,m}_{t_{0}})
=p^\circ$, with $p^\circ=
(p_{1},\cdots,p_{m})$, 
$(P^{1}_{t_{0}},\cdots,P^{d-m}_{t_{0}})
=
\varsigma^{-2}(p^\circ)(p_{m+1},\cdots,p_{d})$
and
$(Q^{1}_{t_{0}},\cdots,Q^{d-m}_{t_{0}})
=
\varsigma^{-2}(p^\circ)(q_{m+1},\cdots,q_{d})$.  

Then, for any deterministic time $S \in [t_{0},T]$, using the same notation as in the statement of Proposition 
\ref{prop:coupling:2},
\begin{equation}
\label{eq:kolmogorov:final:coupling}
\begin{split}
&u(t_{0},p_1,\cdots,p_d) 
= {\mathbb E} \Bigl[ u \Bigl(\varpi_{S},P^{\circ,1}_{\varpi_{S}},\cdots,P^{\circ,m}_{\varpi_{S}}, 
\varsigma^2(P_{\varpi_{S}}^{\circ}) P_{\varpi_{S}}^{1},
\cdots,
\varsigma^2(P_{\varpi_{S}}^{\circ}) P_{\varpi_{S}}^{d-m}
\Bigr) \Bigr]
 + O(S-t_{0}),
\\
&u(t_{0},q_1,\cdots,q_d) 
 ={\mathbb E} \Bigl[ u \Bigl(\varpi_{S},Q^{\circ,1}_{\varpi_{S}},\cdots,Q^{\circ,m}_{\varpi_{S}}, 
\varsigma^2(Q_{\varpi_{S}}^{\circ}) Q_{\varpi_{S}}^{1},
\cdots,
\varsigma^2( Q_{\varpi_{S}}^{\circ}) Q_{\varpi_{S}}^{d-m}
\Bigr) \Bigr]
 + O(S-t_{0}),
\end{split}
\end{equation}
where $\vert O(r) \vert \leq \|h \|_{\infty} r$. 
To make it simpler, we also let
(the notation $(X_{t})_{t_{0} \le t \le T}$
below is rather abusive since 
$(X_{t})_{t_{0} \le t \le T}$
also denotes the solution to 
\eqref{eq:Xti:2}, but, in fact, Proposition 
\ref{prop:representation} says both $(X_{t})_{t_{0} \le t \le T}$'s have the same law)
\begin{equation*}
\begin{split}
&X_{t}=
\Bigl(P^{\circ,1}_{t},\cdots,P^{\circ,m}_{t}, 
\varsigma^2(P_{t}^{\circ}) P_{t}^{1},
\cdots,
\varsigma^2(P_{t}^{\circ}) P_{t}^{d-m}
\Bigr),
\\
&Y_{t}=
\Bigl(Q^{\circ,1}_{t},\cdots,Q^{\circ,m}_{t}, 
\varsigma^2(Q_{t}^{\circ}) Q_{t}^{1},
\cdots,
\varsigma^2(Q_{t}^{\circ}) Q_{t}^{d-m}
\Bigr), \quad t \in [t_{0},T].
\end{split}
\end{equation*}
We then denote by $(\max(X_{t},Y_{t})_{(1)},\cdots,\max(X_{t},Y_{t})_{(d)})$ the order statistic of 
the $d$-dimensional tuple 
$(\max(X_{t},Y_{t})_{1},\cdots,\max(X_{t},Y_{t})_{d})$.
\vskip 5pt

\textit{Second Step.} We first assume 
that $S:=t_{0}+\vert p-q \vert^{1/3} \leq T$. 
The strategy is to split into four events the set $\Omega$ over which the expectations 
appearing in 
\eqref{eq:kolmogorov:final:coupling} are computed. 
\vspace{5pt}

\textbf{1st event.} 
On the event $E_{1}:=\{ \varpi_{S}=\tau \} \subset \{(P_{\varpi_{S}}^{1},\cdots,P_{\varpi_{S}}^{d-m}) =
 (Q_{\varpi_{S}}^{1},\cdots,Q_{\varpi_{S}}^{d-m})\}$, we have, by the induction assumption and from the Lipschitz property of $\varsigma^2$, 
\begin{align}
\bigl\vert u\bigl( \varpi_{S}, X_{\varpi_{S}} \bigr) -
u\bigl( \varpi_{S}, Y_{\varpi_{S}} \bigr)  \bigr\vert
&= \Bigl\vert u \Bigl(\varpi_{S},P^{\circ,1}_{\varpi_{S}},\cdots,P^{\circ,m}_{\varpi_{S}}, 
\varsigma^2(P_{\varpi_{S}}^{\circ}) P_{\varpi_{S}}^{1},
\cdots,
\varsigma^2(P_{\varpi_{S}}^{\circ})  P_{\varpi_{S}}^{d-m}
\Bigr) \nonumber
\\
&\hspace{15pt} - 
u \Bigl(\varpi_{S},Q^{\circ,1}_{\varpi_{S}},\cdots,Q^{\circ,m}_{\varpi_{S}}, 
\varsigma^2(Q_{\varpi_{S}}^{\circ}) Q_{\varpi_{S}}^{1},
\cdots,
\varsigma^2(Q_{\varpi_{S}}^{\circ})  Q_{\varpi_{S}}^{d-m}
\Bigr)\Bigr\vert \nonumber
\\
&\leq \frac{C}{\max\bigl(X_{\varpi_{S}},Y_{\varpi_{S}}\bigr)_{(m)}^{\eta \alpha}} \bigl\vert P^\circ_{\varpi_{S}} - Q_{\varpi_{S}}^\circ \bigr\vert^{\alpha},
\label{eq:1st:event}
\end{align}
the constant $C$ being allowed to change from line to line provided that it only depends on the parameters listed in the induction assumption. 

Assume that $\max(X_{\varpi_{S}},Y_{\varpi_{S}})_{(m)} < 
\max(X_{\varpi_{S}}^{l},Y_{\varpi_{S}}^{l})$
 for any $l=m+1,\cdots,d$, then necessarily 
$\max(X_{\varpi_{S}},Y_{\varpi_{S}})_{(m)} \geq \max(X_{\varpi_{S}}^i,Y_{\varpi_{S}}^i)$ for any $i=1,\cdots,m$. 
We then obtain
$C \vert P^\circ_{\varpi_{S}} - Q_{\varpi_{S}}^\circ \vert^{(1-\eta)\alpha}$
 as upper bound for the right-hand side of 
\eqref{eq:1st:event}. 
Therefore,
we can focus on the complementary 
event when 
$\max(X_{\varpi_{S}},Y_{\varpi_{S}})_{(m)} \geq 
\max(X_{\varpi_{S}}^{l},Y_{\varpi_{S}}^{l})$
 for some $l=m+1,\cdots,d$. 
 We obtain 
\begin{equation*}
\begin{split}
\bigl\vert u\bigl( \varpi_{S}, X_{\varpi_{S}} \bigr) -
u\bigl( \varpi_{S}, Y_{\varpi_{S}} \bigr)  \bigr\vert
&\leq 
 C \vert P^\circ_{\varpi_{S}} - Q_{\varpi_{S}}^\circ \vert^{(1-\eta)\alpha} 
  + \sum_{l=m+1}^{d}\frac{C}{\max\bigl(X_{\varpi_{S}}^{l},Y_{\varpi_{S}}^{l}\bigr)^{\eta \alpha}} \bigl\vert P_{\varpi_{S}}^{\circ} -Q_{\varpi_{S}}^{\circ} \bigr\vert^{\alpha},
\end{split}
\end{equation*}
which we rewrite into (recalling that $p_1=q_1$, $\cdots$, $p_m=q_m$)
\begin{equation}
\label{eq:1stevent}
\begin{split}
\bigl\vert u\bigl( \varpi_{S}, X_{\varpi_{S}} \bigr) -
u\bigl( \varpi_{S}, Y_{\varpi_{S}} \bigr)  \bigr\vert
&\leq
C
\bigl\vert 
P_{\varpi_{S}}^{\circ} -
P_{t_{0}}^{\circ} - \bigl( 
Q_{\varpi_{S}}^{\circ}
-
Q_{t_{0}}^{\circ} \bigr) \bigr\vert^{(1-\eta) \alpha}
\\
&\hspace{10pt} +
 \sum_{l=m+1}^{d}\frac{C}{\max\bigl(X_{\varpi_{S}}^{l},Y_{\varpi_{S}}^{l}\bigr)^{\eta \alpha}} \bigl\vert 
P_{\varpi_{S}}^{\circ} -
P_{t_{0}}^{\circ} - \bigl( 
Q_{\varpi_{S}}^{\circ}
-
Q_{t_{0}}^{\circ} \bigr) \bigr\vert^{\alpha}. 
\end{split}
\end{equation}
%
%
%
In order to upper bound $C/(\max(X_{\varpi_{S}}^{l},Y_{\varpi_{S}}^{l}))^{\eta \alpha}$, we expand 
$((X_{t}^{l})^{- 2 \eta \alpha})_{t_{0} \leq t \leq \varpi_{S}}$ by It\^o's formula, for $l= m+1,\cdots,d$. To do so, 
we recall that $(X^1,\cdots,X^d)$ has the same law as the solution of \eqref{eq:Xti:2}. It is an It\^o process with bounded coefficients (in terms of 
$(\|b_{i}\|_{\infty})_{i=1,\cdots,d}$ and $(\|b_i ^{\circ}\|_{\infty})_{i=1,\cdots,d}$). 
Importantly, the drift of the $l$th coordinate is lower bounded by $\kappa - \|b_l^{\circ}\|_{\infty} p_{l}$ when $p_{l} \leq \delta$. 
Recalling that $\kappa_{0} \geq 2 \varepsilon^2$, we easily deduce that, for a new value of $C$ (say $C \geq 1$) whose value is allowed to change from line to line (see also footnote \ref{foot:notation:geq} for the meaning of the inequality right below),
\begin{equation*}
\begin{split}
&\ud \Bigl( \bigl(X_{t}^{l} \bigr)^{-2\eta \alpha}  \bigl(t-t_{0}\bigr)^{\alpha/2}
\Bigr) 
\\
&\leq \Bigl( C \bigl(t-t_{0}\bigr)^{\alpha/2-1} (X_{t}^l)^{-2\eta \alpha}  -  \eta \alpha
\bigl( 2 \kappa
- \varepsilon^2 (1+ 2 \eta \alpha)
\bigr)
  \bigl(t-t_{0}\bigr)^{\alpha/2}\bigl(X_{t}^l\bigr)^{-(1+2\eta \alpha)}
\Bigr) \ud t + \ud m_{t}
\\
&\leq \Bigl( C \bigl(t-t_{0}\bigr)^{\alpha/2-1} (X_{t}^l)^{-2\eta \alpha}  -   C^{-1}
  \bigl(t-t_{0}\bigr)^{\alpha/2}\bigl(X_{t}^l\bigr)^{-(1+2\eta \alpha)}
\Bigr) \ud t + \ud m_{t}
\\
&= \bigl(t-t_{0}\bigr)^{\alpha/2} \bigl(X_{t}^l\bigr)^{-2\eta \alpha}   \Bigl( C \bigl(t-t_{0}\bigr)^{-1} - C^{-1} \bigl(X_{t}^l \bigr)^{-1}
\Bigr) \ud t +  \ud m_{t}
\\
&\leq \bigl(t-t_{0}\bigr)^{\alpha/2} (X_{t}^l)^{-2\eta \alpha}   \Bigl( C \bigl(t-t_{0}\bigr)^{-1} - C^{-1} \bigl(X_{t}^l\bigr)^{-1}
\Bigr) {\mathbf 1}_{\{X_{t}^{l} \geq (t-t_{0})/C^2\}}\ud t 
 + \ud m_{t}
\\
&\leq C \bigl(t-t_{0}\bigr)^{\alpha/2- 1 - 2\eta \alpha} \ud t + \ud m_{t}, 
\end{split}
\end{equation*}
where $(m_{t})_{t \geq 0}$ is a local martingale. 
Recall that $\eta<1/4$ (see the remark at the very beginning of the proof), we get 
by a standard localization argument:
\begin{equation*}
{\mathbb E} \bigl[ (X_{\varpi_{S}}^{l})^{-2 \eta \alpha}  \bigl(\varpi_{S}-t_{0})^{\alpha/2} \bigr] \leq C. 
\end{equation*}
Using the fact that the coefficients 
in the second equation of 
\eqref{eq:new:system:prop:5:2}
are bounded, we deduce that, from Kolmogorov--Centsov theorem,
$P^\circ- Q^\circ$ has a version that is $1/3$-H\"older continuous and the H\"older constant $\Lambda$
has a finite fourth moment, which we may assume to be bounded by $C$. Hence, 
\begin{equation*}
\begin{split}
{\mathbb E} \biggl[ \frac{\vert P_{\varpi_{S}}^\circ - Q_{\varpi_{S}}^\circ - (P_{t_{0}}^\circ - Q_{t_{0}}^\circ) \vert^{2 \alpha}}{(\varpi_{S} - t_{0})^{\alpha/2}} \biggr] 
&\leq {\mathbb E} \bigl[ \Lambda^{2\alpha}\bigl( \varpi_{S} - t_{0}\bigr)^{\alpha/6} \bigr] \leq C {\mathbb E} \bigl[ \bigl( \varpi_{S} - t_{0}\bigr)^{\alpha/3} \bigr]^{1/2}. 
\end{split}
\end{equation*}
Similarly, 
\begin{equation*}
\begin{split}
{\mathbb E} \Bigl[ {\vert P_{\varpi_{S}}^\circ - Q_{\varpi_{S}}^\circ - (P_{t_{0}}^\circ - Q_{t_{0}}^\circ) \vert^{(1-\eta) \alpha}} \Bigr] 
&\leq {\mathbb E} \bigl[ \Lambda^{(1-\eta) \alpha}\bigl( \varpi_{S} - t_{0}\bigr)^{(1-\eta)\alpha/3} \bigr] 
\\
&\leq C {\mathbb E} \bigl[ \bigl( \varpi_{S} - t_{0}\bigr)^{2(1-\eta)\alpha/3} \bigr]^{1/2}. 
\end{split}
\end{equation*}
From 
\eqref{eq:1stevent}, we get (use the condition $\eta <1/4$ to get the last bound)
\begin{align}
{\mathbb E} \Bigl[ \bigl\vert u\bigl( \varpi_{S}, X_{\varpi_{S}} \bigr) -
u\bigl( \varpi_{S}, Y_{\varpi_{S}} \bigr)  \bigr\vert {\mathbf 1}_{E_{1}} \Bigr] 
 &\leq 
 C {\mathbb E} \bigl[ \bigl( \varpi_{S} - t_{0}\bigr)^{2(1-\eta)\alpha/3} \bigr]^{1/2}
 + C {\mathbb E} \bigl[ \bigl( \varpi_{S} - t_{0}\bigr)^{\alpha/3} \bigr]^{1/4} 
 \nonumber
\\
&\leq  
C \vert p-q\vert^{\alpha/36} + C \vert p-q\vert^{(1-\eta) \alpha/9} \leq 
C \vert p-q\vert^{\alpha/36}. 
\label{eq:conclusion:1stevent}
\end{align}
\vskip 4pt

\textbf{2nd event.} On the event 
$E_{2}:= \{ \vert P^\circ_{\varpi_{S}} - Q^\circ_{\varpi_{S}} \vert \geq \vert \widetilde P_{\varpi_{S}} -
\widetilde Q_{\varpi_{S}} \vert \} \supset \{ \varpi_{S} = \varrho\}$ (with the same notation as in the statement of 
Proposition 
\ref{prop:coupling:2}), we have
\begin{equation*}
\begin{split}
\vert X_{\varpi_{S}} - Y_{\varpi_{S}} \vert^2
&= \sum_{i \in \ESm}\bigl\vert P^{\circ,i}_{\varpi_{S}} - Q^{\circ,i}_{\varpi_{S}}  
\bigr\vert^2 + \sum_{i \in \ESdm} \bigl\vert P^{i}_{\varpi_{S}} - Q^{i}_{\varpi_{S}}  
\bigr\vert^2
\\
&=\sum_{i \in \ESm} \bigl\vert P^{\circ,i}_{\varpi_{S}} - Q^{\circ,i}_{\varpi_{S}}  
\bigr\vert^2 + \sum_{i \in \ESdm} 
\bigl\vert \widetilde P^{i}_{\varpi_{S}} + \widetilde Q^{i}_{\varpi_{S}}  
\bigr\vert^2
\bigl\vert \widetilde P^{i}_{\varpi_{S}} - \widetilde Q^{i}_{\varpi_{S}}  
\bigr\vert^2
 \leq 5 \vert P^\circ_{\varpi_{S}} - Q^\circ_{\varpi_{S}} \vert^2. 
\end{split}
\end{equation*}
Modifying the value of the constant $C$ in \eqref{eq:induction}, we deduce that 
\begin{equation*}
\begin{split}
\bigl\vert u\bigl( \varpi_{S}, X_{\varpi_{S}} \bigr) -
u\bigl( \varpi_{S}, Y_{\varpi_{S}} \bigr)  \bigr\vert
&= \Bigl\vert u \Bigl(\varpi_{S},P^{\circ,1}_{\varpi_{S}},\cdots,P^{\circ,m}_{\varpi_{S}}, 
\varsigma^2(P_{\varpi_{S}}^\circ) P_{\varpi_{S}}^{1},
\cdots,
\varsigma^2(P_{\varpi_{S}}^\circ)  P_{\varpi_{S}}^{d-m}
\Bigr)
\\
&\hspace{15pt} - 
u \Bigl(\varpi_{S},Q^{\circ,1}_{\varpi_{S}},\cdots,Q^{\circ,m}_{\varpi_{S}}, 
\varsigma^2(Q_{\varpi_{S}}^{\circ}) Q_{\varpi_{S}}^{1},
\cdots,
\varsigma^2(Q_{\varpi_{S}}^{\circ})  Q_{\varpi_{S}}^{d-m}
\Bigr)\Bigr\vert 
\\
&\leq \frac{C}{\max\bigl(X_{\varpi_{S}},Y_{\varpi_{S}}\bigr)_{(m)}^{\eta \alpha}} \bigl\vert P^\circ_{\varpi_{S}} - Q_{\varpi_{S}}^\circ \bigr\vert^{\alpha},
\end{split}
\end{equation*}
which is the same as 
\eqref{eq:1st:event}. Therefore, we get the same conclusion as in the 
first step, see \eqref{eq:conclusion:1stevent}:
\begin{equation}
\label{eq:conclusion:2ndevent}
\begin{split}
{\mathbb E} \Bigl[ \bigl\vert u\bigl( \varpi_{S}, X_{\varpi_{S}} \bigr) -
u\bigl( \varpi_{S}, Y_{\varpi_{S}} \bigr)  \bigr\vert {\mathbf 1}_{E_{2}} \Bigr] 
 &\leq 
C \vert p-q\vert^{\alpha/36}. 
\end{split}
\end{equation}\vskip 4pt

\textbf{3rd event.}
We now consider the event $E_{3}:=\{ \sup_{t \in [t_{0},S]} \vert P_{t}^\circ \vert_{1} \geq \tfrac34\} \supset \{ \varpi_{S} = \rho\}$, where we recall 
from 
\eqref{eq:1/2md}
that $\vert P_{t_{0}}^\circ \vert_{1} = \sum_{i=1}^m  p_{i} \leq \tfrac1{2d} \leq \tfrac12$. 
Obviously ({since the SDE for $P^\circ$ has bounded coefficients}), there exists a constant $c$ such that ${\mathbb E}[ \sup_{t \in [t_{0},S]}\vert P_{t}^\circ - P_{t_{0}}^\circ \vert^2 ] \leq c (S-t_{0}) = c \vert p-q \vert^{1/3}$. 
Therefore, using the fact that 
$u$ is bounded
together with Markov's inequality, we deduce that  
\begin{equation}
\label{eq:conclusion:3rdevent}
\begin{split}
&{\mathbb E} \Bigl[ \bigl\vert u\bigl( \varpi \wedge S, X_{\varpi \wedge S} \bigr) -
u\bigl( \varpi \wedge S, Y_{\varpi \wedge S} \bigr)  \bigr\vert {\mathbf 1}_{E_{3}} \Bigr]  
\leq C \vert p-q \vert^{1/3}.
\end{split}
\end{equation}
\vskip 4pt

\textbf{4th event.} Lastly, we let $E_{4}: = \{ \varpi_{S} < \tau \wedge \varrho \wedge \rho   \}$.  
Since $\varsigma^{-2}(p^\circ)=\varsigma^{-2}(q^\circ) \leq 2d$ 
and 
$\vert p-q \vert < \delta^2/(128 d^{3/2})$, we have $\vert 
\varsigma^{-2}(p^\circ)(p_{m+1},\cdots,p_{d})
-
\varsigma^{-2}(q^\circ)(q_{m+1},\cdots,q_{d})\vert < \delta^2/(64\sqrt{d})$. Hence, 
by Proposition \ref{prop:coupling:2}
(with $p$ therein being given by $\varsigma^{-2}(p^\circ)(p_{m+1},\cdots,p_{d})$
and similarly for $q$) and with $\eta$ therein being replaced by $\eta/12$, we know that, for a new value of $C$,  
\begin{equation*}
\begin{split}
{\mathbb P} \bigl( E_{4} \bigr)   \leq 
C  \frac{\vert p -q \vert^{1/12}}{\min_{i=m+1,\cdots,d}(\max(p_i,q_i))^{\eta/12}}
&\leq C  \frac{\vert p -q \vert^{1/12}}{\max(p_{m+1},q_{m+1})^{\eta/12}}
=C  \frac{\vert p -q \vert^{1/12}}{\max(p,q)_{(m+1)}^{\eta/12}},
\end{split} 
\end{equation*}
where the derivation of the last two terms follows from 
\eqref{eq:increasing:order:b} and from the condition $\max(p_{d},q_{d}) \geq 1/(2d)$.
Since the left-hand side is less than 1, we deduce that, for any exponent $\beta \in (0,1]$,
\begin{equation*}
{\mathbb P} \bigl( E_{4} \bigr)   \leq
{\mathbb P} \bigl( E_{4} \bigr)^{\beta}
 \leq 
 C^\beta  \frac{\vert p -q \vert^{\beta/12}}{\max(p,q)_{(m+1)}^{\beta \eta/12}},
\end{equation*}
and then, for a new value of $C$ possibly depending on $\beta$, 
\begin{equation}
\label{eq:conclusion:4thevent}
\begin{split}
&{\mathbb E} \Bigl[ \bigl\vert u\bigl( \varpi \wedge S, X_{\varpi \wedge S} \bigr) -
u\bigl( \varpi \wedge S, Y_{\varpi \wedge S} \bigr)  \bigr\vert {\mathbf 1}_{E_{4}} \Bigr] 
\leq 
C  \frac{\vert p -q \vert^{\beta/12}}{\max(p,q)_{(m+1)}^{\beta \eta/12}}.
\end{split}
\end{equation}

\vskip 4pt

\textbf{Conclusion.}
Here is now the conclusion of the second step. 
For the same $\eta$ and $\alpha$ as before, 
choose the largest $\beta \in (0,1]$ such that 
$\beta/12 \leq \alpha/36$.
Finally, let 
$\alpha'=\beta/12$. 
Deduce that, for a possibly new value of the constant $C$ therein, all the terms in 
\eqref{eq:conclusion:1stevent}, 
\eqref{eq:conclusion:2ndevent}, 
\eqref{eq:conclusion:3rdevent}
and 
\eqref{eq:conclusion:4thevent}
are bounded by 
$C \vert p -q \vert^{\alpha'} / \max(p,q)_{(m+1)}^{\eta\alpha' }$. 
Since $E_{1} \cup E_{2} \cup E_{3} \cup E_{4}=\Omega$, we deduce that 
\begin{equation*}
\begin{split}
&{\mathbb E} \Bigl[ \bigl\vert u\bigl( \varpi \wedge S, X_{\varpi \wedge S} \bigr) -
u\bigl( \varpi \wedge S, Y_{\varpi \wedge S} \bigr)  \bigr\vert  \Bigr] 
\leq 
C  \frac{\vert p -q \vert^{\alpha'}}{\max(p,q)_{(m+1)}^{\alpha' \eta}},
\end{split}
\end{equation*} 
which together with \eqref{eq:kolmogorov:final:coupling}, is ${\mathscr P}_{m+1}$, at least for initial conditions 
$(t_{0},p)$ and 
$(t_{0},q)$ such that $T-t_{0} \geq \vert p-q\vert^{1/3}$, 
and $\eta <1/4$. As for the requirement $\eta< 1/4$, this is not a hindrance: Since the denominator in \eqref{eq:induction} is less than 1, the exponent can be increased for free. 
As for the case
$T-t_{0} < \vert p-q\vert^{1/3}$, it is discussed in the next step. 

\vskip 4pt

\textit{Third Step.} It remains to handle the case  
that $t_{0}+\vert p-q \vert^{1/3} \geq T$. 
We then 
rewrite
\eqref{eq:kolmogorov:final:coupling} 
in the form
\begin{equation*}
\begin{split}
&u(t_{0},p_1,\cdots,p_d) 
= {\mathbb E} \Bigl[ \ell \Bigl(P^{\circ,1}_{T},\cdots,P^{\circ,m}_{T}, 
\varsigma^2(P_{T}^{\circ}) P_{T}^{1},
\cdots,
\varsigma^2(P_{T}^{\circ}) P_{T}^{d-m}
\Bigr) \Bigr]
 + O(T-t_{0}),
\\
&u(t_{0},q_1,\cdots,q_d) 
 ={\mathbb E} \Bigl[ \ell \Bigl(Q^{\circ,1}_{T},\cdots,Q^{\circ,m}_{T}, 
\varsigma^2(Q_{T}^{\circ}) Q_{T}^{1},
\cdots,
\varsigma^2( Q_{T}^{\circ}) Q_{T}^{d-m}
\Bigr) \Bigr]
 + O(T-t_{0}).
\end{split}
\end{equation*}
By 
\eqref{eq:identity:law}, each expectation may be (directly) rewritten by means of the 
solution to the stochastic differential equation \eqref{eq:P}. Since the latter has bounded coefficients and 
since $\ell$ is Lipschitz continuous, we deduce that 
\begin{equation*}
\vert u(t_{0},p) - u(t_{0},q) \vert 
\leq 
\vert \ell(p) - \ell(q) \vert
+
\vert u(t_{0},p) - \ell(p) \vert  
+
\vert u(t_{0},q) - \ell(q) \vert  
 \leq C \vert p-q \vert + C (T-t_{0})^{1/2},
\end{equation*}
for a constant $C$ that only depends on $\kappa$, 
$\| \ell \|_{1,\infty}$, 
$(\| b_{i} \|_{\infty})_{i \in \ES}$
and 
$(\| b_{i}^\circ \|_{\infty})_{i \in \ES}$.
Since $T-t_{0} \leq \vert p-q \vert^{1/3}$, 
this completes the proof of ${\mathscr P}_{m+1}$
in the remaining case when
$T-t_{0} \leq \vert p-q\vert^{1/3}$. 
%
%
\end{proof}

%


 \subsection{Proof of the coupling property}
\label{subse:proof:coupling:property}
We now prove 
Proposition \ref{prop:coupling:2} by means of a reflection coupling, inspired by 
\cite{ChenLi,LindvallRogers}. 
Throughout, we use the notation 
$Z_{t} := \widetilde{P}_{t} - \widetilde{Q}_{t}$, for $t \in [0,T]$.

 \subsubsection{Preliminary result}

\begin{proposition}
\label{prop:coupling:1}
Under the same assumption and notation as in the statement of 
Proposition \ref{prop:coupling:2},
%
%
there exists 
a constant $C$, only depending on $\delta$,  $\kappa$, $(\|b_{i}\|_{\infty})_{i \in \ES}$, 
$(\|b_i ^{\circ}\|_{\infty})_{i \in \ES}$
and $T$ such that (see footnote \ref{foot:notation:geq} for the meaning of the inequality right below), 
\begin{equation}
\label{eq:Z:coupling}
\begin{split}
\ud \vert Z_{t} \vert 
&\leq C \ud t + 
\sum_{i \in \ESdm} \frac{Z_{t}^i}{\vert Z_{t} \vert} \frac{\widetilde \beta^i_{t}}{\max(\widetilde P_{t}^i,\widetilde Q_{t}^i)} 
\ud t + \varepsilon \frac{\varsigma^{-1}(P_{t}^\circ) + \varsigma^{-1}(Q_{t}^\circ)}{2}
\sum_{i,j \in \ESdm : i \not = j}
\frac{Z_{t}^{i}}{\vert Z_{t} \vert} \ud \overline W_{t} ^{i,j}\widetilde P_{t}^j, 
\end{split}
\end{equation}
for $t \in [t_{0},\varpi \wedge T)$,
where for each $i \in \ESdm$, $(\beta^i_{t})_{0 \leq t \leq T}$ is a progressively measurable process that is dominated by 
$4 \| \widetilde b_{i} \|_{\infty}$.
\end{proposition}
\begin{proof}
Throughout the proof, we use the convention $\overline W^{i,i}_{t}: =0$, for $t \in [0,T]$ and $i \in \ESdm$.
\vskip 4pt

\textit{First Step.}
The first step is to perform a change of variable in equation 
\eqref{eq:new:system:prop:5:2}, letting therein 
\begin{equation*}
\widetilde P_{t}^i := \sqrt{P_{t}^i}, \quad t \in [t_{0},T], \ i \in \ESdm. 
\end{equation*}
By It\^o's formula, we get
\begin{equation}
\label{eq:tildeP:coupling}
\begin{split}
\ud \widetilde P_{t}^i &= \varsigma^{-2}(P_{t}^\circ) \Bigl( \frac{ \varphi(\varsigma^2(P_{t}^\circ)P_{t}^i) +  \widetilde b_{i} (t, P_{t}^\circ, P_{t})  + (\widetilde P_{t}^i)^2 \widetilde b_i^{\circ}(t,P_{t}^\circ,P_{t}) }{2 \widetilde P_{t}^i}
- 
\frac{\varepsilon^2}{8 \widetilde P_{t}^i }
+ \frac{\varepsilon^2}8 \widetilde P_{t}^i \Bigr) \ud t
\\
&\hspace{15pt}
+ \varepsilon \frac{\varsigma^{-1}(P_{t}^\circ)}2 \sum_{j \in \ESdm} \widetilde P_{t}^j \ud \overline W_{t}^{i,j},
\end{split} 
\end{equation}
for $t \in [t_{0},T]$, which prompts us to let
\begin{equation*}
\widetilde {B}_{i}\bigl(t,r^\circ,r\bigr)
:=  \frac{\varphi(\varsigma^2(r^\circ)r_{i}) +\widetilde b_{i}(t,r^\circ,r)  + r_i \widetilde b_i^{\circ}(t,r^\circ,r) }{2 \sqrt{r_{i}}}
- 
\frac{\varepsilon^2}{8 \sqrt{r_{i}} }
+ \frac{\varepsilon^2}8 \sqrt{r_{i}}, 
\end{equation*}
for $r^\circ \in \hat{\mathcal S}_{m}$ and $r \in {\mathcal S}_{d-m-1}$. Denoting by $B^\circ$ and $\Sigma^\circ$ the drift and diffusion coefficients in the dynamics of 
$P^\circ$ in equation \eqref{eq:new:system:prop:5:2}, 
we then look at the solutions of the coupled SDEs:
\begin{equation}
\label{eq:coupled:sde}
\begin{split}
&\ud \widetilde P_{t}^i = \varsigma^{-2}(P_{t}^\circ) \widetilde{B}_{i} (t,P_{t}^\circ,  P_{t} ) \ud t
+\varepsilon \frac{\varsigma^{-1}(P_{t}^\circ)}2 \sum_{j \in \ESdm} \widetilde{P}_{t}^j \ud \overline W_{t}^{i,j}
\\
&\ud P_{t}^\circ = B^\circ(t,P_{t}^\circ,P_t) \ud t + \Sigma^{\circ}(t,P_{t}^\circ,P_{t}) \ud W_{t}^\circ
\\
&\ud \widetilde Q_{t}^i = \varsigma^{-2}(Q_{t}^\circ) \widetilde{B}_{i} (t,Q_{t}^\circ, Q_{t} ) \ud t
+\varepsilon \frac{\varsigma^{-1}(Q_{t}^\circ)}2 \sum_{j \in \ESdm} \widetilde{Q}_{t}^j \bigl( R_{t} \ud \overline W_{t} R_{t} \bigr)^{i,j},
\\
&\ud Q_{t}^\circ = B^\circ(t,Q_{t}^\circ,Q_t) \ud t + \Sigma^{\circ}(t,Q_{t}^\circ,Q_{t}) \ud W_{t}^\circ,
\end{split}
\end{equation}
where $ Q_{t}^i := (\widetilde{Q_{t}^i})^2$, for 
$t \in [t_{0},T]$ and $i \in \ESdm$, and where $R_{t}$ denotes the reflection matrix:
\begin{equation}
\label{eq:def:Rt}
R_{t} := I_{d-m} - 2 \frac{(\widetilde{P}_{t} - \widetilde{Q}_{t})(\widetilde P_{t} - \widetilde Q_{t})^\dagger}{\vert 
\widetilde P_{t} - \widetilde Q_{t} \vert^2} {\mathbf 1}_{\{t < \tau\}}, \quad t \in [ {t_{0}},T],  
\end{equation}
with $\tau := \inf \{t \geq t_{0} : \widetilde P_{t} = \widetilde Q_{t} \}$,
$I_{d-m}$ standing for the identity matrix of dimension $d-m$.
The initial conditions are 
$\widetilde P_{t_{0}}=\widetilde{p}:=\sqrt{p}$,
$P^\circ_{t_{0}}=p^\circ$, 
$\widetilde Q_{t_{0}}=\widetilde{q}:=\sqrt{q}$
and 
$Q^\circ_{t_{0}}=q^\circ$,
for the same $p$, $p^\circ$, $q$ and $q^\circ$ as in the statement of Proposition  
\ref{prop:coupling:2}.

We claim that \eqref{eq:coupled:sde} is  {uniquely solvable in the strong sense, see Lemma \ref{lem:coupled:sde:1}}. Importantly, we prove in Lemma
\ref{lem:coupled:sde:2} below that 
$( \int_{t_{0}}^t R_{s} \ud \overline{W}_{s} R_{s}
)_{t \geq t_{0}}$
is an antisymmetric Brownian motion of dimension $(d-m)(d-m-1)/2$.
Since \eqref{eq:new:system:prop:5:2}
is uniquely solvable, this proves that the law of
$(Q_{t}^\circ,\widetilde Q_{t})_{t_{0} \le t \le T}$ coincides  
with the law of the solution to 
the first two equations in \eqref{eq:coupled:sde} when the latter are initiated from $(q^\circ,\sqrt{q})$ at time $t_{0}$. In other words,
 the law of $(Q_{t}^\circ,Q_{t})_{t_{0} \le t \le T}$ coincides  
with the law of the solution to 
\eqref{eq:new:system:prop:5:2} when the latter is initiated from $(q^\circ,{q})$ at time $t_{0}$.   
We then 
define the stopping times $\rho$, $\varrho$ and $\sigma$ as in the statement
of 
Proposition \ref{prop:coupling:2}. 
\vskip 5pt

\textit{Second Step.}
We now have a look at $(Z_{t}
 = \widetilde P_{t} - \widetilde Q_{t})_{t_{0} \leq t \leq T}$.
 Using the fact that $R_{t} \widetilde Q_{t}= \widetilde P_{t}$, for $t \in [t_{0},\tau \wedge T)$, we get 
 \begin{equation*}
\begin{split}
\ud \overline{W}_{t} \widetilde{P}_{t} - R_{t} \ud \overline{W}_{t} R_{t} \widetilde{Q}_{t}
&= \ud \overline{W}_{t} \widetilde{P}_{t} - R_{t} \ud\overline{W}_{t} \widetilde{P}_{t}
 = 2 \frac{Z_{t} Z_{t}^{\dagger}}{\vert Z_{t} \vert^2} \ud \overline{W}_{t} \widetilde P_{t}. 
\end{split}
\end{equation*}
%
%
We deduce the following expression
\begin{equation*}
\begin{split}
&\frac{Z_{t}^{\dagger}}{\vert Z_{t} \vert} \bigl(\ud \overline{W}_{t} \widetilde{P}_{t} - R_{t} \ud\overline{W}_{t} R_{t} \widetilde{Q}_{t}\bigr)
=  
2  \frac{Z_{t}^{\dagger}}{\vert Z_{t} \vert} \ud \overline{W}_{t}
  \widetilde P_{t}.
\end{split}
\end{equation*}
Now, we may compute the bracket of the above right-hand side. We get
\begin{align}
\label{eq:coupling:covariance}
\frac{1}{\ud t}\Bigl\langle \frac{Z_{t}^{\dagger}}{\vert Z_{t} \vert} \ud \overline{W}_{t}
  \widetilde P_{t} \Bigr\rangle
&= \sum_{i,j \in \ESdm} \sum_{i',j' \in \ESdm}
 \frac{Z_{t}^{i} Z_{t}^{i'}}{\vert Z_{t} \vert^2} \bigl( \delta_{i,i'} \delta_{j,j'}
 - \delta_{i,j'} \delta_{j,i'}
 \bigr)
  \widetilde P_{t}^j \widetilde P_{t}^{j'} 
  \\
  &=  1 - 
\bigl\langle \frac{Z_{t}}{\vert Z_{t} \vert},\widetilde P_{t} \bigr\rangle^2
 =    1 -\frac{ (
1-
\langle \widetilde Q_{t},\widetilde P_{t} \rangle )^2}{2 - 2 \langle \widetilde Q_{t}, \widetilde P_{t} \rangle} 
=   1 -
\frac12
\bigl( 
1-
\langle \widetilde Q_{t},\widetilde P_{t} \rangle \bigr)
= 1 -
\frac14 \vert Z_{t} \vert^2, \nonumber
\end{align}
which holds true for $t  < T \wedge \tau$. 
More generally, we need to compute the bracket of $(\int_{t_{0}}^{t \wedge \tau} \ud \overline{W}_{s} \widetilde{P}_{s} - R_{s} d\overline{W}_{s} R_{s} \widetilde{Q}_{s})_{t_{0} \le t \le T}$. For $i,j \in \ESdm$ and $t \in [t_{0},\tau)$, we have
\begin{equation*}
\begin{split}
&\bigl\langle \bigl( \ud \overline {W}_{t} \widetilde P_{t} - R_{t} \ud \overline W_{t} R_{t} \widetilde Q_{t} \bigr)^i,
\bigl( \ud \overline{W}_{t} \widetilde P_{t} - R_{t} \ud \overline W_{t} R_{t} \widetilde Q_{t} \bigr)^j \bigr\rangle
= 4 \Bigl\langle \Bigl( 
\frac{Z_{t} Z_{t}^{\dagger}}{\vert Z_{t} \vert^2} \ud \overline{W}_{t} \widetilde P_{t}
  \Bigr)^i,
\Bigl( 
\frac{Z_{t} Z_{t}^{\dagger}}{\vert Z_{t} \vert^2} \ud \overline{W}_{t} \widetilde P_{t}
\Bigr)^j \Bigr\rangle
\end{split}
\end{equation*}
Here,
\begin{equation*}
\begin{split}
&\Bigl\langle \Bigl( 
\frac{Z_{t} Z_{t}^{\dagger}}{\vert Z_{t} \vert^2} \ud \overline{W}_{t} \widetilde P_{t}
  \Bigr)^i,
\Bigl( 
\frac{Z_{t} Z_{t}^{\dagger}}{\vert Z_{t} \vert^2} \ud \overline{W}_{t} \widetilde P_{t}
\Bigr)^j \Bigr\rangle
= \sum_{k,l \in \ESdm} 
\frac{Z_{t}^i Z_{t}^j Z_{t}^{k} Z_{t}^{l}}{\vert Z_{t} \vert^4} 
\bigl\langle \bigl( \ud \overline {W}_{t} \widetilde P_{t}
\bigr)^k,\bigl( \ud \overline{W}_{t} \widetilde P_{t}  \bigr)^{l}
\bigr\rangle.
\end{split}
\end{equation*}
Now, 
\begin{equation*}
\begin{split}
\frac{1}{\ud t}\bigl\langle \bigl( \ud \overline {W}_{t} \widetilde P_{t}
\bigr)^i,\bigl( \ud \overline{W}_{t} \widetilde P_{t}  \bigr)^j
\bigr\rangle
&= \sum_{k,l \in \ESdm} \frac{1}{\ud t} \bigl\langle 
\ud \overline {W}_{t}^{i,k} \widetilde P_{t}^k,
\ud \overline {W}_{t}^{j,l} \widetilde P_{t}^l
\bigr\rangle
\\
&= \sum_{k,l \in \ESdm}  \widetilde P_{t}^k \widetilde P_{t}^l
\bigl( \delta_{i,j} \delta_{k,l} - \delta_{i,l} \delta_{j,k} \bigr) 
= \delta_{i,j} - \widetilde P_{t}^i \widetilde P_{t}^j. 
\end{split}
\end{equation*}
So, 
\begin{equation*}
\begin{split}
\frac{1}{\ud t} \Bigl\langle \Bigl( 
\frac{Z_{t} Z_{t}^{\dagger}}{\vert Z_{t} \vert^2} \ud \overline{W}_{t} \widetilde P_{t}
  \Bigr)^i,
\Bigl( 
\frac{Z_{t} Z_{t}^{\dagger}}{\vert Z_{t} \vert^2} \ud \overline{W}_{t} \widetilde P_{t}
\Bigr)^j \Bigr\rangle
&= \sum_{k,l \in \ESdm} 
\frac{Z_{t}^i Z_{t}^j Z_{t}^{k} Z_{t}^{l}}{\vert Z_{t} \vert^4} 
\Bigl( \delta_{k,l} - \widetilde P_{t}^k \widetilde P_{t}^l \Bigr)
\\
&= \frac{Z_{t}^i Z_{t}^j}{\vert Z_{t} \vert^2} - 
\frac{Z_{t}^i Z_{t}^j}{\vert Z_{t} \vert^4} \bigl\langle Z_{t},\widetilde P_{t} \bigr\rangle^2.
\end{split}
\end{equation*}
And then,
\begin{equation}
\label{eq:coupling:covariance:2}
\begin{split}
&\frac{1}{\ud t} \bigl\langle \bigl( \ud \overline {W}_{t} \widetilde P_{t} - R_{t} \ud \overline W_{t} R_{t} \widetilde Q_{t} \bigr)^i,
\bigl( \ud \overline{W}_{t} \widetilde P_{t} - R_{t} \ud \overline W_{t} R_{t} Q_{t} \bigr)^j \bigr\rangle
= 4 \Bigl(
\frac{Z_{t}^i Z_{t}^j}{\vert Z_{t} \vert^2} - 
\frac{Z_{t}^i Z_{t}^j}{\vert Z_{t} \vert^4} \bigl\langle Z_{t},\widetilde P_{t} \bigr\rangle^2
\Bigr).
\end{split}
\end{equation}
\vspace{5pt}

\textit{Third Step.}
Now, we return to the equation satisfied by $(Z_{t})_{t_{0} \leq t \leq T}$:
\begin{equation*}
\begin{split}
\ud  Z_{t} 
&=
\bigl[\varsigma^{-2}(P_{t}^\circ) \widetilde{B}(t,P_{t}^\circ,P_{t}) - \varsigma^{-2}(Q_{t}^\circ) \widetilde{B}(t,Q_{t}^\circ,Q_{t})
\bigr] \ud t 
\\
&\hspace{15pt}+ \varepsilon \frac{\varsigma^{-1}(P_{t}^\circ)}2 
 \ud \overline{W}_{t} \widetilde{P}_{t} - \varepsilon \frac{\varsigma^{-1}(Q_{t}^\circ)}2  R_{t} \ud \overline{W}_{t} R_{t} \widetilde{Q}_{t},   \end{split}
\end{equation*}
for $t \in [0,T]$.
The point is to apply It\^o's formula to $\vert Z \vert$. 
Using 
the relationship $R_{t} \widetilde Q_{t}  = \widetilde P_{t}$ together with 
the fact that $Z_{t}^{\dagger} R_{t} = (R_{t} Z_{t})^{\dagger} = -Z_{t}^{\dagger}$, we get, for $t < \varpi \wedge T$, 
\begin{equation*}
\begin{split}
\ud \vert Z_{t} \vert 
&=
\frac{Z_{t}^{\dagger}}{\vert Z_{t} \vert}  \bigl[\varsigma^{-2}(P_{t}^\circ) \widetilde{B}(t,P_{t}^\circ, P_{t}) -
\varsigma^{-2}(Q_{t}^\circ) \widetilde{B}(t,Q_{t}^\circ, Q_{t})
\bigr] \ud t
 + \varepsilon \frac{\varsigma^{-1}(P_{t}^\circ) +  \varsigma^{-1}(Q_{t}^\circ)}2
\frac{Z_{t}^{\dagger}}{\vert Z_{t} \vert} \ud \overline{W}_{t} \widetilde{P}_{t}
  \\
&\hspace{5pt}   + \frac{\varepsilon^2}8 
  \sum_{i,j \in \ESdm} \Bigl[
  \Bigl\langle
  \Bigl(
\varsigma^{-1}(P_{t}^\circ) 
  \ud \overline{W}_{t} \widetilde{P}_{t} - \varsigma^{-1}(Q_{t}^\circ) 
 R_{t} \ud \overline{W}_{t} R_{t} \widetilde{Q}_{t}
 \Bigr)^i,
   \\
   &\hspace{100pt}
\Bigl(
\varsigma^{-1}(P_{t}^\circ) 
  \ud \overline{W}_{t} \widetilde{P}_{t} - \varsigma^{-1}(Q_{t}^\circ) 
 R_{t} \ud \overline{W}_{t} R_{t} \widetilde{Q}_{t}\Bigr)^j
 \Bigr\rangle  
  \times \frac1{\vert Z_{t}\vert}\Bigl( \delta_{i,j} -     \frac{Z_{t}^i Z_{t}^j}{\vert Z_{t} \vert^2}
\Bigr) \Bigr].
\end{split}
\end{equation*}
At this stage of the proof, we have a special look at the brackets in the last two lines in the above expression. By Kunita--Watanabe inequality, we get, for $t < \varpi \wedge T$, 
\begin{equation*}
\begin{split}
 &\Bigl\langle
  \Bigl(
\varsigma^{-1}(P_{t}^\circ) 
  \ud \overline{W}_{t} \widetilde{P}_{t} - \varsigma^{-1}(Q_{t}^\circ) 
 R_{t} \ud \overline{W}_{t} R_{t} \widetilde{Q}_{t}
 \Bigr)^i,
\Bigl(
\varsigma^{-1}(P_{t}^\circ) 
  \ud \overline{W}_{t} \widetilde{P}_{t} - \varsigma^{-1}(Q_{t}^\circ) 
 R_{t} \ud \overline{W}_{t} R_{t} \widetilde{Q}_{t}\Bigr)^j
 \Bigr\rangle 
\\
&=\varsigma^{-2}(P_{t}^\circ) 
\Bigl\langle
  \Bigl(
  \ud \overline{W}_{t} \widetilde{P}_{t} -
   R_{t}\ud \overline{W}_{t} R_{t} \widetilde{Q}_{t}
 \Bigr)^i,
\Bigl(
  \ud \overline{W}_{t} \widetilde{P}_{t} -  
 R_{t} \ud \overline{W}_{t} R_{t} \widetilde{Q}_{t}\Bigr)^j
 \Bigr\rangle 
+ O \Bigl( \bigl\vert \varsigma^{-1}( P_{t}^\circ) - \varsigma^{-1}(Q_{t}^\circ) \bigr\vert
\Bigr) \ud t,
\end{split}
\end{equation*}
where $(O ( \vert \varsigma^{-1}( P_{t}^\circ) - \varsigma^{-1}(Q_{t}^\circ) \vert
))_{0 \leq t < \varpi \wedge T}$ stands for a progressively measurable process that is dominated by $C \bigl( \vert \varsigma^{-1}(P_{t}^\circ) - \varsigma^{-1}(Q_{t}^\circ) \vert \bigr)_{0 \leq t < \varpi \wedge T}$
for a universal constant $C$.
Invoking \eqref{eq:coupling:covariance:2}, we get 
\begin{equation*}
\begin{split}
 &\Bigl\langle
  \Bigl(
\varsigma^{-1}(P_{t}^\circ) 
  \ud \overline{W}_{t} \widetilde{P}_{t} - \varsigma^{-1}(Q_{t}^\circ) 
 R_{t} \ud \overline{W}_{t} R_{t} \widetilde{Q}_{t}
 \Bigr)^i,
\Bigl(
\varsigma^{-1}(P_{t}^\circ) 
  \ud \overline{W}_{t} \widetilde{P}_{t} - \varsigma^{-1}(Q_{t}^\circ) 
 R_{t} \ud \overline{W}_{t} R_{t} \widetilde{Q}_{t}\Bigr)^j
 \Bigr\rangle 
\\
&=4 \varsigma^{-2}(P_{t}^\circ) 
\Bigl(
\frac{Z_{t}^i Z_{t}^j}{\vert Z_{t} \vert^2} - 
\frac{Z_{t}^i Z_{t}^j}{\vert Z_{t} \vert^4} \bigl\langle Z_{t},\widetilde P_{t} \bigr\rangle^2
\Bigr) \ud t
+ O \Bigl( \bigl\vert P_{t}^\circ - Q_{t}^\circ \bigr\vert
\Bigr) \ud t
\\
&=4 \varsigma^{-2}(P_{t}^\circ) 
\Bigl(
\frac{Z_{t}^i Z_{t}^j}{\vert Z_{t} \vert^2} - 
\frac{Z_{t}^i Z_{t}^j}{\vert Z_{t} \vert^4} \bigl\langle Z_{t},\widetilde P_{t} \bigr\rangle^2
\Bigr) \ud t
+ O \bigl( \vert Z_{t} \vert
\bigr) \ud t,
\end{split}
\end{equation*}
where we used the fact that $t < \varrho$ to derive the last line of the statement. As before, 
the process
$(O ( \vert Z_{t} \vert
))_{0 \leq t < \varpi \wedge T}$ is a progressively measurable process that is dominated by $(C \vert Z_{t} \vert \bigr)_{0 \leq t < \varpi \wedge T}$
for a universal constant $C$. Returning to the expression for $d \vert Z_{t} \vert$, we then get that 
\begin{align}
&\frac18 
  \sum_{i,j \in \ESdm} \Bigl[
  \Bigl\langle
  \Bigl(
\varsigma^{-1}(P_{t}^\circ) 
  \ud \overline{W}_{t} \widetilde{P}_{t} - \varsigma^{-1}(Q_{t}^\circ) 
 R_{t} \ud \overline{W}_{t} R_{t} \widetilde{Q}_{t}
 \Bigr)^i, \nonumber
    \\
   &\hspace{100pt} 
\Bigl(
\varsigma^{-1}(P_{t}^\circ) 
  \ud \overline{W}_{t} \widetilde{P}_{t} - \varsigma^{-1}(Q_{t}^\circ) 
 R_{t} \ud \overline{W}_{t} R_{t} \widetilde{Q}_{t}\Bigr)^j
 \Bigr\rangle   \times \frac1{\vert Z_{t}\vert}\Bigl( \delta_{i,j} -     \frac{Z_{t}^i Z_{t}^j}{\vert Z_{t} \vert^2}
\Bigr) \Bigr] \label{eq:bracket:last:line}
\\
&= \frac1{2\vert Z_{t}\vert} \varsigma^{-2}(P_{t}^\circ) 
\sum_{i,j \in \ESdm}
\Bigl(
\frac{Z_{t}^i Z_{t}^j}{\vert Z_{t} \vert^2} - 
\frac{Z_{t}^i Z_{t}^j}{\vert Z_{t} \vert^4} \bigl\langle Z_{t},\widetilde P_{t} \bigr\rangle^2
\Bigr)
\Bigl( \delta_{i,j} -     \frac{Z_{t}^i Z_{t}^j}{\vert Z_{t} \vert^2}
\Bigr) \ud t
+ O(1) \ud t, \nonumber
\end{align}
where 
$(O (1))_{0 \leq t < \varpi \wedge T}$ stands for a progressively measurable process that is dominated by $C$
for a universal constant $C$. The key fact here is that 
\begin{equation*}
\sum_{i,j \in \ESdm}
\frac{Z_{t}^i Z_{t}^j}{\vert Z_{t} \vert^2}
\Bigl( \delta_{i,j} -     \frac{Z_{t}^i Z_{t}^j}{\vert Z_{t} \vert^2}
\Bigr)
= \sum_{i \in \ESdm} \frac{(Z_{t}^i)^2}{\vert Z_{t} \vert^2}
- \biggl( \sum_{i \in \ESdm} \frac{(Z_{t}^i)^2}{\vert Z_{t} \vert^2} \biggr)^2 
= 1 - 1 =0. 
\end{equation*}
We deduce that the last line in \eqref{eq:bracket:last:line} reduces to $O(1)\ud t$. We end up with 
\begin{equation*}
\begin{split}
\ud \vert Z_{t} \vert &=
\frac{Z_{t}^{\dagger}}{\vert Z_{t} \vert} \bigl[ \varsigma^{-2}(P_{t}^\circ) \widetilde{B}(t,P_{t}^\circ,  P_{t}) - 
\varsigma^{-2}(Q_{t}^\circ)\widetilde{B}(t,Q_{t}^\circ,  Q_{t})
\bigr] \ud t
  + \varepsilon \frac{\varsigma^{-1}(P_{t}^\circ) + \varsigma^{-1}(Q_{t}^\circ)}2 \frac{Z_{t}^{\dagger}}{\vert Z_{t} \vert} \ud \overline{W}_{t}
  \widetilde P_{t} 
  \\
&\hspace{15pt}  + O(1) \, \ud t, \quad t < \varpi \wedge T. 
  \end{split}
  \end{equation*}
  \vskip 4pt

\textit{Fourth Step.} We now have a look at the drift in a more precise way. 
Recalling that 
\begin{equation*}
\widetilde B_{i}(t,P_{t}^\circ, P_{t}) = \frac{\varphi(\varsigma^2(P_{t}^\circ) P_{t}^i) + \widetilde b_{i}(t,P_{t}^\circ, P_{t})  + (\widetilde P_{t}^i)^2 \widetilde b_i^{\circ}(t,P_{t}^\circ,P_{t})}{2 \widetilde P_{t}^i}
- 
\frac{\varepsilon^2}{8 \widetilde P_{t}^i }
+ \frac{\varepsilon^2}8 \widetilde P_{t}^i,
\end{equation*}
we write
\begin{equation*}
\begin{split}
\ud \vert Z_{t} \vert &=
\sum_{i \in \ESdm} \frac{Z_{t}^{i}}{\vert Z_{t} \vert} \Bigl[ \varsigma^{-2}(P_{t}^\circ)
\frac{  \varphi(\varsigma^2(P_{t}^\circ) P_{t}^i) + \widetilde b_{i}(t,P_{t}^\circ, P_{t}) - \varepsilon^2/4}{2 \widetilde P_{t}^i} 
\\
&\hspace{50pt}- \varsigma^{-2}(Q_{t}^\circ)
\frac{ \varphi(\varsigma^2(Q_{t}^\circ) Q_{t}^i) + 
\widetilde b_{i}(t,Q_{t}^\circ, Q_{t}) - \varepsilon^2/4}{2 \widetilde Q_{t}^i} 
\Bigr] \ud t
+ O(1) \, \ud t
\\
&\hspace{15pt} + \varepsilon \frac{\varsigma^{-1}(P_{t}^\circ) + \varsigma^{-1}(Q_{t}^\circ)}2 \frac{Z_{t}^{\dagger}}{\vert Z_{t} \vert} \ud \overline{W}_{t}
  \widetilde P_{t},
  \end{split}
\end{equation*}
where the constant dominating $O(1)$ is now allowed to depend on 
$(\| \widetilde b^{\circ}_{i} \|_{\infty})_{i \in \ESdm}$. 
Fix now an index $i \in \ESdm$. Then, on the event $P_{t}^i \leq Q_{t}^i$, we have, for $t <  \varpi \wedge T$, 
\begin{equation*}
\frac{Z_{t}^i}{\vert Z_{t} \vert} \bigl\vert \varsigma^{-2} \bigl( P_{t}^\circ \bigr) - \varsigma^{-2} \bigl( Q_{t}^\circ \bigr)
\bigr\vert \leq C 
\frac{\vert Z_{t}^i \vert}{\vert Z_{t} \vert}
\vert  P_{t}^\circ -  Q_{t}^\circ \vert
\leq 
C \vert Z_{t}^i \vert
 \leq C \bigl( \widetilde P_{t}^i + \widetilde Q_{t}^i \bigr) \leq 2 C \widetilde Q_{t}^i,
\end{equation*}
for a universal constant $C$. Proceeding similarly whenever $Q_{t}^i \leq P_{t}^i$, we can find a collection of non-negative bounded processes $((\zeta_{t}^{i})_{0 \leq t \leq T})_{i \in \ESdm}$ (bounded by $4$ since $\varsigma^{-2}(P_{t}^\circ)$
and 
$\varsigma^{-2}(Q_{t}^\circ)$ are bounded by $4$ for $t < \varpi \wedge T$) such that 
\begin{equation*}
\begin{split}
\ud \vert Z_{t} \vert &=
\sum_{i=1}^{d-m} \frac{Z_{t}^{i}}{\vert Z_{t} \vert} \zeta_{t}^{i} \Bigl[ 
\frac{ \varphi(\varsigma^2(P_{t}^\circ) P_{t}^i) +\widetilde b_{i}(t,P_{t}^\circ,P_{t})  - \varepsilon^2/4}{2 \widetilde P_{t}^i} 
- 
\frac{  \varphi(\varsigma^2(Q_{t}^\circ) Q_{t}^i) + \widetilde b_{i}(t,Q_{t}^\circ,
 Q_{t}) - \varepsilon^2/4}{2 \widetilde Q_{t}^i} 
\Bigr] \ud t
\\
&\hspace{15pt} 
+ O(1) \, \ud t
+ \varepsilon \frac{\varsigma^{-1}(P_{t}^\circ) + \varsigma^{-1}(Q_{t}^\circ)}2 \frac{Z_{t}^{\dagger}}{\vert Z_{t} \vert} \ud \overline{W}_{t}
  \widetilde P_{t}, \quad t < \varpi \wedge T.
  \end{split}
\end{equation*}
Notice that, on the event $\{\max(P_{t}^i,Q_{t}^i) > \delta\}$, we have 
$\min(P_{t}^i,Q_{t}^i) > \delta/2$, for $t <\varpi \wedge   T$, since 
$\vert P_{t}^i - Q_{t}^i \vert \leq 2 \vert \widetilde P_{t}^i - \widetilde Q_{t}^i \vert \leq 
\delta/2$. 
Allowing $O(1)$ to depend on $\delta$, $\kappa$ $(\| \widetilde b_{i} \|_{\infty})_{i \in \ESdm}$ and $(\| \widetilde b_{i}^{\circ} \|_{\infty})_{i \in \ESdm}$, we get 
\begin{equation*}
\begin{split}
\ud \vert Z_{t} \vert &\leq 
\sum_{i=1}^{d-m} \frac{Z_{t}^{i}}{\vert Z_{t} \vert} \zeta_{t}^{i} \Bigl[ 
\frac{\varphi(\varsigma^2(P_{t}^\circ) P_{t}^i) + \widetilde b_{i}(t,P_{t}^\circ, P_{t}) - \varepsilon^2/4}{2 \widetilde P_{t}^i} 
\\
&\hspace{45pt}- 
\frac{\varphi(\varsigma^2(Q_{t}^\circ) Q_{t}^i) + \widetilde b_{i}(t,Q_{t}^\circ, Q_{t}) - \varepsilon^2/4}{2 \widetilde Q_{t}^i} 
\Bigr] 
{\mathbf 1}_{\{\max(P_{t}^i,Q_{t}^i) \leq \delta\}}\ud t
\\
&\hspace{15pt} 
+ O(1) \, \ud t
+\varepsilon \frac{\varsigma^{-1}(P_{t}^\circ) +  \varsigma^{-1}(Q_{t}^\circ)}2 \frac{Z_{t}^{\dagger}}{\vert Z_{t} \vert} \ud \overline{W}_{t}
  \widetilde P_{t}, \quad t < \varpi \wedge T.
  \end{split}
\end{equation*}
Now, it remains to see that 
\begin{equation*}
\begin{split}
&\frac{\varphi(\varsigma^2(P_{t}^\circ) P_{t}^i) + \widetilde b_{i}(t,P_{t}^\circ,  P_{t}) - \varepsilon^2/4}{2 \widetilde P_{t}^i} 
- 
\frac{\varphi(\varsigma^2(Q_{t}^\circ) Q_{t}^i) + \widetilde b_{i}(t,Q_{t}^\circ, Q_{t}) - \varepsilon^2/4}{2 \widetilde Q_{t}^i} 
 \\
 &= -
 \frac{  \varphi(\varsigma^2(Q_{t}^\circ ) Q_{t}^i) + \widetilde b_{i}(t,Q_{t}^\circ ,  Q_{t}) - \varepsilon^2/4}{2} \Bigl( \frac{1}{\widetilde Q_{t}^i}
 - \frac{1}{\widetilde P_{t}^i} \Bigr)
 \\
&\hspace{15pt} + \frac{ \varphi(\varsigma^2(P_{t}^\circ) P_{t}^i) +  \widetilde  b_{i}(t,P_{t}^\circ,\widetilde P_{t}) - [
  \varphi(\varsigma^2(Q_{t}^\circ) Q_{t}^i) + \widetilde b_{i}(t,Q_{t}^\circ, Q_{t})]}{2 \widetilde P_{t}^i}
 \\
 &= -
 \frac{  \varphi(\varsigma^2(Q_{t}^\circ) Q_{t}^i) + \widetilde b_{i}(t,Q_{t}^\circ,Q_{t}) - \varepsilon^2/4}{2} \frac{Z_{t}^i}{\widetilde P_{t}^i \widetilde Q_{t}^i} 
 \\
&\hspace{15pt} + \frac{ \varphi(\varsigma^2(P_{t}^\circ) P_{t}^i) +  \widetilde  b_{i}(t,P_{t}^\circ,  P_{t}) - [ 
\varphi(\varsigma^2(Q_{t}^\circ) Q_{t}^i) +\widetilde b_{i}(t,Q_{t}^\circ, Q_{t})]}{2 \widetilde P_{t}^i}.
%
%
  \end{split}
\end{equation*}
Recalling that 
$\widetilde b_{i}$, as given by Proposition \ref{prop:representation} for each $i \in \{1,\cdots,d-m\}$, has non-negative values, that 
$\kappa \geq \kappa_{0} \geq 2$, see
Proposition 
\ref{prop:coupling:2}, and that 
$\varepsilon \in (0,1)$, 
we deduce that, on the event $\{\max(P_{t}^i,Q_{t}^i) \leq \delta\}$,  
$ \varphi(\varsigma^2(Q_{t}^\circ) Q_{t}^i) +  \widetilde  b_{i}(t,Q_{t}^\circ,\widetilde Q_{t}) - \varepsilon^2/4 \geq 0$. Therefore,
whenever $P_{t}^i \geq Q_{t}^i$ and $P_{t}^i \leq \delta$,
\begin{equation*}
\begin{split}
&\frac{ \varphi(\varsigma^2(P_{t}^\circ) P_{t}^i) + \widetilde b_{i}(t,P_{t}^\circ, P_{t}) - \varepsilon^2/4}{2 \widetilde P_{t}^i} 
- 
\frac{  \varphi(\varsigma^2(Q_{t}^\circ) Q_{t}^i) + \widetilde b_{i}(t,Q_{t}^\circ,  Q_{t}) - \varepsilon^2/4}{2 \widetilde Q_{t}^i} 
 \\
 &\leq \frac{  \varphi(\varsigma^2(P_{t}^\circ) P_{t}^i) +\widetilde b_{i}(t,P_{t}^\circ, P_{t}) - [ 
\varphi(\varsigma^2(Q_{t}^\circ) Q_{t}^i) + \widetilde b_{i}(t,Q_{t}^\circ, Q_{t})]}{2 \widetilde P_{t}^i}
 \\
 &= \frac{\kappa + b_{i}(t,P_{t}^\circ,\widetilde P_{t}) - [\kappa
 + \widetilde b_{i}(t,Q_{t}^\circ, Q_{t})]}{2 \widetilde P_{t}^i}
\leq 
\frac{\| \widetilde b_{i} \|_{\infty}} {\widetilde P_{t}^i}
=
\frac{\| \widetilde b_{i} \|_{\infty}} {\max(\widetilde P_{t}^i,\widetilde Q_{t}^i)}.
\end{split}
\end{equation*}
Proceeding similarly when $Q_{t}^i \geq P_{t}^i$ and $Q_{t}^i \leq \delta$
and then 
 letting
\begin{equation*}
\begin{split}
\widetilde \beta_{t}^i &:= \zeta_{t}^i 
\| \widetilde b_{i} \|_{\infty}
{\mathbf 1}_{\{\max(P_{t}^i,Q_{t}^i) \leq \delta\}}, 
\end{split}
\end{equation*}
we get 
\begin{equation*}
\begin{split}
\ud \vert Z_{t} \vert &\leq 
\sum_{i=1}^{d-m} \frac{Z_{t}^{i}}{\vert Z_{t} \vert} \frac{ \widetilde \beta_{t}^{i}}{\max(\widetilde P_{t}^i,\widetilde Q_{t}^i)} \ud t
+ O(1) \, \ud t
+ \varepsilon \frac{\varsigma^{-1}(P_{t}^\circ) + \varsigma^{-1}(Q_{t}^\circ)}2 \frac{Z_{t}^{\dagger}}{\vert Z_{t} \vert} \ud \overline{W}_{t}
  \widetilde P_{t}, \quad t < \varpi\wedge   T,
  \end{split}
\end{equation*}
which completes the proof. 
\end{proof}

\begin{lemma}
\label{lem:coupled:sde:2}
Take $(\widetilde{P}_{t})_{0 \le t \le T}$ and $(\widetilde{Q}_{t})_{0 \le t \le T}$ two continuous 
${\mathbb F}^{\boldsymbol W^\circ,\boldsymbol W}$-
adapted processes with values in the intersection of the orthant $(\RR_{+})^d$ and of the sphere of dimension $d$. Then, 
letting
\begin{equation*}
R_{t} := I_{d} - 2 \frac{(\widetilde{P}_{t} - \widetilde{Q}_{t})(\widetilde P_{t} - \widetilde Q_{t})^\dagger}{\vert 
\widetilde P_{t} - \widetilde Q_{t} \vert^2} {\mathbf 1}_{t < \tau}, \quad t \in [0,T], 
\end{equation*}
with $\tau = \inf \{t \geq t_{0} : \widetilde P_{t} = \widetilde Q_{t} \}$, 
$I_{d}$ standing for the identity matrix of dimension $d$,
the process 
 $( \int_{0}^t R_{s} \ud \overline{W}_{s} R_{s})_{0 \le t \le T}$, 
 with the convention $\overline W_{t}^{i,i}:=0$ for $t \in [0,T]$ and $i \in \ES$, 
is an antisymmetric Brownian motion of dimension $d(d-1)/2$ independent of 
${\boldsymbol W}^\circ$.
\end{lemma}

\begin{proof}
We first extend the family 
${\boldsymbol W}$ into a new family 
$\widetilde{\boldsymbol W}$ of independent Brownian motions, by letting 
$\widetilde{\boldsymbol W}^{i,j} = 
{\boldsymbol W}^{i,j}$ for $i,j \in \ES$ with $i \not = j$ and by assuming that 
the family
$(\widetilde{\boldsymbol W}^{i,i})_{i \in \ES}$ is a collection of Brownian motions 
that is independent of  ${\boldsymbol W}$. We then observe that, for any 
$i,j \in \ES$ with $i \not =j$, 
\begin{equation*}
\begin{split}
\biggl( \int_{0}^t R_{s} \ud \overline{W}_{s} R_{s} \biggr)_{i,j}
= 
\frac1{\sqrt 2} 
\biggl[ \biggl( \int_{0}^t R_{s} \ud \widetilde{W}_{s} R_{s} \biggr)_{i,j}
- 
\biggl( \int_{0}^t R_{s} \ud \widetilde{W}_{s} R_{s} \biggr)_{j,i} \biggr], \quad t \in [0,T].
\end{split}
\end{equation*}
In order to complete the proof, it suffices to show that 
the family $(( \int_{0}^t R_{s} \ud \overline{W}_{s} R_{s})_{0 \le t \le T})_{i,j \in \ES : i \not =j}$ forms a collection of independent Brownian motions 
that is independent of ${\boldsymbol W}^\circ$. 
Independence between
$(( \int_{0}^t R_{s} \ud \widetilde{W}_{s} R_{s})_{0 \le t \le T})_{i,j \in \ES }$
and  
${\boldsymbol W}^\circ$ is obvious. 
It thus remains to compute the brackets of the family
$(( \int_{0}^t R_{s} \ud \widetilde{W}_{s} R_{s})_{0 \le t \le T})_{i,j \in \ES }$
to conclude. Using the fact that $R^{\dagger}_{s} = R_{s}$
and $R_{s} R_{s} = I_{d}$, 
we have
\begin{equation*}
\begin{split}
\sum_{k,l \in \ES} R^{i,k}_{s} 
d \widetilde{W}_{s}^{k,l} R^{l,j}_{s}
\cdot 
\sum_{k',l' \in \ES}
R^{i',k'}_{s}
d \widetilde{W}_{s}^{k',l'} R^{l',j'}_{s}
&=
\sum_{k,l \in \ES} \sum_{k',l' \in \ES} \Bigl( R^{i,k}_{s} 
R^{i',k'}_{s}
R^{l,j}_{s}
R^{l',j'}_{s}
 \delta_{k,k'} \delta_{l,l'}
 \Bigr)
\\
&= \sum_{k,l \in \ES} \Bigl( R^{i,k}_{s} 
R^{i',k}_{s}
R^{l,j}_{s}
R^{l,j'}_{s}
\Bigr) 
= \delta_{i,i'} \delta_{j,j'},
\end{split}
\end{equation*}
which completes the proof. 
\end{proof}

\begin{lemma}
\label{lem:coupled:sde:1}
Under the assumption and notations of
Propositions 
\ref{prop:coupling:2}
and
\ref{prop:coupling:1}, 
Equation 
\eqref{eq:coupled:sde} is uniquely solvable (in the strong sense). 
\end{lemma}

\begin{proof}
We first observe that, at any time $t \in [0,T]$, the 
coefficients  $\tilde{B}(t,\cdot)$ and 
$B^\circ(t,\cdot)$
are a priori defined as functions of the space variable $(r^\circ,r) \in \hat{\mathcal S}_{m} \times {\mathcal S}_{d-m-1}$. 
By projecting $\RR^m$ onto $\hat{\mathcal S}_{m}$
(which is convex)
 and then $\RR^{d-m}$ onto ${\mathcal S}_{d-m-1}$ (which is also convex), we may easily extend them to the entire ${\mathbb R}^d$. 
We then observe from Proposition 
\ref{prop:representation}
that the full-fledged drift coefficient in the system 
\eqref{eq:coupled:sde}
remains Lipschitz continuous in the four entries $(P^\circ_{t},P_{t},Q^\circ_{t},Q_{t})$ as long the coordinates of the latter remain away from zero. 
Similarly, the diffusion coefficient remains Lipschitz continuous in the same four entries 
as long as the coordinates of the latter remain strictly positive and the distance between $P_{t}$ and $Q_{t}$ remains also 
strictly positive. 

Therefore, we deduce that, for any small $a>0$, the system \eqref{eq:coupled:sde} is uniquely solvable up to the first time $\uptau^a$ when one of the coordinates of the vector $(P^\circ_{\uptau^a},P_{\uptau^a},Q^\circ_{\uptau^a},Q_{\uptau^a})$ is less than $a$ or the distance 
between $P_{\uptau^a}$ and $Q_{\uptau^a}$ becomes less than $a$. 
Letting $a$ tend to $0$, we deduce that 
\eqref{eq:coupled:sde} is uniquely solvable up to $\uptau=\lim_{a \searrow 0} \uptau^a \wedge T$.

By Lemma \ref{lem:coupled:sde:2}, we know that, up to time $\uptau$, we may see 
$(P^\circ_{t},P_{t})_{t_{0} \le t < \uptau}$
and 
$(Q^\circ_{t},Q_{t})_{t_{0} \le t < \uptau}$
as solutions of an SDE of the same type 
as 
\eqref{eq:new:system:prop:5:2}. Hence, by identity in law 
\eqref{eq:identity:law} and by Proposition 
\ref{prop:ito:formula} (or equivalently by Proposition 
\ref{thm:approximation:diffusion:2}, recalling that $\kappa_{0} \geq 2$), we deduce that,
both processes take values in $\hat{\mathcal S}_{m} \times {\mathcal S}_{d-m}$ and that 
\begin{equation*}
\begin{split}
&{\mathbb P} \Bigl( 
\inf_{i \in \ESm} \inf_{t_{0} \leq t< \uptau}
P_{t}^{\circ,i} >0, \
\inf_{i \in \ESm} \inf_{t_{0} \leq t< \uptau}
Q_{t}^{\circ,i} >0, 
  \inf_{i \in \ESdm} \inf_{t_{0}\leq t< \uptau}
P_{t}^{i} >0, \
\inf_{i \in \ESdm} \inf_{t_{0} \leq t< \uptau}
Q_{t}^{i} >0
\Bigr)  = 1.
\end{split}
\end{equation*}
This shows in particular that the drift in 
\eqref{eq:coupled:sde} remains bounded up to time $\uptau$ and that it makes sense to extend (by continuity) the process
$(P,P^\circ,Q,Q^\circ)$ to the closed interval $[0,\uptau]$. 
Moreover, we must have
\begin{equation*}
{\mathbb P} \bigl( \uptau = \tau \bigr)=1,
\end{equation*}
where we recall that $\tau$ denotes the first time when the two processes 
$(P_{t})_{0 \leq t \leq \uptau}$ and 
$(Q_{t})_{0 \leq t \leq  \uptau}$ meet. 
This proves the unique strong solvability on $[0,\tau]$.
Unique solvability from $\tau$ to $T$ is addressed in a similar manner noting that the diffusion coefficient then becomes simpler, see   
\eqref{eq:def:Rt}.
\end{proof}

\subsubsection{Proof of Proposition \ref{prop:coupling:2}}
We recall that  
$\vert p-q \vert < \delta^2/ (64\sqrt{d})$. 
\vspace{5pt}
\\
\textit{First Step.}
The proof mostly relies on a Girsanov argument. Using the same notations as in the statement and in the proof of Proposition 
\ref{prop:coupling:1}, we let
(see
\eqref{eq:tildeP:coupling}
and
\eqref{eq:coupled:sde})
\begin{equation}
\label{eq:W:prime}
W_{t}^{\prime,i,j} := 
W^{i,j}_{t} + \frac1{\varepsilon} \int_{ t_{0}}^t \Psi_{s}^{i,j} \ud s,
\quad 
\Psi_{t}^{i,j} := \frac{2 \sqrt 2 \widetilde P_{t}^j}{
(\varsigma^{-1}(P_{t}^\circ) + \varsigma^{-1}(Q_{t}^\circ))
} \frac{\widetilde \beta_{t}^i}{\max(\widetilde P_{t}^i,\widetilde Q_{t}^i)} {\mathbf 1}_{\{ t  < \varpi\}}, 
\end{equation}
for 
$t \in [t_{0},T]$
and
$i,j \in \ESdm$ with $i \not = j$, and 
$W_{t}^{\prime,i,i} =0$ for $i \in \ESdm$. Then, for all $t \in [t_{0},\varpi \wedge T)$, 
\begin{equation*}
\varepsilon 
\frac{Z_{t}^\dagger}{\vert Z_{t} \vert} 
\ud \overline W_{t}  \widetilde P_{t} 
=
{\varepsilon}
\frac{Z_{t}^\dagger}{\vert Z_{t} \vert} 
\ud \overline W_{t}'  \widetilde P_{t}
-  
\frac{Z_{t}^\dagger}{\sqrt{2} \vert Z_{t} \vert} 
\bigl( \Psi_{t} - \Psi_{t}^\dagger \bigr) 
\widetilde P_{t} \ud t,
\end{equation*}
with
\begin{equation*}
\begin{split}
&
\frac{Z_{t}^\dagger}{\sqrt{2} \vert Z_{t} \vert} 
\bigl( \Psi_{t} - \Psi_{t}^\dagger \bigr) 
\widetilde P_{t} 
\\
&= 
\frac2{ \varsigma^{-1}(P_{t}^\circ) + \varsigma^{-1}(Q_{t}^\circ)}
\sum_{i,j \in \ESdm } 
\frac{Z_{t}^i }{ \vert Z_{t}\vert}
\Bigl( 
\frac{\widetilde \beta_{t}^i \widetilde P_{t}^j 
}{\max(\widetilde P_{t}^i,\widetilde Q_{t}^i)}
-    \frac{\widetilde \beta_{t}^j \widetilde P_{t}^i}{\max(\widetilde P_{t}^j,\widetilde Q_{t}^j)} 
\Bigr) \widetilde P_{t}^j 
\\
&=
\frac2{  \varsigma^{-1}(P_{t}^\circ) + \varsigma^{-1}(Q_{t}^\circ) }
\biggl( \sum_{i \in \ESdm} \frac{Z_{t}^i}{\vert Z_{t}\vert}  \frac{\widetilde \beta_{t}^i}{\max(\widetilde P_{t}^i,\widetilde Q_{t}^i)} - 
\Bigl\langle \frac{Z_{t}}{\vert Z_{t} \vert},\widetilde P_{t} \Bigr\rangle \sum_{j \in \ESdm} 
\frac{\widetilde \beta_{t}^j \widetilde P_{t}^j}{\max(\widetilde P_{t}^j,\widetilde Q_{t}^j)}
 \biggr),
\end{split}
 \end{equation*}
where we used the identity $\sum_{j \in \ESdm} \bigl(\widetilde P_{t}^j \bigr)^2 =1$. 
Plugging the above identity into 
\eqref{eq:Z:coupling}, we get
\begin{equation}
\label{eq:xi}
\begin{split}
&\ud    \vert Z_{t} \vert
\leq  
C \ud t + \varepsilon
\frac{\varsigma^{-1}(P_{t}^\circ) + \varsigma^{-1}(Q_{t}^\circ)}2
\frac{Z_{t}^{\dagger}}{\vert Z_{t} \vert} \ud \overline  W_{t}' \widetilde P_{t}, \quad t \in [t_{0},  \varpi \wedge T), 
\end{split}
\end{equation}
where $C$ is a constant only depending on $\delta$,  $\kappa$, $(\|b_{i}\|_{\infty})_{i=1,\cdots,d}$, 
$(\|b_i^{\circ}\|_{\infty})_{i=1,\cdots,d}$
and $T$. 
\vspace{5pt}

\textit{Second Step.}  
We now  introduce the probability measure:
\begin{equation}
\label{eq:dQ:dP}
\frac{\ud{\mathbb Q}}{\ud{\mathbb P}}
= \exp \biggl( - \frac1{\varepsilon} \sum_{i,j \in \ESdm : i \not =j} 
\int_{0}^{\varpi \wedge T}\Psi_{t}^{i,j} \ud W_{t}^{i,j}
- \frac1{2\varepsilon^{2}}  
\sum_{i,j \in \ESdm : i \not =j} 
\int_{0}^{\varpi \wedge T} \bigl\vert 
\Psi_{t}^{i,j}
 \bigr\vert^2 
\ud t \biggr).
\end{equation}
Under ${\mathbb Q}$, the processes $((W^{\prime,i,j}_{t})_{t_{0} \leq t \leq T})_{i,j \in \ESdm : i \not = j}$
are independent Brownian motions (the fact that we can apply Girsanov's theorem is fully justified in the third step of the proof). 
By 
\eqref{eq:coupling:covariance}, the bracket of the martingale part in 
\eqref{eq:xi} is given by (up to the leading multiplicative factor)
\begin{equation*}
\frac{1}{\ud t} \Bigl\langle \frac{Z_{t}^{\dagger}}{\vert Z_{t} \vert} \ud \overline{W}'_{t}
  \widetilde P_{t} \Bigr\rangle
  = 1 -
\frac14 \vert Z_{t} \vert^2.
\end{equation*}
In particular, there exists a Brownian motion $(B_{t})_{t_{0} \le t \le T}$
under ${\mathbb Q}$ such that 
\begin{equation}
\label{eq:vert Ztvert}
\begin{split}
&d  \vert Z_{t} \vert 
\leq  
C \ud t + 
\varepsilon
\frac{\varsigma^{-1}(P_{t}^\circ) +  \varsigma^{-1}(Q_{t}^\circ)}2
\sqrt{1-\frac14 \vert Z_{t} \vert^2} dB_{t}, \quad t \in [t_{0},\varpi \wedge   T). 
\end{split}
\end{equation}
Let now
\begin{equation*}
\ud \Theta_{t} = C  \ud t + \varepsilon \theta_{t} \ud B_{t},
\quad 
\theta_{t}: = 
\min \Bigl( \frac{\varsigma^{-1}(P_{t}^\circ) + \varsigma^{-1}(Q_{t}^\circ)}2, c' \Bigr) 
\sqrt{1-\frac14 \min(\vert Z_{t} \vert^2,\frac{\delta^2}{4})},
\end{equation*}
with $\vert \Theta_{t_{0}} \vert = \vert Z_{t_{0}} \vert = \vert \widetilde p - \widetilde q\vert$, with 
$\widetilde{p}=(\sqrt{p_{1}},\cdots,\sqrt{p_{d-m}})$ and similarly for $\widetilde{q}$, and 
where $c'= \min \{ \varsigma^{-1}(p^\circ), \vert p^\circ \vert_{1} \geq 3/4+\delta \sqrt{d}/4\}=
(1/4-\delta \sqrt{d}/4)^{-1/2}$. 
(Note that, for 
$t \in [t_{0},\varpi \wedge   T)$, 
$\vert P_{t}^\circ \vert_{1} \leq 3/4$ and 
$\vert Q_{t}^{\circ} \vert_{1} \leq 3/4 + \vert P_{t}^\circ - Q_{t}^\circ \vert_{1}
 \leq 3/4 + \sqrt{m} \vert P_{t}^\circ - Q_{t}^\circ \vert < 3/4 + \delta \sqrt{d} /4 < 
 3/4+1/16=13/16$.)
Obviously, $\vert Z_{t} \vert \leq \Theta_{t}$ for all $t \in [t_{0}, \varpi \wedge  T]$ (because, 
up to time $\varpi \wedge  T$, 
$\theta_{t}$ coincides with the integrand in the stochastic integral appearing in the right-hand side of 
\eqref{eq:vert Ztvert}). 
Since $(\theta_{t})_{t_{0} \le t \leq T}$
stays in a (universal) deterministic compact subset of $(0,+\infty)$, we deduce from a new application of Girsanov's theorem that there exists a new probability measure ${\mathbb Q}'$ under which 
\begin{equation*}
\ud\Theta_{t} = \varepsilon \theta_{t} \ud B_{t}', \quad t \in [t_{0},T],
\end{equation*}
$(B_{t}')_{t_{0} \leq t \leq T}$ being a Brownian motion under ${\mathbb Q}'$.
By expanding the Girsanov transformation, we can check that 
${\mathbb E}^{\mathbb Q'}[ (\ud {\mathbb Q}/\ud {\mathbb Q'})^2]
= {\mathbb E}^{\mathbb Q}[\ud {\mathbb Q}/\ud {\mathbb Q'}]
 \leq \gamma^2$, that is 
${\mathbb Q}(A) = {\mathbb E}^{\mathbb Q'}[
(\ud {\mathbb Q}/\ud {\mathbb Q'}) {\mathbf 1}_{A}]
 \leq \gamma {\mathbb Q}'(A)^{1/2}$ for any event $A \in {\mathcal F}_{T}^{{\boldsymbol W^\circ},{\boldsymbol W}}$, for a constant $\gamma$ that may depend on 
$\varepsilon$.
 
Clearly, $(B_{t}')_{t_{0} \leq t \leq T}$ can be extended into a Brownian motion (under ${\mathbb Q}'$) on the entire $[t_{0},\infty)$
and, similarly, $(\theta_{t})_{t_{0} \leq t \leq T}$ can be also extended to the entire $[t_{0},\infty)$ by letting 
$\theta_{t}=c'$ for $t >T$. The process $(\Theta_{t})_{t_{0} \leq t \leq T}$ can be extended accordingly to the entire $[t_{0},\infty)$. Representing $(\Theta_{t})_{t \geq t_{0}}$ in the form a time-changed Brownian motion, there exists a new Brownian 
motion $(\hat B_{t})_{t \geq 0}$ under ${\mathbb Q}'$ (with respect to a time-changed filtration) such that
\begin{equation*}
\Theta_{t} =
\vert \widetilde p - \widetilde q \vert
+
 \varepsilon \hat{B}_{I_{t}}, 
\quad 
I_{t}:=
\int_{t_{0}}^t \theta_{s}^2 \ud s, \quad t \geq t_{0}. 
\end{equation*}
Obviously, there exists a universal constant $\Gamma \geq 1$ such that, with probability 1 under ${\mathbb Q}'$, 
\begin{equation*}
\Gamma^{-1} (t-t_{0}) \leq I_{t} \leq \Gamma (t-t_{0}), 
\quad t \geq t_{0}. 
\end{equation*}
We now call 
\begin{equation*}
\sigma(\Theta) := \inf\bigl\{ s \geq t_{0} : \vert \Theta_{s} \vert \geq \frac{\delta}{4} \bigr\},
\quad
\tau(\Theta) := \inf\{ s \geq t_{0} : \Theta_{s} =0\}.
\end{equation*}
Then,
recalling that 
$\vert p-q \vert < \delta^2/ (64\sqrt{d})$
and
observing that 
\begin{equation}
\label{eq:z:p-q}
\begin{split}
\vert \widetilde p - \widetilde q \vert^2 &=  \sum_{i \in \ESdm} \bigl\vert \sqrt{p_{i}} - \sqrt{q_{i}} \bigr\vert^2
\leq \sum_{i \in \ESdm} \vert p_{i} - q_{i} \vert
 \leq \sqrt{d} \vert p - q \vert.   
\end{split} 
\end{equation}
we get 
$\vert \widetilde p - \widetilde q \vert  < \delta/8$ and then,  for $t \in [t_{0},T]$, 
\begin{equation*}
\begin{split}
{\mathbb Q}' \Bigl( \tau(\Theta) < \sigma(\Theta), \tau(\Theta) \leq t \Bigr)
&\geq  {\mathbb Q}' \biggl( 
\varepsilon \inf_{0 \leq s \leq  I_{t}} \hat B_{s}
\leq - \vert \widetilde p - \widetilde q \vert, 
\quad
\varepsilon \sup_{0 \leq s \leq I_{t}} \hat B_{s}
\leq  \frac{\delta}{8} 
\biggr) 
\\
&\geq 1- 
{\mathbb Q}' \Bigl( \varepsilon
\sup_{{0} \leq s \leq I_{t}} \hat B_{s}
\leq  \vert \widetilde p - \widetilde q \vert \Bigr)  
-
{\mathbb Q}' \Bigl( \varepsilon
\sup_{{0} \leq s \leq  I_{t}} \hat B_{s}
\geq  \frac{\delta}{8} 
\Bigr)
\\
&\geq 1- 
{\mathbb Q}'\Bigl( \varepsilon
\sup_{0 \leq s \leq (t-t_{0})/\Gamma} \hat B_{s}
\leq  \vert \widetilde p - \widetilde q \vert \Bigr)  
-
{\mathbb Q}' \Bigl( \varepsilon
\sup_{0 \leq s \leq   \Gamma (t-t_{0})} \hat B_{s}
\geq  \frac{\delta}{8} 
\Bigr).
\end{split}
\end{equation*}
We deduce that 
\begin{equation*}
\begin{split}
{\mathbb Q}' \Bigl( \tau(\Theta) < \sigma(\Theta), \tau(\Theta) \leq t \Bigr)
&\geq 1- C'   \frac{ \vert \widetilde p - \widetilde q \vert}{\varepsilon \sqrt{t-t_{0}}}  
 - C'  \exp \bigl( - \frac{1}{C' \varepsilon^2 (t-t_{0})} \bigr),
\end{split}
\end{equation*}
for a new constant $C'$ that is independent of $\varepsilon$. Therefore, 
by 
\eqref{eq:z:p-q}, 
up to a new value of $C'$, 
\begin{equation*}
\begin{split}
{\mathbb Q}' \Bigl( \tau(\Theta) < \sigma(\Theta), \tau(\Theta) \leq t \Bigr)
&\geq 1- C'  \frac{ \sqrt{\vert   p -  q \vert}}{\varepsilon \sqrt{t-t_{0}}}  
 - C'  \exp \bigl( - \frac{1}{C' \varepsilon^2 (t-t_{0})} \bigr).
\end{split}
\end{equation*}
In particular, choosing $t-t_{0}= \vert p-q \vert^{1/3}/2$, which is possible since 
$\vert p-q \vert^{1/3} \leq T-t_{0}$, we deduce that 
(with $S:= t_{0} + \vert p-q \vert^{1/3}$)
\begin{equation*}
{\mathbb Q}' \Bigl( \tau(\Theta) < \sigma(\Theta), \tau(\Theta) < S \Bigr)\geq 1 - \frac{C'}{\varepsilon} \vert p-q \vert^{1/3}, 
\end{equation*}
or, equivalently, 
\begin{equation*}
{\mathbb Q}' \biggl( \Bigl\{ \tau(\Theta) < \sigma(\Theta), \tau(\Theta) \leq S \Bigr\}^{\complement}\biggr) \leq  C \vert p-q \vert^{1/3},
\end{equation*}
where we recall that $C$ is allowed to depend on $\varepsilon$. 
Then, returning to ${\mathbb Q}$, we get, for a new value of $C$,
\begin{equation*}
{\mathbb Q} \bigl( \tau(\Theta) \geq  \sigma(\Theta) \wedge S \bigr)
=
{\mathbb Q} \biggl( \Bigl\{ \tau(\Theta) < \sigma(\Theta), \tau(\Theta) \leq S \Bigr\}^{\complement}\biggr) \leq  C \vert p-q \vert^{1/6}.
\end{equation*}
We then notice from the inequality $\vert Z_{t} \vert \leq \Theta_{t}$, for $t \in [t_{0},\varpi \wedge T]$, that
$t \leq \varpi \wedge T$ implies $t \leq  \tau(\Theta)$ and 
that 
$\sigma \leq \varpi \wedge T$ implies
$\sigma(\Theta) \leq \sigma$. Therefore,
on the event $\{\varpi \geq S\}$, 
we have 
$S \leq \tau(\Theta)$. 
Moreover, 
on the event $\{\tau \wedge \sigma \leq \varpi  \wedge T\}$, 
$$ \bigl\{ \sigma(\Theta) \leq \tau(\Theta) \bigr\} \supset \bigl\{ \sigma \leq \tau \bigr\} = \bigl\{ \tau < \sigma \bigr\}^{\complement}.$$ 
Hence, 
\begin{equation*}
\begin{split}
{\mathbb Q} \bigl( 
 \varpi_{S} < \tau \wedge \varrho \wedge \rho \bigr)  
&\leq {\mathbb Q} \bigl( \bigl\{ S \leq \varpi \} \cup \{\sigma \leq \tau, \ \tau \wedge \sigma 
\leq \varpi \wedge T\} \bigr)
\\
&\leq 
{\mathbb Q} \bigl( \bigl\{ S \leq \tau(\Theta) \} \cup \{\sigma(\Theta) \leq \tau(\Theta) \} \bigr)
={\mathbb Q} \bigl( \tau(\Theta) \geq  \sigma(\Theta) \wedge S \bigr)
\leq  C \vert p-q \vert^{1/6}.
\end{split}
\end{equation*}
  \vskip 5pt
 
 \textit{Third Step.}
 In order to complete the proof, it remains to prove that, for a new value of the constant $C$, for any event $A \in {\mathcal F}_{T}^{{\boldsymbol W^\circ},{\boldsymbol W}}$,
 ${\mathbb P}(A) \leq C {\mathbb Q}(A)^{1/2}$, 
  provided that $\kappa$ is chosen large enough. 
 As for the comparison of ${\mathbb Q}$ and ${\mathbb Q}'$ in the previous step, it suffices to 
 prove that 
$ {\mathbb E} [  \ud {\mathbb P}/\ud {\mathbb Q}  ]^{1/2}
\leq C$ (since
${\mathbb P}(A) 
= {\mathbb E}^{\mathbb Q}[ (
\ud {\mathbb P}/\ud {\mathbb Q}){\mathbf 1}_{A}]
\leq {\mathbb E} [  \ud {\mathbb P}/\ud {\mathbb Q}  ]^{1/2}
{\mathbb Q}(A)^{1/2}
$). 
%
%
%
Here (compare with 
\eqref{eq:dQ:dP}),
\begin{equation*}
\begin{split}
\frac{\ud {\mathbb P}}{\ud {\mathbb Q}}
&= \exp \biggl( \frac1{\varepsilon} \sum_{i,j \in \ESdm : i \not = j} 
\int_{t_{0}}^{\varpi \wedge T}
\Psi^{i,j}_{t}
 \ud W_{t}^{i,j}
+ \frac1{2 \varepsilon^2} 
\sum_{i,j \in \ESdm : i \not = j} 
\int_{t_{0}}^{\varpi \wedge T} \bigl\vert 
\Psi_{t}^{i,j}
 \bigr\vert^2 
\ud t \biggr).
\end{split}
\end{equation*}
Letting 
$$\Bigl(M_{t} := 
\frac1{\varepsilon} \sum_{i,j \in \ESdm : i \not = j} 
\int_{t_{0}}^t \Psi_{s}^{i,j} \ud W_{s}^{i,j}\Bigr)_{t_{0} \leq t \leq T},$$ this may be rewritten as 
\begin{equation*}
\frac{\ud {\mathbb P}}{\ud {\mathbb Q}}
= \exp \Bigl( M_{\varpi \wedge T} +  \frac12 \langle M \rangle_{\varpi \wedge T} \Bigr),
\end{equation*}
and then 
\begin{equation*}
\begin{split}
{\mathbb E} \Bigl[ \frac{\ud{\mathbb P}}{\ud {\mathbb Q}} 
 \Bigr]
&= 
{\mathbb E} \Bigl[ 
\exp \Bigl(   M_{\varpi \wedge T} +  \frac{1}2 \langle M \rangle_{\varpi \wedge T} \Bigr)
 \Bigr]
 \\
 &\leq 
 {\mathbb E} \Bigl[ 
\exp \Bigl(  M_{\varpi \wedge T} -    \langle M \rangle_{\varpi \wedge T} 
+  (1 + \frac{1}{2}) \langle M \rangle_{\varpi \wedge T} \Bigr) 
\Bigr]
 \leq
 {\mathbb E}  \Bigl[ 
\exp \Bigl( 3 \langle M \rangle_{\varpi \wedge T} \Bigr) 
\Bigr]^{1/2},
 \end{split}
\end{equation*}
where to get the last line, we used the fact that 
$ {\mathbb E}[ 
\exp( 2M_{\varpi \wedge T} -  2  \langle M \rangle_{\varpi \wedge T} )] \leq 1$. 
Returning to the definition of 
$\Psi$ in 
\eqref{eq:W:prime}
and recalling that $\varsigma^{-1}$ is lower bounded by $1$, 
the point is to prove that 
\begin{equation*}
 {\mathbb E}  \Bigl[ 
\exp \Bigl( \frac3{\varepsilon^2}
\sum_{i,j \in \ESdm : i \not = j} 
\int_{t_{0}}^{\varpi \wedge T} \bigl\vert 
\Psi_{t}^{i,j}
 \bigr\vert^2 
 \ud t \Bigr) \Bigr]
 \leq
 {\mathbb E}  \Bigl[ 
\exp \Bigl( \frac{6}{\varepsilon^2}
\sum_{i \in \ESdm} 
\int_{t_{0}}^{\varpi \wedge T} 
 \Bigl\vert \frac{\widetilde \beta_{t}^i}{\max(\widetilde P_{t}^i,\widetilde Q_{t}^i)}
\Bigr\vert^2 \ud t 
\Bigr) \Bigr] 
\end{equation*}
is finite, provided that $\kappa$ is chosen large enough and then to find a tractable bound. 
The proof is similar to that of Proposition 
\ref{expfin}, but we feel better to expand it as it plays a key role in the determination of the constant 
$\kappa$. Recalling 
the bound for 
$(\beta^i_{t})_{0 \le t \le T}$ in 
the statement of Proposition 
\ref{prop:coupling:1}, using the fact that $(\widetilde P_{t}^i)^2=P_{t}^i$ and invoking H\"older's inequality, it suffices to upper bound
\begin{equation}
\label{eq:Girsanov:needed:bounded}
\sup_{i \in \ESdm} {\mathbb E}  \Bigl[ 
\exp \Bigl( 
\frac{6d}{\varepsilon^2}
\int_{t_0}^{\varpi \wedge T}  \frac{(4 \| \widetilde b_{i} \|_{{\infty}})^2}{\max(P_{t}^i,Q_{t}^i)} 
 \ud t \Bigr) \Bigr]^{1/d}.
\end{equation}
Here, we recall from Proposition 
\ref{prop:representation}
that the coefficients 
$(\widetilde b_{j})_{j \in \ESdm}$
are bounded by a constant that only depends on 
$(\| b_{j} \|_{\infty})_{j \in \ES}$.
Moreover, 
we recall from 
\eqref{eq:new:system:prop:5:2} that 
\begin{equation*}
\begin{split}
\ud P_{t}^i &= \varsigma^{-2}(P_{t}^\circ) \Bigl(  \varphi\bigl(\varsigma^2(P_{t}^\circ) P_{t}^i \bigr) + \widetilde b_{i}(t,P_{t}^\circ,P_{t})  + P_{t}^i \widetilde b_i^{\circ}(t,P_{t}^\circ,P_{t}) \Bigr) \ud t 
+ \varepsilon \varsigma^{-1}(P_{t}^\circ) \sum_{j=1}^{d-m} \sqrt{P_{t}^i P_{t}^j} \ud \overline W_{t}^{i,j}, 
\end{split}
\end{equation*}
for $i \in \ESdm$ and 
$t \in [t_{0},T]$.
Using  again the fact that 
$(1/\ud t)
\langle \sum_{j \in \ESdm} \sqrt{P_{t}^i P_{t}^j} \ud \overline W_{t}^{i,j}
\rangle = P_{t}^i \bigl( 1 - P_{t}^i \bigr)$, 
we get, 
by It\^o's formula,
\begin{equation}
\label{ito:ln}
\begin{split}
\ud \bigl[ \ln P_{t}^i \bigr] &= \varsigma^{-2}(P_{t}^\circ) \Bigl( \frac{  \varphi( \varsigma^{2}(P_{t}^\circ) P_{t}^i) + \widetilde b_{i}(t,P_{t}^\circ,P_{t})}{P_{t}^i}  + \widetilde b_i^{\circ}(t,P_{t}^\circ,P_{t}) \Bigr) \ud t 
\\
&\hspace{15pt}
+ \varepsilon \varsigma^{-1}(P_{t}^\circ) \frac1{\sqrt{P_{t}^i}} \sum_{j \in \ESdm}  \sqrt{P_{t}^j} \ud \overline W_{t}^{i,j}
 - \frac{\varepsilon^2}2 \varsigma^{-2}(P_{t}^\circ) \frac{1- P_{t}^i}{P_{t}^i} \ud t,
\quad t \in [t_{0},T].
\end{split}
\end{equation}
Recalling that $\widetilde b_{i}$ takes non-negative values and 
denoting by 
$(O (1))_{0 \leq t \leq T}$ a progressively measurable process that is dominated by a constant $C$ that may depend on 
$\delta$, $\kappa$, $(\|b_{j}\|_{\infty})_{j \in \ES}$, $(\|b_{j}^{\circ}\|_{\infty})_{j \in \ES}$ and $T$,
recalling that $\varphi \equiv \kappa$ on $[0,\delta]$ and $\varsigma^{-2}(P_{t}^\circ) \leq 4$ for $t \leq \varpi$, 
and choosing $\kappa$ as large as needed (in terms of the sole 
$(\|b_{j}\|_{\infty})_{j \in \ES}$), we get, for $t \leq \varpi \wedge T$,  
\begin{equation*}
\begin{split}
&\varsigma^{-2}(P_{t}^\circ) \Bigl( \frac{ \varphi(\varsigma^2(P_{t}^\circ)P_{t}^i) + \widetilde b_{i}(t,P_{t}^\circ,P_{t})}{P_{t}^i}  + \widetilde b^{\circ}_{i}(t,P_{t}^\circ,P_{t}) 
- \frac{\varepsilon^2}2  \frac{1- P_{t}^i}{P_{t}^i}
\Bigr)
\\
&\geq \frac{\varsigma^{-2}(P_{t}^\circ)}{P_{t}^i} \bigl( \kappa -   \tfrac12
\bigr) {\mathbf 1}_{P_{t}^i \leq \delta} - O(1)
 \geq \frac{\kappa}{2} \frac{\varsigma^{-2}(P_{t}^\circ)}{P_{t}^i}    - O(1).   
\end{split}
\end{equation*}
Hence, integrating \eqref{ito:ln}, multiplying by some $\eta>0$ and then taking exponential, 
\begin{equation*}
\label{eq:E:Pt:eta}
\begin{split}
&(P_{\varpi \wedge T}^i)^{\eta} \exp \biggl( - \eta \varepsilon \int_{t_{0}}^{\varpi \wedge T} \varsigma^{-1}(P_{t}^\circ) \frac1{\sqrt{P_{t}^i}}
\sum_{j \in \ESdm} \sqrt{P_{t}^j} \ud \overline W_{t}^{i,j} - \frac{\eta^2 \varepsilon^2}2 
\int_{t_{0}}^{\varpi \wedge T} \varsigma^{-2}(P_{t}^\circ) \frac{1-P_{t}^i}{P_{t}^i} \ud t 
\biggr)
\\
&\geq (P_{0}^i)^{\eta}
\exp \biggl( \bigl(   \frac{\eta \kappa-\eta^2 \varepsilon^2}{2}\bigr) \int_{t_{0}}^{\varpi \wedge T} \frac{\varsigma^{-2}(P_{t}^\circ)}{P_{t}^i} \ud t  -C\biggr),
\end{split}
\end{equation*}
where $C$ is a constant as before. 
For any given $\eta \in (0,1)$, we can choose $\kappa$ as large as needed ($\kappa$ now depending on 
$\varepsilon$, $\eta$ 
and
$(\| b_{j} \|_{{\infty}})_{j \in \ES}$) such that 
\begin{equation}
\label{eq:condition:kappa0:epsilon}
\frac{\eta \kappa-\eta^2 \varepsilon^2}{2} \geq 
\frac{6d}{\varepsilon^2}
\bigl(4 \max_{j \in \ESdm}\| \widetilde b_{j} \|_{{\infty}}\bigr)^2,
\end{equation}
and then (compare with
\eqref{eq:Girsanov:needed:bounded})
\begin{equation*}
 {\mathbb E}  \biggl[ 
\exp \biggl( 
\frac{6 d}{\varepsilon^2}
\int_{t_{0}}^{\varpi \wedge T} \varsigma^{-2}(P_{t}^\circ) \frac{(4  \max_{j \in \ESdm}\| \widetilde b_{j} \|_{{\infty}})^2}{ P_{t}^i}
 \ud t \biggr) \biggr] \leq \frac{C}{p_i^{\eta}},
\end{equation*}
where $C$ is independent of $p_{0}$ but depends on $\delta$, $\varepsilon$, $(\| b_{j} \|_{\infty})_{j \in \ES}$ and 
$(\| b_{j}^\circ \|_{\infty})_{j \in \ES}$ and $T$. 
Since $\varsigma^{-1}$ is above $1$, 
\begin{equation*}
 {\mathbb E}  \biggl[ 
\exp \biggl( 
\frac{6 d}{\varepsilon^2}
\int_{t_{0}}^{\varpi \wedge T} \frac{(4  \max_{j \in \ESdm}\| \widetilde b_{j} \|_{{\infty}})^2}{ P_{t}^i}
 \ud t \biggr) \biggr] \leq \frac{C}{p_i^{\eta}},
\end{equation*}
Similarly,  we have
the same inequality, but replacing $P_{t}^i$ by $Q_{t}^i$ in the left-hand side and 
$p_{i}$ by $q_{i}$ in the right-hand side. Hence,  
we can upper bound 
\eqref{eq:Girsanov:needed:bounded}
by 
%
$C/\max(p_i,q_i)^\eta$. 
\vskip 5pt

\textit{Conclusion.}
We deduce from 
the conclusions of the second and third steps 
 that 
\begin{equation*}
\begin{split}
&{\mathbb P} \bigl( 
 \varpi_{S} < \tau \wedge \varrho \wedge \rho \bigr)  
%
%
 \leq C  \frac{\vert p -q \vert^{1/12}}{\min_{i \in \ESdm}(\max(p_i,q_i))^{d\eta/2}}, 
\end{split}
\end{equation*}
where $C$ depends on $\delta$, $\varepsilon$, $\kappa$, $\eta$, $(\| b_{i} \|_{\infty})_{i=1,\cdots,d}$, $(\|b^{\circ}_{i}\|_{\infty})_{i=1,\cdots,d}$ and $T$. 
Since the value of $\eta$ is arbitrary (provided that it belongs to $(0,1)$), 
we can easily apply the above inequality with $2\eta /d$ instead of $\eta$ (observe that, whenever 
$\eta \in (0,1)$, $2\eta/d$ also belongs to $(0,1)$, since $d \geq 2$).  \qed

\subsection{Proof of Proposition \ref{prop:representation}}
\label{subsubse:proof:proposition 5.2}

\begin{proof}
{\textit{First Step.}}
We introduce some useful notations. 
Having in mind the shape of the coefficients in equation 
\eqref{eq:P}, we let, 
for 
$i \in \ES$ and for
$p \in {\mathcal S}_{d-1}$, 
\begin{equation*}  
{\mathfrak b}_{i}(t,p) :=   \varphi(p_{i}) +  b_{i}(t,p) + p_{i} b_{i}^{\circ}(t,p).
\end{equation*}
Importantly, we recall
from 
\eqref{eq:zero:sum}
 that, for any {$(t,p) \in [0,T] \times {\mathcal S}_{d-1}$}, $\sum_{i \in \ES}  {\mathfrak b}_{i}(t,p)=0$. 
In fact, we can easily extend ${\mathfrak b}_{i}$, for each $i \in \ES$, 
to the entire $[0,T] \times {\mathbb R}^d$ by composing ${\mathfrak b}_{i}$ with the orthogonal projection from 
$\RR^d$ into ${\mathcal S}_{d-1}$. This allows us to define the drift $B^\circ$ entering the dynamics of the second equation in 
\eqref{eq:new:system:prop:5:2}. For a given coordinate $i \in \ESm$, we indeed let
\begin{equation*}
B^{\circ}_{i}(t,r^\circ,r)
:= {\mathfrak b}_{i} \Bigl(t, \bigl( r^\circ,\varsigma^2(r^\circ)r\bigr) \Bigr), \quad t \in [0,T],
\end{equation*}
for $r^\circ \in \RR^m$ and
$r \in \RR^{d-m}$. Notice that 
the definition is especially interesting for our purpose whenever $r^\circ \in \hat{\mathcal S}_{m}$
and $r \in {\mathcal S}_{d-m-1}$, but it is well defined in any case (with the obvious convention that 
$\varsigma^2(r^\circ)=1-(r^\circ_{1}+\cdots+r^\circ_{m})$ even if it is negative). Similarly, for $i \in \ESdm$, we let
\begin{equation*}
B_{i}(t,r^\circ,r)
:=  {\mathfrak b}_{m+i} \Bigl(t, \bigl( r^\circ,\varsigma^2(r^\circ)r\bigr) \Bigr), \quad t \in [0,T].
\end{equation*}
For a new collection of antisymmetric Brownian motions $\overline W^\circ = (\overline W^\circ_{t}=(\overline W^{\circ,i,j}_{t})_{i,j \in \ES : i \not =j})_{0 \leq t \leq T}$
of dimension $d(d-1)/2$ (with the convention  that $\overline{\boldsymbol W}^{\circ,i,i} \equiv 0$ for 
$i \in \ES$), we 
consider the system
\begin{equation}
\label{eq:system:tilde:samelaw}
\begin{split}
&\ud P_{t}^i = \varsigma^{-2}(P_{t}^\circ) \Bigl(
B_{i}(t,P^\circ_{t},P_{t})
- P_{t}^i \sum_{j \in \ESdm}
B_{j} (t,P^\circ_{t},P_{t})
 \Bigr) \ud t 
 \\
&\hspace{35pt}
+ \varepsilon \varsigma^{-1}(P_{t}^\circ) \sum_{j \in \ESdm} \sqrt{P_{t}^i P_{t}^j} \ud \overline W_{t}^{i+m,j+m}, 
 \quad i \in \ESdm,
\\
&\ud \bigl( P_{t}^\circ \bigr)^{i} = B^{\circ}_{i}(t,P_{t}^\circ,P_t) \ud t 
+\varepsilon \sum_{j \in \ESm}
\sqrt{(P^\circ_{t})^{i} ({P}_{t}^\circ)^{j}}
\ud \overline{W}^{\circ,i,j}_{t} 
\\
&\hspace{35pt} +
\varepsilon \varsigma(P_{t}^\circ) 
 \sum_{j \in \ESdm}
  \sqrt{(P_{t}^\circ)^i ({P}_{t})^{j}}
\ud \overline{W}^{\circ,i,m+j}_{t}, 
 \quad i \in \ESm,
\end{split}
\end{equation}
for $t \in [t_{0},T]$.
 {The unique solvability of 
\eqref{eq:system:tilde:samelaw}
is addressed in the next two steps.} 
\vskip 4pt

\textit{Second Step.}
Observing that 
${\mathfrak b}$ is Lipschitz continuous, we deduce that 
the coefficients 
of 
\eqref{eq:system:tilde:samelaw}
are Lipschitz continuous in the entries $(P^\circ,P)$ 
as long the coordinates of the 
latter remain bounded and away from zero and as long as 
the sum of the coordinates of $P^\circ$ remains away below $1$, we deduce that, for any small $a>0$,
the system \eqref{eq:system:tilde:samelaw} is uniquely solvable up to the first $\uptau^{a}$
when one of the coordinates of $P^\circ_{\uptau^{a}}$ or of $P_{\uptau^a}$ becomes lower than 
$a$ or when the sum of the coordinates of $P^\circ_{\uptau^a}$ becomes greater than $1-a$
. 
Letting $a$ tend to $0$, we deduce that 
\eqref{eq:system:tilde:samelaw} is uniquely solvable up to time $\uptau = \lim_{a \searrow 0} \uptau^{a} \wedge T$.

Hence, unique solvability follows if we can prove that
\begin{equation}
\label{eq:proba:to:be:1}
{\mathbb P}
\biggl( \inf_{i \in \ESm} \inf_{t_{0} \le t < \uptau} P^{\circ,i}_{t} >0, 
\quad \sup_{t_{0} \le t < \uptau} \sum_{i \in \ESm} P^{\circ,i}_{t} < 1, 
\quad 
\forall t \in [t_{0},\uptau), \ \sum_{i \in \ESdm} P_{t}^i =1
\biggr) = 1,
\end{equation}
since the latter implies that $\uptau=T$.

In order to 
check \eqref{eq:proba:to:be:1}, we first observe that, for $t \in [t_{0},\uptau)$, 
$\ud ( \sum_{i \in \ESdm} P_{t}^i)=0$. Hence, 
\begin{equation*}
\sum_{i \in \ESdm} P_{t}^i = 1, \quad t \in [t_{0},\uptau].
\end{equation*}
(Notice that the time interval is closed: Observing that the coefficients in 
\eqref{eq:system:tilde:samelaw} are bounded, we may indeed easily extend the solution in hand 
at time $\uptau$ itself.)
This prompts us to let
\begin{equation*}
\begin{split}
&\widetilde X_{t}^i = (P_{t}^\circ)^i, \quad i \in \ESm \ ; \quad 
\widetilde X_{t}^i = 
\varsigma^2 \bigl( P_{t}^\circ \bigr) P_{t}^{i-m}, 
\quad 
i = m+1,\cdots,d, 
\end{split}
\end{equation*}
for $t \in [t_{0},\uptau]$. 
Observe in particular that $\sum_{i \in \ES} 
\widetilde X_{t}^i =1$, for all $t \in [t_{0},\uptau]$.
If we prove that 
$(\widetilde X^1_{t},\cdots,\widetilde X_{t}^d)_{t_{0} \le t \le \uptau}$
satisfies the SDE \eqref{eq:P}
but for a new choice of the noise, then we are done: 
Not only we then deduce from  
Proposition 
\ref{prop:ito:formula}
(or, equivalently, 
Proposition 
\ref{thm:approximation:diffusion:2}) that 
\eqref{eq:proba:to:be:1} indeed holds true, but 
we also obtain the required identity in law, see 
\eqref{eq:identity:law}. 
\vskip 4pt

\textit{Third Step.} 
In order to prove that 
$(\widetilde X^1_{t},\cdots,\widetilde X_{t}^d)_{t_{0} \le t \le \uptau}$
satisfies \eqref{eq:P} (for a new choice of noise), we proceed as follows. 
First, we notice that, 
for $i \in \ESm$, 
\begin{equation*}
\ud \widetilde X_{t}^i  = 
{\mathfrak b}_i\bigl(t, \widetilde X_{t} \bigr) \ud t + \varepsilon
 \sum_{j \in \ES}
\sqrt{\widetilde{X}_{t}^{i}}\sqrt{\widetilde{X}_{t}^{j}} \ud \overline{W}^{\circ,i,j}_{t}, \quad t \in [t_{0},\uptau]. 
\end{equation*}
And, for $i=m+1,\cdots,d$, 
\begin{align}
\ud \widetilde X_{t}^i &=
 \Bigl(
B_{i-m}(t,P^\circ_{t},P_{t})
- P_{t}^{i-m} \sum_{j \in \ESdm}
B_{j} (t,P^\circ_{t},P_{t})
 \Bigr) \ud t 
+ \varepsilon \varsigma(P_{t}^\circ) \sum_{j \in \ESdm} \sqrt{P_{t}^{i-m} P_{t}^j} \ud \overline W_{t}^{i,j+m}
\nonumber
\\
&\hspace{5pt}
- P_{t}^{i-m}
\sum_{j \in \ESm}
B^{\circ}_{j}(t,P_{t}^\circ,P_t) \ud t 
- \sum_{j \in \ESdm}
\ud \bigl\langle
P_{t}^i, \bigl( P_{t}^\circ \bigr)^j\bigr\rangle 
\label{eq:widetildeXti}
\\
&\hspace{5pt}
- \varepsilon P_{t}^{i-m}
 \sum_{j,l \in \ESm}
\sqrt{(P^\circ_{t})^{j} ({P}_{t}^\circ)^{l}}
\ud \overline{W}^{\circ,j,l}_{t}  - \varepsilon P_{t}^{i-m}
\varsigma(P_{t}^\circ) 
 \sum_{j \in \ESm}
 \sum_{l \in \ESdm}
 \sqrt{(P_{t}^\circ)^j ({P}_{t})^{l}}
\ud \overline W^{\circ,j,m+l}_{t}.
\nonumber
\end{align}
Obviously, the bracket on the second line is zero since the underlying noises are independent. 
Hence, using the fact that $\sum_{i \in \ES} {\mathfrak b}_{i}(t,p)=0$ for any $(t,p) \in [0,T] \times \RR^{d}$, 
the drift reads 
\begin{equation*}
\begin{split}
&B_{i-m}(t,P^\circ_{t},P_{t})
- P_{t}^{i-m} \sum_{j \in \ESdm}
B_{j} (t,P^\circ_{t},P_{t})
- P_{t}^{i-m} \sum_{j \in \ESm}
B_{j}^\circ (t,P^\circ_{t},P_{t})
\\
&= {\mathfrak b}_{i}(t,\widetilde X_{t}) 
- P_{t}^{i-m} \sum_{j \in \ES}
{\mathfrak b}_{j} (t,P^\circ_{t},P_{t}) = {\mathfrak b}_{i}(t,\widetilde X_{t}).
\end{split}
\end{equation*}
Therefore, in order to prove that 
$(\widetilde X_{t})_{t_{0} \le t \le \uptau}$ 
satisfies \eqref{eq:P} (for a new choice of noise), it suffices to identify the martingale structure in 
\eqref{eq:widetildeXti}. To do so, we rewrite the three martingale increments in the above expansion for 
$i=m+1,\cdots,d$ in the form
\begin{equation*}
\begin{split}
&\varepsilon \varsigma(P_{t}^\circ) \sum_{j \in \ESdm} \sqrt{P_{t}^{i-m} P_{t}^j} \ud \overline W_{t}^{i,j+m}
\\
&\hspace{15pt}
- \varepsilon P_{t}^{i-m}
 \sum_{j,l \in \ESm}
\sqrt{(P^\circ_{t})^{j} ({P}_{t}^\circ)^{l}}
\ud \overline W^{\circ,j,l}_{t} 
- \varepsilon P_{t}^{i-m}
\varsigma(P_{t}^\circ) 
 \sum_{j \in \ESm}
 \sum_{l \in \ESdm}
 \sqrt{(P_{t}^\circ)^j ({P}_{t})^{l}}
\ud \overline W^{\circ,j,m+l}_{t}
\\
&=\varepsilon\varsigma^{-1}(P_{t}^\circ) \sum_{j=m+1}^{d} \sqrt{\widetilde X_{t}^{i} \widetilde X_{t}^j} \ud \overline W_{t}^{i,j}
- \varepsilon  \varsigma^{-2}(P_{t}^\circ) \widetilde X_{t}^i \sum_{j \in \ESm} \sum_{l \in \ES}
\sqrt{\widetilde X_{t}^{j} \widetilde X_{t}^l}\ud \overline W^{\circ,j,l}_{t}.
\end{split}
\end{equation*}
Hence, in order to complete the analysis, it remains to compute the various brackets 
$\ud \langle \widetilde X_{t}^i,\widetilde X_{t}^j \rangle $
for $i,j \in \ES$. 
Obviously, whenever $i,j \in \ESm$, 
\begin{equation*}
\begin{split}
\ud \langle \widetilde X_{t}^i,\widetilde X_{t}^j \rangle 
= \varepsilon^2 \Bigl( \widetilde X_{t}^i \delta_{i,j} - {\widetilde X_{t}^i \widetilde X_{t}^j }\Bigr) \ud t. 
\end{split}
\end{equation*}
Now, $i,j=m+1,\cdots,d$,
\begin{equation*}
\begin{split}
\ud \langle \widetilde X_{t}^i,\widetilde X_{t}^j \rangle_{t}
&= \varepsilon^2 \varsigma^{-2} (P_{t}^\circ)
\sqrt{\widetilde X_{t}^i \widetilde X_{t}^j}
\sum_{l,l'=m+1}^d \sqrt{\widetilde X_{t}^{l} \widetilde X_{t}^{l'}} \bigl( \delta_{i,j} \delta_{l,l'} 
- \delta_{i,l'} \delta_{j,l}
\bigr) \ud t
\\
&\hspace{15pt}
+  \varepsilon^2 \varsigma^{-4} (P_{t}^\circ) {\widetilde X_{t}^i \widetilde X_{t}^j}
\sum_{k,k' \in \ESm} 
\sum_{l,l' \in \ES}
\sqrt{\widetilde X_{t}^k \widetilde X_{t}^l
\widetilde X_{t}^{k'} \widetilde X_{t}^{l'}}
\bigl( \delta_{k,k'} \delta_{l,l'} - \delta_{k,l'} \delta_{k',l}
\bigr) \ud t
\\
&=  \varepsilon^2 \varsigma^{-2} (P_{t}^\circ)
\Bigl(\delta_{i,j}  {\widetilde X_{t}^i} \sum_{l=m+1}^d 
\widetilde X_{t}^l - {\widetilde X_{t}^i \widetilde X_{t}^j} \Bigr) \ud t
\\
&\hspace{15pt}
+  \varepsilon^2 \varsigma^{-4} (P_{t}^\circ)
\widetilde X_{t}^i \widetilde X_{t}^j
\Bigl(
\sum_{k \in \ESm} 
 \sum_{l \in \ES} 
 \widetilde X_{t}^k
 \widetilde X_{t}^l
 -  
\sum_{k \in \ESm} 
 \sum_{l \in \ESm} 
 \widetilde X_{t}^k
 \widetilde X_{t}^l 
 \Bigr) \ud t
 \\
 &=  \varepsilon^2 \varsigma^{-2} (P_{t}^\circ)
\Bigl(\delta_{i,j}  {\widetilde X_{t}^i} \sum_{l=m+1}^d 
\widetilde X_{t}^l - {\widetilde X_{t}^i \widetilde X_{t}^j} \Bigr) \ud t
+  \varepsilon^2 \varsigma^{-4} (P_{t}^\circ)
\widetilde X_{t}^i 
\widetilde X_{t}^j  
\sum_{k=1}^m 
 \sum_{l=m+1}^d 
 \widetilde X_{t}^k
 \widetilde X_{t}^l \ud t.
\end{split}
\end{equation*}
Now, the key point is to observe that 
$\sum_{l=m+1}^d \widetilde X_{t}^l = 1 - \sum_{l=1}^m \widetilde X_{t}^l =
\varsigma^2(P_{t}^\circ)$. Therefore,
 \begin{equation*}
\begin{split}
\ud \langle \widetilde X_{t}^i,\widetilde X_{t}^j \rangle_{t}
&= \Bigl( \varepsilon^2  \delta_{i,j}  {\widetilde X_{t}^i} 
 - \varepsilon^2 \varsigma^{-2}(P_{t}^\circ)
 {\widetilde X_{t}^i \widetilde X_{t}^j} 
+  \varepsilon^2 \varsigma^{-2} (P_{t}^\circ)
{\widetilde X_{t}^i}
{\widetilde X_{t}^j} 
\bigl( 1 - \varsigma^{2} (P_{t}^\circ) \bigr) \Bigr) \ud t
\\
&= \varepsilon^2  \Bigl( \delta_{i,j}  {\widetilde X_{t}^i} -
{\widetilde X_{t}^i}
{\widetilde X_{t}^j}  \Bigr) \ud t.
\end{split}
\end{equation*}
Now, for $i=1,\cdots,m$ and for $j=m+1,\cdots,d$, 
\begin{equation*}
\begin{split}
\ud \langle \widetilde X_{t}^i,\widetilde X_{t}^j \rangle_{t}
&=
- \varepsilon^2
\varsigma^{-2}(P_{t}^\circ)
\widetilde X_{t}^j
 \sum_{l,l'=1}^{d}
 \sum_{k=1}^m 
\sqrt{\widetilde{X}_{t}^{i} \widetilde{X}_{t}^{l}} 
\sqrt{\widetilde X_{t}^{k} \widetilde X_{t}^{l'}}
\bigl( \delta_{i,k} \delta_{l,l'} - \delta_{i,l'} \delta_{k,l} \bigr) \ud t
\\
&= 
\Bigl( - \varepsilon^2
\varsigma^{-2}(P_{t}^\circ)
\widetilde X_{t}^i 
\widetilde X_{t}^j
+  \varepsilon^2
\varsigma^{-2}(P_{t}^\circ)
\widetilde X_{t}^i 
\widetilde X_{t}^j
 \sum_{k=1}^m 
\widetilde{X}_{t}^{k}  \Bigr) \ud t
= - \varepsilon^2
\widetilde X_{t}^i 
\widetilde X_{t}^j \ud t,
\end{split}
\end{equation*}
which completes the proof.
\end{proof}

\noindent {{\bf Acknowledgment.} We are thankful to the anonymous referee for his or her suggestions, which helped us improve the paper.}

\bibliographystyle{abbrv}
\bibliography{references}

\end{document}